\documentclass[a4paper,nosumlimits,twoside,dvipsnames,reqno]{amsart}% it also holds amsfonts
\usepackage{amssymb,amsmath,amsthm}
\usepackage{a4wide}
\usepackage{xcolor,graphicx}
\usepackage{float}
\usepackage{srcltx}
\usepackage[all]{xy}
\SelectTips{lu}{12}
\usepackage{cancel}
\usepackage{version}
\usepackage[T1]{fontenc}
\usepackage{xcolor}
\usepackage{enumerate}
\usepackage[inline]{enumitem}
\usepackage[normalem]{ulem}
\usepackage{latexsym}
\usepackage[2emode]{psfrag}
\usepackage{yhmath}
\usepackage{array}
\usepackage{dsfont}
\usepackage{nicefrac}
\usepackage[bookmarks=false,colorlinks,urlcolor=cyan,citecolor=Plum,linkcolor=Blue,hypertexnames=false]{hyperref}% [backref=page] 
\usepackage[capitalise,noabbrev]{cleveref} 	
\usepackage[final]{showlabels}%serve per mostrare le labels (a fine lavoro scriver [final] prima [inner])

\usepackage{mathtools,thmtools}

\usepackage{etoolbox}
\usepackage{orcidlink}

\makeatletter
\patchcmd{\@setaddresses}{\indent}{\noindent}{}{}
\patchcmd{\@setaddresses}{\indent}{\noindent}{}{}
\patchcmd{\@setaddresses}{\indent}{\noindent}{}{}
\patchcmd{\@setaddresses}{\indent}{\noindent}{}{}
\makeatother

\let\savering\ring
\let\savesquare\square
\let\ring\relax
\let\square\relax
\usepackage{mathabx}
\let\ring\savering
\let\square\savesquare

\makeatletter
\def\namedlabel#1#2{\begingroup
    #2%
    \def\@currentlabel{#2}%
    \phantomsection\label{#1}\endgroup}
\makeatother

\newtheorem{theorem}{\sc Theorem}[section]
\newtheorem{proposition}[theorem]{\sc Proposition}
\newtheorem{notation}[theorem]{\sc Notation}

\newtheorem{lemma}[theorem]{\sc Lemma}
\newtheorem{corollary}[theorem]{\sc Corollary}
\newtheorem*{proposition*}{Proposition}
\newtheorem*{theorem*}{\sc Theorem}

\theoremstyle{definition}
\newtheorem{definition}[theorem]{\sc Definition}

\declaretheorem[name={\sc Example},qed={$\lozenge$},sibling=theorem]{example}

\theoremstyle{remark}
\newtheorem{remark}[theorem]{\sc Remark}

\allowdisplaybreaks
\newenvironment{invisible}{{\noindent\sc \colorbox{yellow}{Invisible:}\;}\color{gray}}{\medskip}
\excludeversion{invisible}
\excludeversion{proof?}

\newcommand{\Cc}{\mathcal{C}}

\newcommand{\id}{\mathrm{Id}}
\newcommand{\im}{\mathrm{im}}

\newcommand{\op}{\mathrm{op}}
\newcommand{\cop}{\mathrm{cop}}

\newcommand{\can}{\mathfrak{can}}
\newcommand{\N}{\mathbb N}
\newcommand{\Q}{\mathbb Q}
\newcommand{\Z}{\mathbb Z}

\newcommand{\pred}[1]{{^\star #1}}

\def\Alg{{\sf Alg}}
\def\Coalg{{\sf Coalg}}
\def\Bialg{{\sf Bialg}}
\def\Hopf{{\sf Hopf}}
\def\Grp{{\sf Grp}}
\def\Mon{{\sf Mon}}

\newcommand{\qB}{Q(B)}
\newcommand{\qq}{q_B}

\crefname{part}{\S}{\S\S}
\crefname{chapter}{\S}{\S\S}
\crefname{section}{\S}{\S\S}
\crefname{subsection}{\S}{\S\S}

\begin{document}
\allowdisplaybreaks

\title[On the Hopf envelope of finite-dimensional bialgebras]{On the Hopf envelope of finite-dimensional bialgebras}

\thanks{This paper was written while the authors were members of the ``National Group for Algebraic and Geometric Structures and their Applications'' (GNSAGA-INdAM).
It is partially based upon work from COST Action CaLISTA CA21109 supported by COST (European Cooperation in Science and Technology) \href{www.cost.eu}{www.cost.eu}.
The authors acknowledge partial support by the European Union - NextGenerationEU under NRRP, Mission 4 Component 2 CUP D53D23005960006 - Call PRIN 2022 No.\, 104 of February 2, 2022 of Italian Ministry of University and Research; Project 2022S97PMY \textit{Structures for Quivers, Algebras and Representations (SQUARE)}. PS was partially supported by the project PID2024-157173NB-I00 funded by MCIN/AEI/10.13039/501100011033 and by FEDER, UE. 
Esta publicación ha sido parcialmente financiada con fondos propios de la Junta de Andalucía, en el marco de la ayuda DGP\_EMEC\_2023\_00216.
We wish to express our gratitude to the referee for the valuable comments and suggestions, which have contributed to improving the previous version of the paper. We also thank A. Chirv\u{a}situ, J.~Cuadra D\'iaz, D.~Ferri, R.~Fioresi, M.~Gran, D.~Nikshych and P.~Schauenburg for insightful discussions concerning the themes of this work. }

\begin{abstract}
The Hopf envelope of a bialgebra is the free Hopf algebra generated by the given bialgebra. Its existence, as well as that of the cofree Hopf algebra, is a well-known fact in Hopf algebra theory, but their construction is not particularly handy or friendly. In this note, we offer a novel realisation of the Hopf envelope and of the cofree Hopf algebra of a finite-dimensional bialgebra as a particular quotient and sub-bialgebra, respectively, of the bialgebra itself. Our construction can also be extended to the infinite-dimensional case, provided that the bialgebra satisfies additional conditions, such as being right perfect as an algebra or admitting a $n$-antipode, the latter being a notion hereby introduced. Remarkably, the machinery we develop also allows us to give a new description of the Hopf envelope of a commutative bialgebra and of the cofree cocommutative Hopf algebra of a cocommutative bialgebra.
\end{abstract}


\author[A.~Ardizzoni]{Alessandro Ardizzoni\, \orcidlink{0000-0001-7384-611X}}
\address{%
\parbox[b]{\linewidth}{University of Turin, Department of Mathematics ``G. Peano'', via
Carlo Alberto 10, I-10123 Torino, Italy}}
\email{alessandro.ardizzoni@unito.it}
\urladdr{\url{www.sites.google.com/site/aleardizzonihome}}

\author[C.~Menini]{Claudia Menini\, \orcidlink{0000-0003-2782-7377}}
\address{%
\parbox[b]{\linewidth}{University of Ferrara, Department of Mathematics and Computer Science, Via Machiavelli
30, Ferrara, I-44121, Italy}}
\email{men@unife.it}
\urladdr{\url{https://sites.google.com/a/unife.it/claudia-menini}}

\author[P.~Saracco]{Paolo Saracco\, \orcidlink{0000-0001-5693-7722}}
\address{University of Turin, Department of Mathematics ``G. Peano'', via Carlo Alberto 10, I-10123 Torino, Italy}
\email{p.saracco@unito.it}
\urladdr{\url{https://sites.google.com/view/paolo-saracco}}

 % Primary  ?? ; Secondary ??
% 16T05 - Hopf algebras and their applications
% 16T10 - bialgebras
% 18M05 - monoidal categories, symmetric monoidal categories

\date{\today}

\maketitle
\tableofcontents

\section{Introduction}

In \cite[\S7]{Manin-book}, Manin introduced the \emph{Hopf envelope} of a bialgebra $B$, which can also be legitimately referred to as the \emph{free Hopf algebra generated by} $B$, by generalising the construction of the free Hopf algebra over a coalgebra given earlier by Takeuchi \cite{Takeuchi}. The Hopf envelope of $B$ is a Hopf algebra $\mathrm{H}\left( B\right) $ together with a bialgebra map $\eta _{B}:B\rightarrow \mathrm{H} \left( B\right) $ such that for any Hopf algebra $H$ and for every bialgebra map $f:B\rightarrow H$, there exists a unique Hopf algebra map  $g:\mathrm{H}\left( B\right) \rightarrow H$ such that $g\circ \eta _{B}=f$ (see also \cite[Theorem 2.6.3]{Pareigis}, and see \cite{PaulJoost} for a recent extension of this construction to Hopf categories). However, the general construction of the Hopf envelope of a bialgebra may be particularly unfriendly, as it involves the coproduct of infinitely many copies of $B$, possibly with opposite multiplication and comultiplication. Consequently, it is always desirable to have at disposal more handy realisations of it, at least in some favourable cases, as it has been proven useful recently in \cite{Farinati,FarinatiGaston} and in the forthcoming \cite{AMSaracco}.

Our first aim is to suggest an alternative approach which, among other things, allows us to explicitly describe the Hopf envelope of a finite-dimensional and, more generally, of a right perfect bialgebra $B$ as an explicit quotient of $B$ itself. More precisely we have the following:

\begin{theorem*}[\ref{prop:HBfd}, \ref{prop:HBArt}, \ref{prop:HBfdArt}]
Let $B$ be a finite-dimensional or, more generally,  right perfect bialgebra. Then $\mathrm{H}(B) = \frac{ B}{\ker(i_B)} \cong \mathrm{im}(i_B)$ where $i_B\colon B \to B\oslash B\coloneqq (B \otimes B)/(B \otimes B) B^+$, $b \mapsto \overline{b \otimes 1}$.
\end{theorem*}

Interestingly, in all the cases considered herein we can show that $\mathrm{H}(B) \cong B \oslash B$.

Secondly, by dualizing our techniques, our approach allows us to handily describe also the dual construction, that is, the \emph{cofree Hopf algebra on a bialgebra} $B$, at least in the finite-dimensional setting. The cofree Hopf algebra of $B$ is a Hopf algebra $\mathrm{C}(B)$ with a bialgebra morphism $\epsilon_B \colon \mathrm{C}(B) \to B$ such that for every Hopf algebra $H$ and any bialgebra morphism $f \colon H \to B$, there exists a unique Hopf algebra morphism $\tilde f \colon H \to \mathrm{C}(B)$ satisfying $\epsilon_B \circ \tilde f = f$. A non-constructive proof of its existence, based on special adjoint functor theorem, was given in \cite{Agore}. Here, we prove the following:

\begin{theorem*}[\ref{prop:CBfd}]
Let $B$ be a finite-dimensional bialgebra. \\
Then $\mathrm{C}(B)= \im(p_B)$ 
%$\mathrm{C}(B)= \{b \in B \mid \Delta(b) \in \im(p_B) \otimes B\}$ 
where $p_B\colon B\boxslash B\coloneqq \left(B \otimes B\right)^{\mathrm{co}B} \to B$, $x^i \otimes y_i \mapsto x^i\varepsilon(y_i)$.
\end{theorem*}

The proofs of these results rely on a thoughtful study of the canonical maps $i_B$ and $p_B$, which we introduce in \cref{ssec:iB}. A cardinal role in this is played by the surjectivity of $i_B$ and the injectivity of $p_B$ (studied in \cref{ssec:i_B} and \cref{ssec:p_B}), by the related notion of $n$-Hopf algebra and its one-sided analogues (introduced and studied in \cref{ssec:nHopf}), and by their relationship with the property of being right perfect, as an algebra, for the given bialgebra (analysed in \cref{ssec:Artin}). In particular, they use in a crucial way the facts that a finite-dimensional bialgebra is always a $n$-Hopf algebra for some $n \in \mathbb{N}$ (\cref{coro:fdS}), so that it has both $i_B$ surjective and $p_B$ injective (\cref{prop:semiantip}), and that, more generally, also right perfect bialgebras have both $i$ surjective and $p$ injective (\cref{pro:Artinian}).

Even if, from a categorical perspective, the free and cofree constructions, as well as the canonical maps $i_B$ and $p_B$ and 
the objects $B \oslash B$ and $B \boxslash B$ 
%the functors \rd{[questi funtori non compaiono più nei teoremi citati sopra. Forse possiamo sostituirli con $B\oslash B$ e $B\boxslash B$.]}\ps{[ok]} $Q(B) \coloneqq B/\ker(i_B)B$ and $K(B) \coloneqq \{b \in B \mid \Delta(b) \in \im(p_B) \otimes B\}$, 
are dual to each other, their treatment cannot be reduced to a straightforward dualization, as not every result we prove admits an elementary dualisation and not every tool we use, such as the fundamental theorem of coalgebras, has a dual counterpart. 
%\rd{[Potremmo dire che in vari (tutti i?) casi considerati, incluso quello in cui $B$ è commutativa, riusciamo a dimostrare che $\mathrm{H}(B)\cong B\oslash B$].}\ps{[Ho provato ad aggiungere una frase sopra. Il caso commutativo forse lo lascerei implicito qui: senza menzionare Q e K, non è molto chiaro perché proprio il caso commutativo sia sorprendente.]}
For example, 
for a commutative bialgebra $B$, the free Hopf algebra $\mathrm{H}(B)$, which we can prove to coincide with the free commutative Hopf algebra $\mathrm{HP}_c(B)$ up to isomorphism (\cref{prop:HPcB}), can be realised as the quotient $B\oslash B$ (\cref{thm:intcomm}). The dual statement does not hold in cocommutative setting, as shown by \cref{thm:CcB} and the paragraph after it.
%\ps{Similarly,} for a monoid bialgebra $\Bbbk M$ we always have $\mathrm{C}(\Bbbk M) = K\big(K(\Bbbk M)\big)$ (\cref{thm:monoidbialgebra}), that is, iterating twice the $K(\cdot)$ construction leads to a Hopf algebra. On the other hand, iterating the $Q(\cdot)$ construction on the monoid bialgebra $\Bbbk \mathbb{N}$ always returns $\Bbbk \mathbb{N}$.  \rd{[Forse qui sopra toglierei il riferimanto a Q e K.]}
Therefore, in the present work we treat both cases in parallel.

The careful reader may wonder why we treat separately the finite-dimensional and the right perfect case. Our motivation is twofold. First of all, we use different techniques in the finite-dimensional setting than in the perfect one. Namely, properties of $n$-Hopf algebras, which are interesting in and of themselves, too. Secondly, the results we achieve for finite-dimensional bialgebras are, as one may expect, stronger than those that we have for right perfect ones, also thanks to the role played by $n$-Hopf algebras. For instance, even if we are able to explicitly realise the cofree Hopf algebra of a right perfect bialgebra $B$ as a sub-bialgebra of $B$ (\cref{prop:CBleftArt}, via an iterative construction based on Kelly's well-(co)pointed endofunctors \cite{Kelly1}), its description is not as handy and as elegant as for the finite-dimensional ones. For the sake of completeness, we address also the construction of the free and cofree Hopf algebras on a $n$-Hopf algebra via the same iterative construction.

We conclude with a number of concrete examples in \cref{ssec:example}; among all, that of monoid algebras. % \ps{and some indications of how to extend our approach to the infinite-dimensional setting. [?]}

%\ps{AGGIUNGERE COMMENTO MANCANZA SIMMETRIA $Q(B)$ / $K(B)$ E PERCHÉ TENIAMO SEPARARATI DIM FINITA E ARTINIANO}

% \ps{Even if many results which are true for $n$-Hopf algebras are true for Artinian bialgebras as well, and conversely, there is no direct relationship between the two notions. For instance, any infinite-dimensional Hopf algebra is an example of an $0$-Hopf algebra which is not Artinian.
% %\cref{exa:Radual} exhibits an example of a $1$-Hopf algebra which is not Artinian (take as $A$ any non-Artinian algebra therein, such as $A=\Bbbk[X]$), and
% On the other hand, (EXAMPLE OF A LEFT ARTINIAN WHICH IS NOT $n$-HOPF?)
% %the monoid bialgebra $\mathbb{C}[\mathbb{Z}_4]$ over the multiplicative monoid $(\mathbb{Z}_4,\cdot,\bar{1})$ is Artinian (because finite-dimensional), but it cannot be a $n$-Hopf algebra for any $n$ since $\bar 2^2 = \bar 0$.
% }

\section{Preliminaries}

Let $\Bbbk$ be a field and let $\mathfrak{M}$ be the category of vector spaces over $\Bbbk$. If $B$ is a $\Bbbk$-bialgebra, then we can consider the category $\mathfrak{M}_{B}^{B}$ of right Hopf modules over $B$ and the free Hopf module functor $F = \left( -\right) \otimes B_{\bullet }^{\bullet } \colon \mathfrak{M} \to \mathfrak{M}_{B}^{B}$, which is fully faithful and admits a right adjoint given by the space of coinvariant elements functor $G = \left(-\right) ^{\mathrm{co}B} \colon \mathfrak{M}_{B}^{B} \to \mathfrak{M}$.
These functors are involved in the celebrated structure theorem for Hopf modules and they are known to be part of an adjoint triple $E\dashv F\dashv G$ (see e.g.\ \cite[Section 3]{Sar21}), where $E=\overline{\left( -\right) }^{B}:\mathfrak{M}_{B}^{B}\rightarrow \mathfrak{M}$ maps a $B$-Hopf module $M$ to  $\overline{M}^{B} \coloneqq M/MB^{+}$ ($B^{+}=\ker (\varepsilon _{B})$ is the augmentation ideal). The units and counits of these adjunctions are given by
\begin{gather*}
\eta _{M} \colon M\rightarrow \overline{M}^{B}\otimes B, \quad m\mapsto \sum \overline{%
m_{0}}\otimes m_{1},\qquad \epsilon _{V} \colon \overline{\left( V\otimes B\right) }%
^{B}\overset{\cong }{\rightarrow }V,\quad \overline{v\otimes b}\mapsto
v\varepsilon _{B}\left( b\right)  \\
\nu _{V} \colon V\overset{\cong }{\rightarrow }\left( V\otimes B\right) ^{\mathrm{co} %
B},\quad v\mapsto v\otimes 1_{B},\qquad \theta _{M} \colon M^{\mathrm{co}B}\otimes
B\rightarrow M,\quad m\otimes b\mapsto mb.
\end{gather*}%
As observed in \cite[Section 3]{Sar21}, the two adjoints of $\left( -\right) \otimes B_{\bullet }^{\bullet }$ are
connected by a natural transformation $\sigma :(-)^{\mathrm{co}B}\rightarrow
\overline{(-)}^{B}$ defined on components by \[\sigma _{M}:M^{\mathrm{co}%
B}\rightarrow \overline{M}^{B},m\mapsto \overline{m} \coloneqq m+MB^{+}.\]

For the sake of clarity and simplicity, let us introduce the following notation.

\begin{notation}\label{not:oslash}
Let $B$ be a bialgebra. For every $X$ in $\mathfrak{M}_B$, the category of right $B$-modules, we denote by $X\oslash B$ the quotient
\[X\oslash B = \frac{X_\bullet\otimes B_\bullet}{(X_\bullet\otimes B_\bullet)B^+},\]
where $B^+ = \ker(\varepsilon_B)$.  We also denote by $x\oslash a$ the equivalence class of $x\otimes a$, for all $x\in X,a\in B$, so that $xb_1\oslash ab_2=x\oslash a\varepsilon(b)$, for every $b\in B$.

Similarly, for $N \in \mathfrak{M}^{B}$, the category of right $B$-comodules, we denote by $N \boxslash B$ the subspace
\[
N \boxslash B = \left(N^\bullet \otimes B_\bullet^\bullet\right)^{\mathrm{co}B}
\]
of $N \otimes B$ and by $n^i \boxslash b_i$ the element $n^i \otimes b_i$ in $N \boxslash B$, so that
\begin{equation}\label{eq:coinv}
{n^i}_{0} \boxslash {b_i}_{1} \otimes {n^i}_{1} {b_i}_{2} = n^i \boxslash b_i \otimes 1_B.
\end{equation}
\end{notation}

In this way,
\[\overline{X_{\bullet }\otimes B^{\bullet}_{\bullet }}^{B}=\frac{X_{\bullet }\otimes B_{\bullet }}{\left( X_{\bullet }\otimes B_{\bullet }\right) B^{+}}= X\oslash B
\quad\text{and}\quad
\left(N^\bullet \otimes B^\bullet\right) \square_B \Bbbk \cong \left(N^\bullet \otimes B_\bullet^\bullet\right)^{\mathrm{co}B} = N \boxslash B\]
for every $X$ in $\mathfrak{M}_B$, $N$ in $\mathfrak{M}^B$.
We are mainly interested in the quotient $B \oslash B$ and in the subspace $B \boxslash B$.

\begin{comment}
\rd{[Non so se è rilevante ma $B \boxslash B$ è $L(B,B)$ nelle notazioni di DEFINITION 3.2 in \href{https://link.springer.com/article/10.1023/A:1008608028634}{Schauenburg}, cioè lo Ehresmann-Schauenburg bialgebroid.]} \ps{[Grazie! Sì, lo è nel caso della Hopf-Galois triviale $\Bbbk \subseteq B$ su $B$. Vorrei tanto che fosse rilevante, perché lo Ehresmann-Schauenburg bialgebroid è qualcosa su cui sto lavorando con Eliezer Batista, un coautore brasiliano, ma, a parte accorgercene, non sono ancora riuscito a trovare il tempo di pensarci.]}
\end{comment}

% \begin{remark}
% \label{rmk:BotBact}
% Given a bialgebra $B$, we have that $B\oslash B$ becomes a left $B\otimes B$-module by setting  $(a\otimes b)\cdot(x\oslash y)=ax\oslash by$, for every $a,b,x,y\in B$.
% \end{remark}

\begin{proposition}\label{lem:oslash}
Let $B$ be a bialgebra. Then $B\oslash B$ is a coalgebra in $_{B\otimes B^\cop}\mathfrak{M}$ with comultiplication, counit and left $B\otimes B$-module structure given, for every $x,y,a,b\in B$, by
\begin{equation}\label{def:oslashcoalg}
   \Delta(x\oslash y) = (x_1\oslash y_2) \otimes (x_2\oslash y_1),
   \quad
   \varepsilon(x\oslash y)=\varepsilon(x)\varepsilon(y)
   \quad \text{and}\quad
   (a\otimes b)\cdot(x\oslash y)=ax\oslash by.
   \end{equation}
% \rd{[Questa struttura di modulo su $B\oslash B$ non mi sembra quella proveniente dall'iso col preduale. Forse basta dire che è una coalgebra in $_{B^\cop}\mathfrak{M}$ con la vecchia struttura di modulo. Oppure, togliendo il remark precedente, che è una coalgebra in $_{B\otimes B^\cop}\mathfrak{M}$ dove $(a\otimes b)(x\oslash y)=ax\oslash by.$ Poi nell'altro articolo prendiamo quella indotta da $b$.]}
In particular, $B\oslash B$ is a quotient coalgebra of $B\otimes B^\cop.$ Analogously, $B \boxslash B$ is an algebra in $\prescript{B \otimes B^{\mathrm{op}}}{}{\mathfrak{M}}$ with multiplication, unit and $B \otimes B$-coaction given, for all $u^i \boxslash v_i,x^j \boxslash y_j \in B \boxslash B$, by
\begin{equation}\label{eq:boxslashalg}
\begin{gathered}
\left(u^i \boxslash v_i\right)\cdot \left(x^j \boxslash y_j\right) = u^ix^j \boxslash y_jv_i, \quad 1_{B \boxslash B} = 1_B \boxslash 1_B \quad \text{and} \\
\left(x^j \boxslash y_j\right)_{-1} \otimes \left(x^j \boxslash y_j\right)_0 = \left({x^j_{}}_1 \otimes {y_j}_1\right) \otimes \left({x^j_{}}_{2} \boxslash {y_j}_{2}\right).
\end{gathered}
\end{equation}
%\begin{invisible}
In particular, $B\boxslash B$ is a subalgebra of $B\otimes B^\op$ and $x^iy_i = \varepsilon\left(x^i\right)\varepsilon\left(y_i\right)1_B$ for all $x^i \boxslash y_i \in B \boxslash B$.
%\end{invisible}%
%\rd{[Alla luce di [Schauenburg, "Bialgebras Over Noncommutative Rings and a Structure Theorem for Hopf Bimodules"] verrebbe da chiedersi (ma non in questo articolo) se, anche senza antipodo, non sia vero che $_B\mathfrak{M}_B^B\cong{}_{B\boxslash B}\mathfrak{M}$ visto che $B\boxslash B=L(B,B)$ è lo Ehresmann-Schauenburg bialgebroid. Comunque forse citerei $L(B,B)$.]} \ps{[Effettivamente è una bella domanda, anche se ho qualche problema ad immaginarmi un'azione "naturale" di $B$ su un oggetto della forma $N \otimes B$ dove $N$ è un $B \boxslash B$-modulo, non avendo un antipodo a disposizione. Sicuramente, abbiamo il funtore $_B\mathfrak{M}_B^B \to {}_{B\boxslash B}\mathfrak{M}$ dato dai coninvarianti. Sarebbe curioso vedere anche come si configura in relazione alla adjoint triple $- \otimes_B \Bbbk  \dashv - \otimes B \dashv {}_B\hom_B^B(B \otimes B)$ tra $_B\mathfrak{M}_B^B$ e ${}_{B}\mathfrak{M}$.]}
\end{proposition}

\begin{proof}
    Note that $\Delta$ is well-defined as
    \begin{align*}
    \Delta(xb_1\oslash yb_2) & = (x_1b_1\oslash y_2b_4) \otimes (x_2b_2\oslash y_1b_3) = (x_1b_1\oslash y_2b_3) \otimes (x_2\oslash y_1)\varepsilon(b_2) \\
    & = (x_1b_1\oslash y_2b_2) \otimes (x_2\oslash y_1) = (x_1\oslash y_2)\varepsilon(b) \otimes (x_2\oslash y_1) = \Delta(x\oslash y)\varepsilon(b)
    \end{align*}
    and $\varepsilon$, too, since
    \[\varepsilon(xb_1\oslash yb_2)=\varepsilon(xb_1)\varepsilon(yb_2)=\varepsilon(x)\varepsilon(y)\varepsilon(b)
    =\varepsilon(x\oslash y)\varepsilon(b).\]
  Now, every bialgebra $A$ is a coalgebra in the monoidal category $_{A}\mathfrak{M}$. In particular, this holds for $A = B \otimes B^\cop$. Since we just showed that $B \oslash B$ is a quotient of $B \otimes B^\cop$ by a coideal left ideal, it is a quotient coalgebra in the category $_{B \otimes B^\cop}\mathfrak{M}$.
    % Consider now the left $B\otimes B$-module structure of $B\oslash B$ defined by $(a\otimes b)\cdot(x\oslash y) \coloneqq ax\oslash by$. Since
    % \begin{align*}
    % \Delta((a\otimes b)\cdot(x\oslash y)) & = \Delta(ax\oslash by) = (a_1x_1 \oslash b_2y_2) \otimes (a_2x_2 \oslash b_1y_1)
    % \\
    % & = (a_1 \otimes b_2)\cdot (x_1\oslash y_2) \otimes (a_2 \otimes b_1) \cdot (x_2\oslash y_1)
    % \end{align*}
    % and
    % \[\varepsilon((a\otimes b)\cdot(x\oslash y)) = \varepsilon(ax\oslash by) = \varepsilon(ax)\varepsilon(by) = \varepsilon(a)\varepsilon(x)\varepsilon(b)\varepsilon(y)
    % = \varepsilon(a \otimes b) \varepsilon(x\oslash y),\]
    % we deduce that $B\oslash B$ is a coalgebra in $_{B \otimes B^\cop}\mathfrak{M}$.
    Dually, one proves that $B \boxslash B$ is an algebra in $\prescript{B\otimes B^{\mathrm{op}}}{}{\mathfrak{M}}$.
    \begin{invisible}
    %\rd{[Lascerei il verdino  in modo da decidere con Claudia e poi eventualmente mettere in invisible.]}
   \rem{Similarly, observe that if $u^i \otimes v_i$ and $x^j \otimes y_j$ are in $B \boxslash B$, then
   \begin{align*}
   \left(u^ix^j\right)_1 \otimes \left(y_jv_i\right)_1 \otimes \left(u^ix^j\right)_2 \left(y_jv_i\right)_2 & = {u^i_{}}_1{x^j_{}}_1 \otimes {y_j}_1{v_i}_1 \otimes {u^i_{}}_2{x^j_{}}_2{y_j}_2{v_i}_2 = {u^i_{}}_1x^j \otimes y_j{v_i}_1 \otimes {u^i_{}}_2{v_i}_2 \\
   & = u^ix^j \otimes y_jv_i \otimes 1,
   \end{align*}
   whence multiplication and unit are well-defined. It is also clear that $\Bbbk \to B \boxslash B, 1_\Bbbk \mapsto 1_B \otimes 1_B,$ is left $B \otimes B$-colinear. Finally, from
   \begin{align*}
   \left(u^ix^j \otimes y_jv_i\right)_{-1} \otimes \left(u^ix^j \otimes y_jv_i\right)_0 & = \left({u^i_{}}_1{x^j_{}}_1 \otimes {y_j}_1{v_i}_1\right) \otimes \left({u^i_{}}_{2}{x^j_{}}_{2} \otimes {y_j}_{2}{v_i}_{2}\right) \\
   & = \left({u^i_{}}_1{x^j_{}}_1 \otimes {y_j}_1{v_i}_1\right) \otimes \left({u^i_{}}_{2} \otimes {v_i}_{2}\right)\cdot \left({x^j_{}}_{2} \otimes {y_j}_{2}\right) \\
   & = \left(u^i \otimes v_i\right)_{-1}\left(x^j \otimes y_j\right)_{-1} \otimes \left(u^i \otimes v_i\right)_0 \cdot \left(x^j \otimes y_j\right)_0,
   \end{align*}
   we deduce that $B \boxslash B$ is an algebra in $\prescript{B\otimes B^{\mathrm{op}}}{}{\mathfrak{M}}$.}
   \end{invisible}
   %
   %\begin{invisible}
   Finally, recall from \eqref{eq:coinv} that for every $x^i \boxslash {y_i} \in B\boxslash B=\left(B^\bullet \otimes B^\bullet\right)^{\mathrm{co}B}$ we have that
    \begin{equation}\label{eq:coinvs}
    x^i_1 \otimes {y_i}_1 \otimes x^i_2{y_i}_2 = x^i \otimes {y_i} \otimes 1_B
    \end{equation}
    so that $x^iy_i = \varepsilon\left(x^i\right)\varepsilon\left(y_i\right)1_B$ by applying $\varepsilon \otimes \varepsilon \otimes \id_B$ to both sides.
    %\end{invisible}
\end{proof}

\begin{remark}%[\ps{added 14 07 2025}]
    It is noteworthy the fact that when $B$ is a Hopf algebra, $B \boxslash B = L(B,B) \cong B$ is a particular instance of the so-called \emph{Ehresmann-Schauenburg bialgebroid}, see \cite{Schauneburg-BialgOverNoncommRings}.
\end{remark}
%\rd{[Serve Hopf? In questo caso è B stesso...]} \ps{[No, per definire l'oggetto no, ma nell'articolo di Schauenburg, $H$ è  sempre Hopf perch\'e lui parte sempre da una Hopf-Galois.]}

\subsection{The maps \texorpdfstring{$i_B$}{i\_B} and \texorpdfstring{$p_B$}{p\_B}}
\label{ssec:iB}
For every bialgebra $B$, one can consider the (somehow, canonical - see later) maps $i_B \colon B\to B\oslash B,x\mapsto x\oslash 1$ and $p_B \colon B \boxslash B \to B, x^i \boxslash y_i \mapsto x^i\varepsilon(y_i)$.
We discuss here some of their properties because they play a central role in the construction of the free and cofree Hopf algebras generated by a bialgebra $B$.

In order to properly motivate the introduction of $i_B$ and $p_B$, recall that a \emph{right Hopf algebra} is a bialgebra with a \emph{right antipode}  $S^{r}$, i.e.\ a right convolution inverse of $\mathrm{Id}_{B}$. Similarly, a \emph{left Hopf algebra} is a bialgebra with a \emph{left antipode} $S^{l}$. Recall also that a one-sided antipode always preserves unit and counit, as observed in \cite[Remark 3.8]{Sar21}. Hence, an anti-multiplicative and anti-comultiplicative one-sided antipode is automatically an anti-bialgebra map.

\begin{proposition}
\label{prop:Frobenius}Let $B$ be a bialgebra. The following are equivalent.

\begin{enumerate}[label=(\alph*),ref= {\itshape(\alph*)},itemsep=0.1cm,leftmargin=*]
\item\label{item:1sided1} The coinvariant functor $\left( -\right) ^{\mathrm{co}B}:\mathfrak{M}%
_{B}^{B}\rightarrow \mathfrak{M}$ is Frobenius.

\item\label{item:1sided2} The canonical natural transformation $\sigma :\left( -\right) ^{%
\mathrm{co}B}\rightarrow \overline{\left( -\right) }^{B}$ is invertible.

\item\label{item:1sided3} The map $i_{B}:B\rightarrow B\oslash B,b\mapsto b\oslash 1_{B},$ is invertible.

\item\label{item:1sided5} The map $p_B \colon B \boxslash B \to B, x^i \boxslash y_i \mapsto x^i\varepsilon\left(y_i\right),$ is invertible.

\item\label{item:1sided4} $B$ is right Hopf algebra with anti-multiplicative and
anti-comultiplicative right antipode $S^{r}$.
\end{enumerate}
Moreover, if \ref{item:1sided4} holds true, then $i_{B}^{-1}\left( x\oslash y\right)
=xS^{r}\left( y\right) $ and $p_B^{-1}(x) = x_1 \boxslash S^r(x_2)$ for every $x,y\in B.$
\end{proposition}

\begin{proof}
We will use some results from \cite[Section 3]{Sar21}. First note that using the notations in
loc.~cit., we have $\widehat{B}:=B\widehat{\otimes }B=B_{\bullet
}\otimes B_{\bullet }^{\bullet }$ so that $\overline{B\widehat{\otimes }B}%
^{B}=B\oslash B$ in our notations.

Thus we have $\sigma _{\widehat{B}}:\left( B\widehat{\otimes }B\right) ^{%
\mathrm{co}B}\rightarrow B\oslash B,x\otimes y\mapsto x\oslash y$. By \cite[%
Remark 3.2]{Sar21}, every element in $\left( B\widehat{\otimes }B\right) ^{%
\mathrm{co}B}$ is of the form $x\otimes 1.$ Hence $\omega _{B}:B\rightarrow
\left( B\widehat{\otimes }B\right) ^{\mathrm{co}B},x\mapsto x\otimes 1$, is
bijective with inverse $B\otimes \varepsilon $ (see also \cite[Remark 3.15]%
{Sar21} for this bijection). Since we have a commutative diagram%
\begin{equation*}
\xymatrix@R=0.6cm{B\ar[r]^-{\omega _{B}}\ar[dr]_{i_{B}}& \left( B\widehat{\otimes }%
B\right) ^{\mathrm{co}B}\ar[d]^{\sigma _{\widehat{B}}}\\ &B\oslash B}
\end{equation*}%
we get that $i_{B}$ is bijective if and only if $\sigma _{\widehat{B}}$ is
bijective. Therefore, by \cite[Theorem 3.7]{Sar21}, the conditions \ref{item:1sided2}, \ref{item:1sided3} and \ref{item:1sided4} are equivalent.
Moreover \ref{item:1sided1} and \ref{item:1sided2} are equivalent in view of \cite[Lemma 2.3]{Sar21}.
In order to deal with \ref{item:1sided5} as well,
%for any right $B$-comodule $N$ set $N \operatorname{\tilde \otimes} B \coloneqq N^\bullet \otimes B^\bullet_\bullet$ and
consider the Hopf module $\tilde B \coloneqq B \operatorname{\tilde \otimes} B = B^\bullet \otimes B^\bullet_\bullet$. The associated component of $\sigma$ is
\[\sigma_{\tilde B} \colon B \boxslash B \to \overline{B^\bullet \otimes B^\bullet_\bullet}^B, \qquad x^i \otimes y_i \mapsto \overline{x^i \otimes y_i}.\]
A moment of thought shall convince the reader that
\[\overline{B^\bullet \otimes B^\bullet_\bullet}^B \cong \left(B \otimes B_\bullet\right) \otimes_B \Bbbk \cong B \otimes \left(B_\bullet \otimes_B \Bbbk\right) \cong B\]
via the isomorphism sending any $\overline{x \otimes y}$ to $x\varepsilon(y)$. Therefore, as above, $\sigma_{\tilde B}$ is bijective if and only if $p_B$ is bijective and so \ref{item:1sided2} implies \ref{item:1sided5}. If we assume \ref{item:1sided5} instead, then we know that $\sigma_{\tilde B}$ is invertible and, by mimicking the proof of \cite[Proposition 3.4]{Sar21}, we can show that $\sigma$ is a natural isomorphism, i.e., \ref{item:1sided2}.
\begin{invisible}
Let us see that this entails that $\sigma$ is a natural isomorphism. First of all, let $M$ be a Hopf module and observe that $\mu_M \colon M^\bullet \otimes B^\bullet_\bullet \to M^\bullet_\bullet$ is a morphism of Hopf modules, whence the following diagram commutes by naturality of $\sigma$
\[
\xymatrix @C=45pt{
\left(M^\bullet  \otimes B^\bullet_\bullet\right)^{\mathrm{co}B} \ar[r]^-{\mu_M^{\mathrm{co}B}} \ar[d]_-{\sigma_{M \tilde\otimes B}} & M^{\mathrm{co}B} \ar[d]^-{\sigma_M} \\
\overline{M^\bullet \otimes B^\bullet_\bullet}^B \ar[r]_-{\overline{\mu_M}} & \overline{M}^B
}
\]
and hence, if $\sigma_M^{-1}$ exists, then it has to satisfy
\[\sigma_M^{-1}\left(\overline{m}\right) = \sigma_M^{-1}\left(\overline{\mu_M}\left(\overline{m \otimes 1}\right)\right) = \mu_M^{\mathrm{co}B}\left(\sigma_{M \tilde\otimes B}^{-1}\left(\overline{m \otimes 1}\right)\right)\]
for every $m \in M$. Then, observe that $\delta_M \colon M^\bullet \to M \otimes B^\bullet$ is a right colinear map and therefore the following diagram commutes as well
\[
\xymatrix @C=45pt{
\left(M^\bullet  \otimes B^\bullet_\bullet\right)^{\mathrm{co}B} \ar[r]^-{(\delta_M \otimes B)^{\mathrm{co}B}} \ar[d]_-{\sigma_{M \tilde \otimes B}} & \left(M \otimes B^\bullet \otimes B^\bullet_\bullet\right)^{\mathrm{co}B} \ar[r]^-{\cong} & M \otimes \left(B \operatorname{\tilde \otimes} B\right)^{\mathrm{co}B} \ar[d]^-{M \otimes \sigma_{\tilde B}} \\
\overline{M^\bullet \otimes B^\bullet_\bullet}^B \ar[r]_-{\overline{\delta_M \otimes B}} & \overline{M \otimes B^\bullet \otimes B^\bullet_\bullet}^B \ar[r]_-{\cong} & M \otimes \overline{B \operatorname{\tilde \otimes} B}^B
}
\]
whence, if $\sigma_{M \tilde\otimes B}^{-1}$ exists, then it has to satisfy
\[\sigma_{M \tilde\otimes B}^{-1}\left(\overline{m \otimes b}\right) = (M \otimes \varepsilon \otimes B)(\delta_M \otimes B)\sigma_{M \tilde\otimes B}^{-1}\left(\overline{m \otimes b}\right) = (M \otimes \varepsilon \otimes B)\left(M \otimes \sigma_{\tilde B}^{-1}\right)\left(m_0 \otimes \overline{m_1 \otimes b}\right).\]
Summing up, if $\sigma_M^{-1}$ exists, then it has to satisfy
\[\sigma_M^{-1}\left(\overline{m}\right) = \mu_M\left((M \otimes \varepsilon \otimes B)\left(M \otimes \sigma_{\tilde B}^{-1}\right)\left(m_0 \otimes \overline{m_1 \otimes 1}\right)\right)\]
Define
\[S(a) \coloneqq (\varepsilon \otimes B)\sigma_{\tilde B}^{-1}\left(\overline{a \otimes 1}\right) = (\varepsilon \otimes B)\left(p_B^{-1}(a)\right),\]
so that
\[\sigma_M^{-1}\left(\overline{m}\right) = m_0S(m_1).\]
Now, since $p_B$ makes the following diagram commute
\[
\xymatrix @C=45pt{
\left(B \operatorname{\tilde \otimes} B\right)^{\mathrm{co}B} \ar[d]_-{\Delta \otimes B} \ar[r]^-{p_B} & B \ar[d]^-{\Delta} \\
B \otimes \left(B \operatorname{\tilde \otimes} B\right)^{\mathrm{co}B} \ar[r]_-{B \otimes p_B} & B \otimes B
}
\]
we conclude that for every $a \in B$,
\begin{align*}
\sigma_{\tilde B}^{-1}\left(\overline{a \otimes 1}\right) & = (B \otimes \varepsilon \otimes B)(\Delta \otimes B)\sigma_{\tilde B}^{-1}\left(\overline{a \otimes 1}\right) = (B \otimes \varepsilon \otimes B)(\Delta \otimes B)p_B^{-1}\left(a\right) \\
& = a_1 \otimes (\varepsilon \otimes B)\left(p_B^{-1}(a_2)\right) = a_1 \otimes S(a_2).
\end{align*}
Furthermore, since $\varepsilon(p_B(x^i \otimes y_i)) = \varepsilon(x^i)\varepsilon(y_i)$, we see that
\[\varepsilon(S(a)) = \varepsilon\left((\varepsilon \otimes B)\left(p_B^{-1}(a)\right)\right) = (\varepsilon \otimes \varepsilon)\left(p_B^{-1}(a)\right) = \varepsilon(a).\]
Therefore, by recalling that $\sigma_{\tilde B}^{-1}\left(\overline{a \otimes 1}\right) \in \left(B \operatorname{\tilde \otimes} B\right)^{\mathrm{co}B}$, we find that
\[a_1 \otimes S(a_2) \otimes 1 = a_1 \otimes S(a_3)_1 \otimes a_2S(a_3)_2\]
i.e.
\[S(a) \otimes 1 = S(a_2)_1 \otimes a_1S(a_2)_2.\]
If we also remark that for every $x^i \otimes y_i \in \left(B \operatorname{\tilde \otimes} B\right)^{\mathrm{co}B}$,
\[x^i_1 \otimes {y_i}_1 \otimes x^i_2{y_i}_2 = x^i \otimes {y_i} \otimes 1\]
entails that
\[x^i{y_i} = \varepsilon(x^i)\varepsilon({y_i})1,\]
whence $a_1S(a_2) = \varepsilon S(a)1 = \varepsilon (a)1$, and if we remark that $p_B$ is a morphism of algebras where $\left(B \operatorname{\tilde \otimes} B\right)^{\mathrm{co}B}$ has the restriction of the algebra structure of $B^\op \otimes B$, then for all $a,b \in B$,
\[a_1S(ba_2) = a_1(\varepsilon \otimes B)\left(p_B^{-1}(ba_2)\right) = a_1(\varepsilon \otimes B)\left(p_B^{-1}(a_2)\right)(\varepsilon \otimes B)\left(p_B^{-1}(b)\right) = \varepsilon(a)S(b).\]
In view of \cite[Proposition 3.4]{Sar21}, we conclude that $\sigma$ is a natural isomorphism, whence \ref{item:1sided2}.
\end{invisible}

Looking at \cite[Equality (10)]{Sar21}, we have $\sigma _{\widehat{B}%
}^{-1}\left( x\oslash y\right) =xS^{r}\left( y\right) \otimes 1$ so that $%
i_{B}^{-1}\left( x\oslash y\right) =\left( B\otimes \varepsilon \right)
\sigma _{\widehat{B}}^{-1}\left( x\oslash y\right) =xS^{r}\left( y\right) .$ Similarly, $p_B^{-1}(x) = \sigma_{\tilde B}^{-1}(\overline{x \otimes 1}) = x_1 \otimes S^r(x_2)$, since in general $\sigma_M^{-1}(\overline{m}) = m_0 S^r(m_1)$ for all $ m\in M$ right Hopf module.
\end{proof}

By using the map $i_B$, we can address the question whether $B\oslash B$ is not only a quotient coalgebra of $B\otimes B^\cop$, as stated in \cref{lem:oslash}, but even a quotient bialgebra, and also the dual question.

\begin{lemma}
\label{lem:invosl}
    Let $B$ be a bialgebra. Then $B \oslash B$ is a quotient bialgebra of $B\otimes B^\cop$ if and only if $\Delta(B^+)(B\otimes B)\subseteq (B\otimes B)\Delta(B^+)$. Moreover, when $B \oslash B$ is such a quotient and it further admits an antipode $S$, then the latter is given by $S(a\oslash b)= b\oslash a$ for every $a,b\in B$.
    %\begin{invisible}
    Similarly, $B \boxslash B$ is a sub-bialgebra of $B \otimes B^{\mathrm{op}}$ if and only if $\left({x^i}_1 \otimes {y_i}_1\right) \otimes \left({x^i}_2 \otimes {y_i}_2\right) \in (B \boxslash B) \otimes (B \otimes B)$ for all $x^i \otimes y_i \in B \boxslash B$ and if it further admits an antipode $S$, then $S(x^i \boxslash y_i) = y_i \boxslash x^i$.
    %\end{invisible}
\end{lemma}

\begin{proof}
% It follows from \cref{lem:oslash} that $(B\otimes B)B^+=(B\otimes B)\Delta(B^+)$ is a coideal in $B\otimes B^\cop$, and it is clearly a left ideal.
% Now, if $\Delta(B^+)(B\otimes B)\subseteq (B\otimes B)\Delta(B^+)$, then we get that $(B\otimes B)\Delta(B^+)$ is also a bi-ideal so that $B\oslash B=(B\otimes B)/(B\otimes B)\Delta(B^+)$ becomes a quotient bialgebra of $B\otimes B^\cop$.
% Conversely, if $B\oslash B$ is a quotient bialgebra of $B\otimes B^\cop$, then
% $(B\otimes B)\Delta(B^+)$ has to be a two-sided ideal and hence $\Delta(B^+)(B\otimes B)\subseteq (B\otimes B)\Delta(B^+)(B\otimes B)\subseteq (B\otimes B)\Delta(B^+)$.
It follows from \cref{lem:oslash} that $(B\otimes B)B^+=(B\otimes B)\Delta(B^+)$ is a coideal left ideal in $B\otimes B^\cop$. For $B \oslash B$ to be a quotient bialgebra of $B \otimes B^\cop$ it is necessary and sufficient that $(B \otimes B)\Delta(B^+)$ be 
%\rd{[credo che in inglese ci vada proprio "be" come dice il referee {\color{teal} Secondo ChatGPT, il "be" \`e tipico del American English, mentre in British English \`e pi\`u naturale il "is" o addirittura il "should be", ma possiamo anche mettere "be" come scrive il referee}]} 
a two-sided ideal, and this is clearly equivalent to the inclusion $\Delta(B^+)(B \otimes B) \subseteq (B \otimes B)\Delta(B^+)$.

Now, denote by $i_B^\dagger$ the linear map $B \to B \oslash B, b \mapsto 1 \oslash b$. When $B \oslash B$ is such a quotient, then its algebra structure is given by $(a\oslash b)(a'\oslash b')=aa'\oslash bb'$ with unity $1\oslash 1$, and so $(i_B*i_B^\dagger)(b)=i_B(b_1)i_B^\dagger(b_2)=(b_1\oslash 1)(1\oslash b_2)=b_1\oslash b_2=\varepsilon(b)1\oslash 1$, for every $b\in B$, i.e., $i_B*i_B^\dagger=u_{B\oslash B} \circ \varepsilon_B$.
Assume now that the bialgebra $B\oslash B$ has an antipode $S$.
Since $i_B*i_B^\dagger=u_{B\oslash B} \circ \varepsilon_B$ and since the coalgebra map $i_B \colon B \to B\oslash B$ is convolution invertible with inverse $S \circ i_B$, we get that $i_B^\dagger = S \circ i_B$ and so it is also the left convolution inverse of $i_B$.
In a similar way, $i_B^\dagger\colon B^{\mathrm{cop}} \to B \oslash B$ has right convolution inverse $i_B$ and left convolution inverse $S \circ i_B^\dagger$, so that we also have $i_B=S \circ i_B^\dagger$.
As a consequence $S(a \oslash b) = S(i_B(a)i_B^\dagger(b)) = Si_B^\dagger(b)Si_B(a) = i_B(b)i_B^\dagger(a) = b\oslash a$, for $a,b\in B$.

%\begin{comment}
\begin{invisible}
In view of \cref{lem:oslash}, $B \boxslash B$ is a subalgebra of $B \otimes B^{\mathrm{op}}$ and it is clearly an augmented subalgebra of it. Thus, $B \boxslash B$ is a sub-bialgebra of $B \otimes B^{\mathrm{op}}$ if and only if
\[\left({x^i}_1 \otimes {y_i}_1\right) \otimes \left({x^i}_2 \otimes {y_i}_2\right) \in (B \boxslash B) \otimes (B \boxslash B)\]
for all $x^i \otimes y_i \in B \boxslash B$. We already know from \cref{lem:oslash} that
\[\left({x^i}_1 \otimes {y_i}_1\right) \otimes \left({x^i}_2 \otimes {y_i}_2\right) \in (B \otimes B) \otimes (B \boxslash B),\]
so it is necessary and sufficient to have
\[\left({x^i}_1 \otimes {y_i}_1\right) \otimes \left({x^i}_2 \otimes {y_i}_2\right) \in (B \boxslash B) \otimes (B \otimes B).\]
\end{invisible}
Dually, one proves that $B \boxslash B$ is a sub-bialgebra of $B \otimes B^{\mathrm{op}}$ if and only if $\left({x^i}_1 \otimes {y_i}_1\right) \otimes \left({x^i}_2 \otimes {y_i}_2\right) \in (B \boxslash B) \otimes (B \otimes B)$ for all $x^i \otimes y_i \in B \boxslash B$.
Finally, suppose that $B \boxslash B$ is such a sub-bialgebra. Denote by $p_B^\dagger$ the linear map $B \boxslash B \to B$, $x^i \boxslash y_i \mapsto \varepsilon\left(x^i\right)y_i$.
Then, we have
\[(p_B \otimes p_B^\dagger)\Delta(x^i \boxslash y_i) = p_B({x^i}_1 \boxslash {y_i}_1) \otimes p_B^\dagger({x^i}_2 \boxslash {y_i}_2) = {x^i}_1\varepsilon({y_i}_1) \otimes \varepsilon({x^i}_2){y_i}_2 = x^i \otimes y_i\]
for every $x^i \boxslash y_i \in B \boxslash B$ so that
\begin{equation}
\label{eq:pBpBdag}
 (p_B \otimes p_B^\dagger)\Delta(x^i \boxslash y_i)=  x^i \otimes y_i. 
\end{equation}  
Therefore
\[\left(p_B*p_B^\dagger\right)(x^i \boxslash y_i) = m_B(p_B \otimes p_B^\dagger)\Delta(x^i \boxslash y_i)\overset{\eqref{eq:pBpBdag}}=m_B(x^i \boxslash y_i) = x^iy_i = \varepsilon\left(x^i\right)\varepsilon\left(y_i\right)1_B,\] 
in view of \cref{lem:oslash} again. If $B \boxslash B$ admits an antipode $S$, then necessarily $p_B^\dagger = p_B \circ S$. Similarly, the algebra map $p_B^\dagger \colon B \boxslash B \to B^{\mathrm{op}}$ has $p_B$ as right convolution inverse
\begin{invisible}
indeed
\[\left(p_B^\dagger*p_B\right)(x^i \boxslash y_i) = {x^i}_2\varepsilon\left({y_i}_2\right)\varepsilon\left({x^i}_1\right){y_i}_1 = x^iy_i = \varepsilon\left(x^i\right)\varepsilon\left(y_i\right)1_B,\]
\end{invisible}%
and $p_B \circ S$ as left one, implying that $p_B = p_B^\dagger \circ S$.
As a consequence, for all $x^i \boxslash y_i \in B \boxslash B$ we have
\begin{align*}
S\left(x^i \boxslash y_i\right) & \overset{\eqref{eq:pBpBdag}}= \left(p_B \otimes p_B^\dagger\right)\Delta\left(S\left(x^i \boxslash y_i\right)\right) = \left((p_B \circ S) \otimes (p_B^\dagger \circ S)\right)\Delta^{\mathrm{cop}}\left(x^i \boxslash y_i\right) \\
& = \left(p_B^\dagger \otimes p_B\right)\left({x^i}_2 \boxslash {y_i}_2 \otimes {x^i}_1 \boxslash {y_i}_1\right) = y_i \otimes x^i. \qedhere
\end{align*}
%for all $x^i \boxslash y_i \in B \boxslash B$.
%\end{comment}
\end{proof}

%\ps{The interested reader can verify that the dual result for $B \boxslash B$ holds as well.}

\begin{lemma}
\label{lem:foslf}
A bialgebra map $f:B\to C$ induces 
\begin{itemize}[leftmargin=*]
    \item a coalgebra map
$f\oslash f \colon B\oslash B\rightarrow C\oslash C, \; x\oslash y\mapsto f(x)\oslash f(y)$, and
\item an algebra map
$f \boxslash f \colon B \boxslash B \to C \boxslash C, \; x^i \boxslash y_i \mapsto f(x^i) \boxslash f(y_i)$. 
\end{itemize}
In particular, $\left( f\oslash f\right) \circ i_{B}=i_{C}\circ f$ and  $p_C \circ (f \boxslash f) = f \circ p_B$.
\begin{invisible}
, \rem{i.e.\ the following diagrams commute
\[
\xymatrix@R=0.6cm{
B \ar[r]^-{i_B} \ar[d]_-{f} & B \oslash B \ar[d]^-{f \oslash f} \\
C\ar[r]_-{i_{C}} & C \oslash C
}
\hspace{1.5cm}
\xymatrix@R=0.6cm{
B \boxslash B \ar[r]^-{p_B} \ar[d]_-{f \boxslash f} & B \ar[d]^-{f} \\
C \boxslash C \ar[r]_-{p_{C}} & C
}
\]}
\end{invisible}
\end{lemma}

\begin{proof}
We have $\left( f\otimes f\right) \left( \left( B\otimes B\right)
B^{+}\right) \subseteq \left( C\otimes C\right) C^{+}$ so that $f\otimes f$
induces a map $f\oslash f$ and we have
$\left( f\oslash f\right) i_{B}\left( b\right) =\left( f\oslash f\right)
\left( b\oslash 1_{B}\right) =f\left( b\right) \oslash 1_{C}=i_{C}f\left(
b\right)$ for every $b\in B$. Since $f$ is also a coalgebra map, it follows immediately from the definition of the coalgebra structure \eqref{def:oslashcoalg} that $f \oslash f$ is a morphism of coalgebras, too. Dually for $f \boxslash f$.
\begin{invisible}
\rem{Similarly, \eqref{eq:coinv} entails that
\[f(x^i)_1 \otimes f(y_i)_1 \otimes f(x^i)_2f(y_i)_2 = f({x^i}_1) \otimes f({y_i}_1) \otimes f({x^i}_2{y_i}_2) = f(x^i) \otimes f(y_i) \otimes f(1) = f(x^i) \otimes f(y_i) \otimes 1\]
so that $f \otimes f$ induces also a map $f \boxslash f$ and
\[p_C\left(f(x^i) \boxslash f(y_i)\right) = f(x^i)\varepsilon\left(f(y_i)\right) = f(x^i\varepsilon\left(y_i\right)) = f\left(p_B\left(x^i
 \boxslash y_i\right)\right)\]
 for every $x^i
 \boxslash y_i \in B \boxslash B$. Finally, the definition of the algebra structure \eqref{eq:boxslashalg} makes it clear that $f \boxslash f$ is a morphism of algebras.}
 \end{invisible}
\end{proof}

\begin{remark}
\label{rem:freecom}
Let $B$ be a bialgebra. 
\begin{enumerate}[label=(\alph*), ref=(\alph*), leftmargin=*]
\item\label{item:freecom1} The canonical map $i_B \colon B \to B \oslash B$ is a coalgebra map in view of \cref{lem:oslash} and if $f \colon B\to H$ is a bialgebra map into a right Hopf algebra $H$ with anti-multiplicative and anti-comultiplicative right antipode $S^r$, then we have a coalgebra map $\hat f \colon B\oslash B \to H$ defined by setting $\hat f \coloneqq i_H^{-1} \circ (f \oslash f)$. That is, $\hat f(x\oslash y) = f(x)S^r\big(f(y)\big)$ for all $x,y \in B$ and $\hat f$ satisfies $\hat f\circ i_B =f$.

\item\label{item:freecom2} If $f \colon B \to H$ is a morphism of bialgebras and $B \oslash B$ and $H \oslash H$ are quotient bialgebras of $B \otimes B^\cop$ and $H \otimes H^\cop$, respectively, then $f \oslash f$ is a morphism of bialgebras, too. This happens, for instance, when $B$ and $H$ are commutative. 

\item\label{item:freecom3} The canonical map $p_B \colon B \boxslash B \to B$ is an algebra map by \cref{lem:oslash} again and if $g \colon H \to B$ is a bialgebra map from a right Hopf algebra $H$ with anti-multiplicative and anti-comultiplicative right antipode $S^r$, then we have an algebra map $\tilde g \colon H \to B\boxslash B$ defined by setting $\tilde g \coloneqq (g \boxslash g) \circ p_H^{-1}$. That is, $\tilde g(x) = g\left(x_1\right) \boxslash g\left(S^r\left(x_2\right)\right)$ for every $x \in H$ and $\tilde g$ satisfies $p_B \circ \tilde g = g$.

\item\label{item:freecom4} If $g \colon H \to B$ is a morphism of bialgebras and $B \boxslash B$ and $H \boxslash H$ are sub-bialgebras of $B \otimes B^\op$ and $H \otimes H^\op$, respectively, then $g \boxslash g$ is a morphism of bialgebras, too. This happens, in particular, when $B$ and $H$ are cocommutative.

\item\label{item:freecom5} If $H$ is a right Hopf algebra as in \ref{item:freecom1} or \ref{item:freecom3}, then $H \oslash H$ becomes such a right Hopf algebra by transport of structures, i.e., with unit $1 \oslash 1$, multiplication
\[(x \oslash y) (u \oslash v) = xS^r(y)u \oslash v \qquad \text{for all } x,y,u,v \in H,\]
and right antipode $\mathcal{S}^r(x \oslash y) = 1 \oslash xS^r(y)$.
When $H$ is commutative, it is a commutative Hopf algebra and $H \oslash H$ is a quotient Hopf algebra of $H \otimes H^\cop$. In particular, if $B$ and $H$ are commutative and $f \colon B \to H$ is a bialgebra map, then $\hat f \colon B \oslash B \to H$ is a morphism of bialgebras, too. 
Similarly, $H \boxslash H$ is a right Hopf algebra with counit $\varepsilon \boxslash \varepsilon$, comultiplication
\[\underline{\Delta}(x \boxslash y) = \big(x_1 \boxslash S^r(x_2)\big) \otimes \big(x_3 \boxslash S^r(x_4)\varepsilon(y)\big) \qquad \text{for all } x,y \in H\]
and right antipode $\mathcal{S}^r(x \boxslash y) = S^r(x_2) \boxslash (S^r)^2(x_1)\varepsilon(y)$. When it is cocommutative, it is a cocommutative Hopf algebra and $H \boxslash H$ is a sub-Hopf algebra of $H \otimes H^\op$. In particular, if $B$ and $H$ are cocommutative and $g \colon H \to B$ is a bialgebra map, then $\tilde g \colon H \to B \boxslash B$ is a morphism of bialgebras. 
\end{enumerate}
\end{remark}

\begin{proposition}\label{prop:preserve_right_ants}
    Let $f:B\to C$ be a bialgebra map and assume that both $i_B$ and $i_C$ are bijective (equivalently, both $p_B$ and $p_C$ are bijective). Then $f$ is a right Hopf algebra map.
\end{proposition}

\begin{proof}
By \cref{prop:Frobenius}, both $B$ and $C$ have a right antipode, say $S_B$ and $S_C$, respectively. Moreover $i_{B}^{-1}\left( x\oslash y\right)
=xS_B\left( y\right) $ for every $x,y\in B$ and similarly for $S_C.$ Thus, \cref{lem:foslf} entails that
\[fS_B(y)=fi_{B}^{-1}\left( 1_B\oslash y\right) = i_{C}^{-1}(f\oslash f)\left( 1_B\oslash y\right) = i_{C}^{-1}\left( 1_C\oslash f(y)\right) = S_Cf(y)\]
and hence $fS_B=S_Cf$.
Similarly, if both $p_B$ and $p_C$ are bijective.
\begin{invisible}
\rem{, then
\[f\left(x_1\right) \boxslash f\left( S_B\left(x_2\right)\right) = \left(f \boxslash f\right)\left(p_B^{-1}(x)\right) = p_C^{-1}\left(f\left(x\right)\right) = f(x)_1 \boxslash S_C\left(f(x)_2\right)\]
and so, by applying $\varepsilon \otimes C$ to both sides, $f S_B = S_C f$.}
\end{invisible}
\end{proof}

\begin{remark}
 Note that a bialgebra map $f \colon B \to C$ between one-sided Hopf algebras needs not to preserve one-sided antipodes in general. Indeed, a genuine one-sided Hopf algebra has infinitely many one-sided antipodes \cite[\S2.2 Exercise 7, page 91]{Jacobson}, hence it suffices to consider the identity morphism of such a genuine one-sided Hopf algebra and two different one-sided antipodes. This does not contradict our result, as one-sided Hopf algebras admit at most one one-sided antipode which is an anti-bialgebra map (see \cite{NT}).
\end{remark}

\cref{prop:Frobenius} characterizes the bijectivity of $i_B$ and $p_B$.
However, often one has weaker conditions on a bialgebra $B$ guaranteeing that they are either injective or surjective. We begin by focusing on $i_B$ in \cref{ssec:i_B} and then we discuss the case of $p_B$ in \cref{ssec:p_B}.

\subsection{Injectivity and surjectivity of \texorpdfstring{$i_B$}{iB}}\label{ssec:i_B}

Let us start to investigate its injectivity.

\begin{proposition}
\label{lem:iB-inj}
Let $B$ be a bialgebra and consider the map $%
i_{B}:B\rightarrow B\oslash B,b\mapsto b\oslash 1_{B}.$

\begin{enumerate}[label=\alph*),ref={\itshape\alph*)},leftmargin=*]

\item\label{ib-inj_item1} If $B$ is a right Hopf algebra with right antipode $S^{r}$ which is an
anti-algebra map, then $i_{B}$ is injective.

\item\label{ib-inj_item2} If $f \colon B\rightarrow C$ is a bialgebra map, then $f(\ker(i_B))\subseteq \ker(i_C)$. In particular, if $i_{C}$ is injective, then $\ker\left( i_{B}\right) \subseteq \ker\left( f\right).$

\item\label{ib-inj_item3} If $B$ embeds into a bialgebra $C$ with $i_{C}$ injective, then so is $%
i_{B}$.

\item\label{ib-inj_item4} %\begin{invisible}\rd{[added on 25/06/2024]}\end{invisible}
If $i_B$ is injective, then so is $\can:B\otimes B\to B\otimes B,x\otimes y\mapsto xy_1\otimes y_2$.
\end{enumerate}
\end{proposition}

\begin{proof}
\ref{ib-inj_item1} If $\mathrm{Id}_{B}$ has a right convolution inverse $S^{r}$ which is an
anti-algebra map, define $f:B\oslash B\rightarrow B,a\oslash b\mapsto
aS^{r}\left( b\right) .$ This is well-defined as $f\left( ah_{1}\oslash
bh_{2}\right) =ah_{1}S^{r}\left( bh_{2}\right) =ah_{1}S^{r}\left(
h_{2}\right) S^{r}\left( b\right) =aS^{r}\left( b\right) \varepsilon \left(
h\right) =f\left( a\oslash b\right) \varepsilon \left( h\right) .$ Moreover $%
fi_{B}\left( b\right) =f\left( b\oslash 1_{B}\right) =bS^r\left( 1_{B}\right)
=b.$

\ref{ib-inj_item2} It follows from $\left( f\oslash f\right) \circ i_{B}=i_{C}\circ f$ (see \cref{lem:foslf}).
%By \cref{lem:foslf}, we have $\left( f\oslash f\right) \circ i_{B}=i_{C}\circ f$. Thus, if $b\in \ker\left( i_{B}\right) ,$ we get $i_{C}f\left( b\right) =0$ and hence $f\left( b\right) \in $ i.e. $b\in \ker\left( f\right) .$

\ref{ib-inj_item3} It follows from \ref{ib-inj_item2}.

\ref{ib-inj_item4} It follows from the equality
\begin{equation}
\label{eq:canpi}
(\pi\otimes B)\circ (B\otimes \Delta) \circ \can=i_B\otimes B
\end{equation}
where $\pi:B\otimes B\to B\oslash B$ is the canonical projection.
\end{proof}

Now we turn to the investigation of the surjectivity of $i_B$. Since this is the property that will play a central role in our simplified construction of the Hopf envelope, we dedicate additional attention to it. For the sake of future reference, we start with a few sufficient conditions.

\begin{proposition}
\label{lem:iB-su}Let $B$ be a bialgebra and consider the map $%
i_{B}:B\rightarrow B\oslash B,b\mapsto b\oslash 1_{B}.$

\begin{enumerate}[label = \alph*), ref = {\itshape\alph*)},leftmargin=*]
\item\label{iBsurj.item1} %\begin{invisible}\rd{[added on 25/06/2024]}\end{invisible}
If $\can:B\otimes B\to B\otimes B,x\otimes y\mapsto xy_1\otimes y_2,$ is surjective,  then so is  $i_B$.

\item\label{iBsurj.item2} If $B$ is a left Hopf algebra, then $i_{B}$ is surjective.

\item\label{iBsurj.item3} If $B$ is right Hopf algebra with right antipode $S^{r}$ which is an
anti-coalgebra map, then $i_{B}$ is surjective.

\item\label{iBsurj.item4} If $B$ is a quotient of a bialgebra $C$ with $i_{C}$ surjective, then so is $%
i_{B}$.
\end{enumerate}
\end{proposition}

\begin{proof}
\ref{iBsurj.item1} It follows from the equality $\pi\circ \can=i_B\circ (B\otimes\varepsilon)$, which can be deduced from \eqref{eq:canpi} by applying $B \otimes \varepsilon$ to both sides. %, where $\pi:B\otimes B\to B\oslash B$ is still the canonical projection.

\ref{iBsurj.item2} Let $S^l$ be a given left antipode. For all $a,b\in B,$ we have $a\otimes b=aS^{l}\left( b_{1}\right)
b_{2}\otimes b_{3}$. Thus, $\can$ is surjective and the claim follows from \ref{iBsurj.item1}.
%$=aS^{l}\left( b_{1}\right) \varepsilon \left( b_{2}\right)
%\oslash 1_{B}=i_{B}(aS^{l}\left( b\right) )$ so that $i_{B}$ is surjective.

\ref{iBsurj.item3} For all $a,b\in B,$ we have
\[
a\oslash b=aS^{r}\left( b_{2}\right)
_{1}\oslash b_{1}S^{r}\left( b_{2}\right) _{2}=aS^{r}\left( b_{3}\right)
\oslash b_{1}S^{r}\left( b_{2}\right) =aS^{r}\left( b_{2}\right) \varepsilon
\left( b_{1}\right) \oslash 1_{B}=i_{B}(aS^{r}\left( b\right) ).
\]

\ref{iBsurj.item4} Let $f \colon C\to B$ be the projection making $B$ a quotient of $C$. By \cref{lem:foslf}, the map $f\otimes f$ induces a (necessarily surjective) map $f\oslash f:C\oslash C\to B\oslash B $ such that $(f\oslash f)\circ i_C=i_B\circ f$. Since both $f\oslash f$ and $i_C$ are surjective so is $i_B\circ f$ and so $i_B.$
\end{proof}

\begin{remark}\label{rem:iBsurjnoninj} %\begin{invisible}\rd{[added on 25/06/2024]}\end{invisible}
We will see in \cref{prop:semiantip} that $i_B$ is surjective when $B$ is a finite-dimensional bialgebra. However, in the finite-dimensional case $\can$ surjective would imply $\can$ bijective, whence the existence of an antipode, which is not always the case. Thus, the converse implication of \cref{lem:iB-su}\,\ref{iBsurj.item1} is not true in general.
\end{remark}

We now get a more precise characterization of surjectivity of $i_B$.

\begin{proposition}\label{lem:iBsu}
  The following assertions are equivalent for a bialgebra $B$.
  \begin{enumerate}[label=(\alph*),ref={\itshape(\alph*)},leftmargin=*]
      \item\label{item:epi1} The map $i_B:B\to B\oslash B,\, x\mapsto x\oslash 1$, is surjective.
    \item\label{item:epi2} There is $S\in \mathrm{End}_\Bbbk(B)$  such that $1\oslash y= S(y)\oslash 1$, for every $y\in B$.
    \item\label{item:epi3} There is $S\in \mathrm{End}_\Bbbk(B)$  such that $x\oslash y= xS(y)\oslash 1$, for every $x,y\in B$.
     \item\label{item:epi4} For every $y\in B$ there is $y'\in B$ such that $1\oslash y= y'\oslash 1$.
     \item\label{item:epi5} For every $x,y\in B$ there is $y'\in B$ such that $x\oslash y= xy'\oslash 1$.
    \end{enumerate}
\end{proposition}
\begin{proof} \ref{item:epi1} $\Rightarrow$ \ref{item:epi2}. If $i_B$ is surjective, then there is a $\Bbbk$-linear map $s:B\oslash B\to B$ such that $i_B\circ s=\id_B$. For $y\in B$, set $S(y):=s(1\oslash y)$. Then $S(y)\oslash 1=i_B(S(y))=i_B(s(1\oslash y))=\id_B(1\oslash y)=1\oslash y$.

\ref{item:epi2} $\Rightarrow$ \ref{item:epi3}.
By employing the left $B\otimes B$-module structure of $B\oslash B$ given in \cref{lem:oslash}, and assuming there is $S\in \mathrm{End}_\Bbbk(B)$  such that $1\oslash y= S(y)\oslash 1$, for every $y\in B$, we get $x\oslash y=(x\otimes 1)\cdot (1\oslash y)=(x\otimes 1)\cdot (S(y)\oslash 1)=xS(y)\oslash 1$.

\ref{item:epi3} $\Rightarrow$ \ref{item:epi1}. It is clear.

The equivalence between \ref{item:epi1}, \ref{item:epi4} and \ref{item:epi5} can be handled in a similar way.
\begin{invisible}
$(5)\Rightarrow(1)\Rightarrow (4)$ are clear.  Let us prove $(4)\Rightarrow (5)$. We have
$x\oslash y=(x\otimes 1)\cdot (1\oslash y)=(x\otimes 1)\cdot (y'\oslash 1)=xy'\oslash 1$.
\end{invisible}
\end{proof}

\begin{example}
%\begin{invisible}\ps{[added 5/7/24]}\end{invisible}
\label{ex:regmon}
Let $M$ be a regular monoid (see e.g.\ \cite[\S1.9]{CP}). Then, for every $x \in M$ there exists $x^\dagger \in M$ such that $x \cdot x^\dagger \cdot x = x$.
Let $B$ be the monoid bialgebra $\Bbbk M$. In $B \oslash B$ we have
\[1 \oslash x = x^\dagger x \oslash xx^\dagger x = x^\dagger x \oslash x = x^\dagger \oslash 1\]
and therefore $i_B$ is surjective by \cref{lem:iBsu}.
\begin{invisible}
In AdjMon abbiamo già i regolari (Example 2.21). Abbiamo voluto avere anche qui un esempio in cui si vedesse "concretamente" come può essere fatta la $S$. Il caso di dimensione finita è potente, ma forse poco "concreto".]
\end{invisible}

For a concrete example, let $f \colon H \to G$ be a group homomorphism and let $M \coloneqq H \sqcup G$ be equipped with the multiplication $\cdot \colon M \times M \to M$ given by
\[h \cdot h'= hh', \quad h\cdot g = f(h)g, \quad g \cdot h = gf(h), \quad g\cdot g' = gg'\]
for all $h,h'\in H$, $g,g'\in G$. Then $M$ is a regular (even inverse) monoid with neutral element $1_H$, $h^\dagger = h^{-1}$ and $g^\dagger = g^{-1}$, but it is not a group, in general: the elements of $G$ have no inverse with respect to $\cdot$.
\end{example}

\begin{corollary}%[\ps{extracted on 14 07 2025}]
\label{cor:Sprops}
    If $B$ is a bialgebra for which $i_B$ is surjective, then 
    \begin{equation}\label{eq:Sbar}
    \begin{gathered}
    S(ab) - S(b)S(a) ~ \in ~ \ker(i_B), \qquad a_1S(a_2) - \varepsilon(a)1_B ~ \in ~ \ker(i_B) \qquad \text{and} \\
    S(a_2) \otimes S(a_1) - S(a)_1 \otimes S(a)_2 ~ \in ~ \ker(i_B \otimes i_B) \qquad \text{for all } a,b \in B,
    \end{gathered}
    \end{equation}
    where $S$ is a linear endomorphism of $B$ as in \cref{lem:iBsu}.
\end{corollary}

\begin{proof}
    Since $i_B$ is surjective, \cref{lem:iBsu}\,\ref{item:epi2} entails that there exists $S \colon B \to B$ such that $1 \oslash b = S(b) \oslash 1$ for all $b \in B$. By using the left $B \otimes B$-module structure on $B \oslash B$, one shows that
    \[S(ab) \oslash 1_B = 1_B \oslash ab = (1_B \otimes a)(1_B \oslash b) = (1_B\otimes a)(S(b) \oslash 1_B) = (S(b) \otimes 1_B)(1_B \oslash a) = S(b)S(a) \oslash 1_B\]
    for all $a,b \in B$. Furthermore, for all $a \in B$ we have also
    $a_1S(a_2) \oslash 1 = a_1 \oslash a_2 = \varepsilon(a)1_B \oslash 1_B$
    and
    \begin{align*}
    \left(S(a_2) \oslash 1_B\right) \otimes \left(S(a_1) \oslash 1_B\right) & = \left(1_B \oslash a_2\right) \otimes \left(1_B \oslash a_1\right) \stackrel{\eqref{def:oslashcoalg}}{=} \Delta(1_B \oslash a) \\
    & = \Delta(S(a) \oslash 1_B) = \left(S(a)_1 \oslash 1_B\right) \otimes \left(S(a)_2 \oslash 1_B \right). \qedhere
    \end{align*}
    %for all $a \in B$. 
    % and hence
    % \[
    % \begin{aligned}
    % S(ab) - S(b)S(a) & \in K \subseteq KB = \ker(q_B), \\
    % a_1S(a_2) - \varepsilon(a)1_B & \in K \subseteq KB = \ker(q_B), \\
    % S(a_2) \otimes S(a_1) - S(a)_1 \otimes S(a)_2 & \in \ker(i_B \otimes i_B) = K \otimes B + B \otimes K \\
    % & \subseteq KB \otimes B + B \otimes KB = \ker(q_B \otimes q_B).
    % \end{aligned}
    % \]
\end{proof}

\begin{invisible}
\begin{remark}%[\ps{added on 14 07 2025 - coming from IntBialg}]
    If $B \oslash B$ is a Hopf algebra with antipode $S_{B \oslash B}$ and if $i_B \colon B\to B\oslash B$ is surjective, then 
    % \ps{$i_B$ is automatically a bialgebra map, but the only proof that I know uses the Hopf envelope, and} 
    \[S_{B\oslash B}(a\oslash b)=S^2(b)\oslash a,\] 
    where $S$ is as in \cref{lem:iBsu}. Indeed, 
    \eqref{eq:Sbar} entails that $i_B * (i_B \circ S) = u_{B \oslash B} \circ \varepsilon_B$ and so $i_B \circ S = S_{B \oslash B} \circ i_B$ by uniqueness of the convolution inverse. Therefore,
    \begin{align*}
    S_{B\oslash B}(a\oslash b) & = S_{B\oslash B}(aS(b) \oslash 1) = S_{B\oslash B}i_B(aS(b)) = i_B\big(S(aS(b))\big) \\
    & \stackrel{\eqref{eq:Sbar}}{=} i_B(S^2(b)S(a)) = S^2(b)S(a) \oslash 1 = S^2(b) \oslash a.
    \end{align*}
    %
    % $i_BS(x) = S(x)\oslash 1 = 1\oslash x = i_B^\dagger(x)$ for all $x \in B$ entails that $i_B^\dagger = i_B\circ S$ and so $x\oslash y = xS(y)\oslash 1 = i_B(xS(y)) = i_B(x)i_B(S(y)) = i_B(x)i_B^\dagger(y) = (x\oslash 1)(1\oslash y)$. Thus $(i_B*i_B^\dagger)(x) = i_B(x_1)i_B^\dagger(x_2) = x_1 \oslash x_2 = \varepsilon(x)1\oslash 1$ and therefore $S_{B\oslash B}\circ i_B = i_B^\dagger$. Hence
    % \begin{align*}
    % S_{B\oslash B}(x\oslash y) & = S_{B\oslash B}(xS(y) \oslash 1) = S_{B\oslash B}i_B(xS(y)) = i_B^\dagger(xS(y)) = 1 \oslash xS(y) \\
    % & = (1 \otimes x) \cdot (1 \oslash S(y)) = (1 \otimes x) \cdot (S^2(y) \oslash 1) = S^2(y) \oslash x.
    % % S_{B\oslash B}(x\oslash y) & =S_{B\oslash B}(1\oslash y)S_{B\oslash B}(x\oslash 1) = S_{B\oslash B}(S(y)\oslash 1)S_{B\oslash B}(x\oslash 1) \\
    % % & = S_{B\oslash B}i_B(S(y))S_{B\oslash B}i_B(x) = j_B(S(y))j_B(x) \\
    % % & = i_B(S^2(y))j_B(x) = S^2(y)\oslash  x.
    % \end{align*}
\end{remark}
\end{invisible}

% \ps{[From IntBialg]:} \rd{[If $i_B:B\to B\oslash B$ is a surjective bialgebra map and $B\oslash B$ has antipode $S_{B\oslash B}$, then \[S_{B\oslash B}(x\oslash y)=S^2(y)\oslash x,\] where $S$ is as in \cref{lem:iBsu}. Indeed
% $i_BS(x)=S(x)\oslash 1=1\oslash x=j_B(x)$ so that $j_B=i_B\circ S$. Thus
% $x\oslash y=xS(y)\oslash 1=i_B(xS(y))=i_B(x)i_B(S(y))=i_B(x)j_B(y)=(x\oslash 1)(1\oslash y)$. Thus $(i_B*j_B)(x)=i_B(x_1)j_B(x_2)=x_1\oslash x_2=\varepsilon(x)1\oslash 1$. Therefore $S_{B\oslash B}\circ i_B=j_B$.\\
% Hence
% \begin{align*}
% S_{B\oslash B}(x\oslash y) & =S_{B\oslash B}(1\oslash y)S_{B\oslash B}(x\oslash 1) = S_{B\oslash B}(S(y)\oslash 1)S_{B\oslash B}(x\oslash 1) \\
% & = S_{B\oslash B}i_B(S(y))S_{B\oslash B}i_B(x) = j_B(S(y))j_B(x) \\
% & = i_B(S^2(y))j_B(x) = S^2(y)\oslash  x.
% \end{align*}]}

\subsection{Injectivity and surjectivity of \texorpdfstring{$p_B$}{pB}}\label{ssec:p_B}
% (\ps{added 05 01 2025})

Symmetrically to what we did with $i_B$, let us start by investigating the surjectivity of $p_B$. Since most of the proofs are dual to those of \S\ref{ssec:i_B}, we often omit them.

\begin{proposition}%[\ps{added 06 01 2025}]
\label{prop:pB-surj}
Let $B$ be a bialgebra and $p_B \colon B\boxslash B \to B, x^i \boxslash {y_i} \mapsto x^i\varepsilon({y_i})$.

\begin{enumerate}[label=\alph*),ref={\itshape\alph*)},leftmargin=*]

\item\label{pB-surj_item1} If $B$ is a right Hopf algebra with right antipode $S^{r}$ which is an
anti-coalgebra map, then $p_{B}$ is surjective.

\item\label{pB-surj_item2} Let $f \colon C\rightarrow B$ be a bialgebra map. Then $f(\im(p_C)) \subseteq \im(p_B)$. In particular, if $p_{C}$ is surjective, then $\im(f) \subseteq \im(p_B).$

\item\label{pB-surj_item3} If $B$ is a quotient of a bialgebra $C$ with $p_{C}$ surjective, then so is $%
p_{B}$.

\item\label{pB-surj_item4} %\begin{invisible}\rd{[added on 25/06/2024]}\end{invisible}
If $p_B$ is surjective, then so is $\can' \colon B\otimes B\to B\otimes B,x\otimes y\mapsto x_1\otimes x_2y$.
\end{enumerate}
\end{proposition}

\begin{invisible}
\rem{
\begin{proof}
\ref{pB-surj_item1} If $\mathrm{Id}_{B}$ has a right convolution inverse $S^{r}$ which is an
anti-coalgebra map, define $f \colon B \to B\boxslash B, b \mapsto
b_1 \boxslash S^{r}\left( b_2\right) .$ This lands, in fact, in $B \boxslash B$ as
\[{b_1}_1 \otimes S^{r}\left( b_2\right)_1 \otimes {b_1}_2 S^{r}\left( b_2\right)_2 = {b_1} \otimes S^{r}\left( b_4\right) \otimes {b_2} S^{r}\left( b_3\right) = {b_1} \otimes S^{r}\left( b_2\right) \otimes 1.\]
Moreover, $p_Bf\left( b\right) = {b_1} \varepsilon\left( S^{r}\left( b_2\right) \right)
= b$.

\ref{pB-surj_item2}
It follows from $p_B \circ \left( f\boxslash f\right) = f \circ p_C$ (see \cref{lem:foslf}).
% By \cref{lem:foslf}, we have $p_B \circ \left( f\boxslash f\right) = f \circ p_C$. This, together with the surjectivity of $p_C$, implies that $\id_B$ induces a (unique) surjective linear morphism $\mathrm{coker}(f) \to \mathrm{coker}(p_B)$ and hence $f\left( C\right) \subseteq p_B\left(B \boxslash B\right)$.

\ref{pB-surj_item3} It follows from \ref{pB-surj_item2}.

\ref{pB-surj_item4}
%  Recall from \eqref{eq:coinv} that for every $x^i \boxslash {y_i} \in B\boxslash B=\left(B^\bullet \otimes B^\bullet\right)^{\mathrm{co}B}$ we have that
% \begin{equation}\label{eq:coinvs}
%     x^i_1 \otimes {y_i}_1 \otimes x^i_2{y_i}_2 = x^i \otimes {y_i} \otimes 1
%     \end{equation}
% so that
It follows from \eqref{eq:coinvs} that we have
$x^i_1 \otimes x^i_2{y_i} = x^i \varepsilon\left( y_i \right) \otimes 1$ for every $x^i \boxslash y_i \in B \boxslash B$.
Thus, the claim follows from the equality
\begin{equation}\label{eq:canj}
\can' \circ (B \otimes m) \circ (j_B \otimes B) = p_B\otimes B
\end{equation}
where $j_B \colon B\boxslash B\to B\otimes B$ is the inclusion.
\end{proof}
}
\end{invisible}

Now we turn to the investigation of the injectivity of $p_B$. \cref{prop:Frobenius} ensures that $p_B$ is injective for $B$ a right Hopf algebra with right antipode which is an anti-bialgebra morphism, in particular for $B$ a Hopf algebra. More generally, we have the following analogue of \cref{lem:iB-su}.

%\ps{[Possiamo prendere in considerazione l'idea di commentare tutte le dim "ripetitive"]}

\begin{proposition}
\label{prop:pbInj}
    Let $B$ be a bialgebra and $p_B \colon B\boxslash B \to B, x^i \boxslash {y_i} \mapsto x^i\varepsilon({y_i})$.
    \begin{enumerate}[label=\alph*),ref={\itshape\alph*)},leftmargin=*]
        \item\label{item:pbInj1} If $\mathfrak{can}' \colon B \otimes B \to B \otimes B, x \otimes y \mapsto x_1 \otimes x_2y,$ is injective, then so it $p_B$.
        \item\label{item:pbInj2} If $B$ is a left Hopf algebra, then $p_B$ is injective.
        \item\label{item:pbInj3} If $B$ is a right Hopf algebra with anti-multiplicative right antipode, then $p_B$ is injective.
        \item\label{item:pbInj4} If $B$ is a sub-bialgebra of a bialgebra $C$ with $p_C$ injective, then $p_B$ is injective.
    \end{enumerate}
\end{proposition}

\begin{invisible}
\rem{
\begin{proof}
    \ref{item:pbInj1}
Composing both sides of \eqref{eq:canj} with $\id_{B\boxslash B}\otimes u$ on the right yields $\can' \circ j_B = (B\otimes u)\circ p_B$
    % Observe that for every $x^i \otimes {y_i} \in B\boxslash B=\left(B^\bullet \otimes B^\bullet\right)^{\mathrm{co}B}$ we have that
    % \begin{equation}\label{eq:coinvs}
    % x^i_1 \otimes {y_i}_1 \otimes x^i_2{y_i}_2 = x^i \otimes {y_i} \otimes 1
    % \end{equation}
    % and so
    % \[x^i_1 \otimes x^i_2{y_i} = x^i \varepsilon\left( {y_i} \right) \otimes 1.\]
    % Hence, the following diagram commutes:
    % \[
    % \xymatrix{
    % B\boxslash B \ar[r]^-{p_B} \ar[d]_-{\gamma} & B \ar[d]^-{B \otimes u} \\
    % B \otimes B \ar[r]_-{\mathfrak{can}} & B \otimes B
    % }
    % \]
    % where $\gamma \colon B\boxslash B \to B \otimes B$ is the canonical inclusion.
    Thus, if $\mathfrak{can}'$ is injective, so is $(B \otimes u) \circ p_B$ and hence $p_B$ is injective.

    \ref{item:pbInj2} If $S^l$ is a left antipode, then $B \otimes B \to B \otimes B, x \otimes y \mapsto x_1 \otimes S^l(x_2)y$ is a left inverse of $\mathfrak{can}'$, which is therefore injective. Thus, it follows from \ref{item:pbInj1}.

    \ref{item:pbInj3} Let $S^r$ be an anti-multiplicative right antipode and consider $\varphi \colon B \to B \otimes B, b \mapsto b_1 \otimes S^r(b_2)$. %Then, we have $1 = \varepsilon(1)1 = 1S^r(1) = S^r(1)$ \rd{[Il fatto che $S^r(1)=1$ è già usato in \cref{rem:freecom}, quindi sposterei là il conto precedente]}
    Recall from \cite[Remark 3.8]{Sar21} that $S^r(1) = 1$. From \eqref{eq:coinvs} it follows that
    \[x^i \otimes {y_i} = x^i \otimes {y_i}S^r(1) = x^i_1 \otimes {y_i}_1 S^r(x^i_2{y_i}_2) = x^i_1 \otimes S^r(x^i_2)\varepsilon({y_i}) = \left(\varphi \circ p_B\right)\left(x^i \otimes {y_i}\right),\]
    for all $x^i \otimes y_i \in B \boxslash B$, i.e. $\varphi \circ p_B$ is the inclusion $j_B \colon B \boxslash B\to B\otimes B,$  and so the claim follows.
    %from \cref{prop:pBin}\ref{item:pBin3} with $T = S^r$.

    \ref{item:pbInj4}
    By \cref{lem:foslf}, a morphism of bialgebras $f \colon B \to C$ induces a morphism of algebras $f \boxslash f \colon B\boxslash B \to C\boxslash C$ such that $p_C \circ (f \boxslash f) = f \circ p_B$. If $f$ is injective, then $f \otimes f$ is injective and so $f \boxslash f$ is injective as well. In this case, if also $p_C$ is injective, then $f \circ p_B$ has to be injective and hence $p_B$ is injective, too.
\end{proof}
}
\end{invisible}

\begin{remark}
\cref{prop:semiantip} will entail also that a counterpart of \cref{rem:iBsurjnoninj} holds for $p_B$: the converse of \cref{prop:pbInj}\,\ref{item:pbInj1} does not hold, in general.
\end{remark}

Again, in view of the role that $p_B$ will play later, we are interested in getting a more precise characterization of its injectivity as we did for the surjectivity of $i_B$. Let us then prove the following analogue of \cref{lem:iBsu}.

\begin{proposition}
 \label{prop:pBin}
 The following assertions are equivalent for a bialgebra $B$.
 \begin{enumerate}[label=(\alph*),ref={\itshape(\alph*)},leftmargin=*]
    \item\label{item:pBin1} The map $p_B:B\boxslash B\to B,x^i\otimes {y_i}\mapsto x^i\varepsilon({y_i})$ is injective.

    \item\label{item:pBin2} There is $T\in\mathrm{End}_\Bbbk(B) $ such that for every $x^i\otimes {y_i}\in B\boxslash B$ we have
    \begin{equation}\label{eq:Tcond}
    T\left(x^i\right)\varepsilon\left(y_i\right) = \varepsilon\left(x^i\right)y_i.
    \end{equation}
    %$T(x^i)\varepsilon({y_i})=\varepsilon(x^i) {y_i}$, 
    %for every $x^i\otimes {y_i}\in B\boxslash B$.

    \item\label{item:pBin3} There is  $T\in\mathrm{End}_\Bbbk(B) $ such that $x^i_1\otimes T(x^i_2)\varepsilon({y_i}) = x^i\otimes {y_i}$, for every $x^i\otimes {y_i}\in B\boxslash B$.

     \item\label{item:pBin4} %\rd{[added on 2025/01/07]}
     For every $x^i\otimes {y_i}\in B\boxslash B$, there is a sub-coalgebra $C$ of $B$ with $x^i,y_i\in C$ and a linear map $\tau:C\to B$ such that $\tau(x^i)\varepsilon({y_i})=\varepsilon(x^i) {y_i}$.

     \item\label{item:pBin5} %\rd{[added on 2025/01/07]}
     For every $x^i\otimes {y_i}\in B\boxslash B$, there is a sub-coalgebra $C$ of $B$ with $x^i,y_i\in C$ and a linear map $\tau:C\to B$ such that $x^i_1\otimes \tau(x^i_2)\varepsilon({y_i})
     =x^i\otimes {y_i}$.
 \end{enumerate}
 \end{proposition}

 \begin{proof}
 First of all, note that $B\boxslash B$ is a left $B$-comodule via $\rho \colon B\boxslash B\to B\otimes (B\boxslash B)$ given by $\rho(x^i\otimes {y_i}) = x^i_1\otimes(x^i_2\otimes {y_i})$ for all $x^i \otimes y_i \in B \boxslash B$.

 \ref{item:pBin1} $\Rightarrow$ \ref{item:pBin2}. Assume $p_B$ is injective. Then there is a linear map $t:B\to B\boxslash B$ such that $t\circ p_B=\id.$ Define $T\in\mathrm{End}_\Bbbk(B) $ by setting $T(b):=(\varepsilon\otimes B)(t(b))$, for every $b\in B$. Then, for every $x^i\otimes {y_i}\in B\boxslash B$ \[T(x^i)\varepsilon({y_i})=Tp_B(x^i\otimes {y_i})=(\varepsilon\otimes B)(tp_B(x^i\otimes {y_i}))=(\varepsilon\otimes B)(x^i\otimes {y_i})=\varepsilon(x^i) {y_i}.\]

\ref{item:pBin2} $\Rightarrow$ \ref{item:pBin3}.  Let $T\in\mathrm{End}_\Bbbk(B) $ be such that $T(x^i)\varepsilon({y_i})=\varepsilon(x^i) {y_i}$ for every $x^i\otimes {y_i}\in B\boxslash B$. Then
\[x^i_1\otimes T(x^i_2)\varepsilon({y_i})\overset{(*)}{=}x^i_1\otimes \varepsilon(x^i_2){y_i}
 =x^i\otimes {y_i}\]
where in $(*)$ we used that $x^i_1 \otimes x^i_2\otimes {y_i}\in B \otimes B\boxslash B$.

% \ref{item:pBin3} $\Rightarrow$ \ref{item:pBin1}. Define $g:B\to B\otimes B$ by $g=(B\otimes T)\circ \Delta$. Then
%  \[g(p_B(x^i\otimes {y_i}))=g(x^i\varepsilon( {y_i}))=x^i_1\otimes T(x^i_2)\varepsilon({y_i})
%  =x^i\otimes {y_i}.\]
% Thus $g\circ p_B: B\boxslash B\to B\otimes B$ in injective and hence so is $p_B$. 
% \rd{[Rimuovere (c)=>(a)?] {\color{teal} Buona idea!}}

% \ref{item:pBin2} $\Rightarrow$ \ref{item:pBin4}. It is trivial.
\ref{item:pBin3} $\Rightarrow$ \ref{item:pBin5}. It is trivial.

% \ref{item:pBin4} $\Rightarrow$ \ref{item:pBin5}.
% Let $\{u^j\mid j \in I\}$ be a basis for $B$. For $x^i \otimes y_i \in B \boxslash B$, write $\Delta(x^i) = u^{j} \otimes v^i_{j}$.
% Again, $u^j \otimes \left(v^i_{j} \otimes {y_i}\right) \in B \otimes \left(B\boxslash B\right)$, so there is a sub-coalgebra $C$ of $B$ with $v^i_{j},y_i\in C$ and a linear map $\tau \colon C\to B$ such that $\tau(v^i_{j})\varepsilon({y_i})=\varepsilon(v^i_{j}) {y_i}$.
% Thus, $x^i=\varepsilon(u^j)v^i_{j}\in C$ and
% \[x^i_1\otimes \tau(x^i_2)\varepsilon({y_i})=u^j\otimes \tau(v^i_{j})\varepsilon({y_i})=u^j\otimes \varepsilon(v^i_{j}){y_i}=x^i_1\otimes \varepsilon(x^i_2){y_i}
%  =x^i\otimes {y_i}.\]
\ref{item:pBin5} $\Rightarrow$ \ref{item:pBin4}.
This is a straightforward verification, since under condition \ref{item:pBin5} we have
\[\tau(x^i)\varepsilon(y_i) = \varepsilon(x^i_1)\tau(x^i_2)\varepsilon(y_i) = \varepsilon(x^i)y_i\]
for every $x^i \otimes y_i \in B \boxslash B$.

% \ref{item:pBin5} $\Rightarrow$ \ref{item:pBin1}.
% Let $x^i\otimes {y_i} \in \ker(p_B)$ and let $C$ and $\tau$ be as in statement \ref{item:pBin5}.
% Define $\gamma:C\to B\otimes B$ by $\gamma=(\iota\otimes \tau)\circ \Delta_C$, where $\iota:C\to B$ is the inclusion. Then
%  \[0=\gamma(0)=\gamma(p_B(x^i\otimes {y_i}))=\gamma(x^i\varepsilon( {y_i}))=x^i_1\otimes \tau(x^i_2)\varepsilon({y_i})
%  =x^i\otimes {y_i}\]
% so that $p_B$ is injective.
\ref{item:pBin4} $\Rightarrow$ \ref{item:pBin1}.
Let $\{u^j\mid j \in J\}$ be a basis for $B$ and let $x^i\otimes {y_i} \in \ker(p_B)$. For every $i$, write $\Delta(x^i) = u^{j} \otimes v^i_{j}$.
Since $u^j \otimes \left(v^i_{j} \otimes {y_i}\right) \in B \otimes \left(B\boxslash B\right)$, for every $j \in J$ we have $v^i_{j} \otimes {y_i} \in B \boxslash B$, and so there is a sub-coalgebra $C_j$ of $B$ with $v^i_{j},y_i\in C_j$ and a linear map $\tau_j \colon C_j \to B$ such that $\tau_j(v^i_{j})\varepsilon({y_i}) = \varepsilon(v^i_{j}) {y_i}$. Now, $x^i\varepsilon(y_i) = 0$ entails that $0 = \Delta(x^i)\varepsilon(y_i) = u^j \otimes v^i_j\varepsilon(y_i)$ and hence $v^i_j\varepsilon(y_i) = 0$ for every $j \in J$. Therefore, 
\[0 = \tau_j\left(v^i_j\varepsilon(y_i)\right) = \tau_j(v^i_{j})\varepsilon({y_i}) = \varepsilon(v^i_{j}) {y_i}\]
for every $j \in J$, which implies that
\[x^i \otimes y_i = u^j \otimes \varepsilon(v_j^i)y_i = 0,\]
so that $p_B$ is injective.
%\rd{[La correzione va bene ma, detta tra noi, si poteva modificare (d) prendendo una famiglia di elementi invece di uno solo. {\color{teal}In effetti, lo stavo facendo in principio. Poi ho pensato che sarebbe caduta la simmetria con (e) e quindi ho deciso di girare la frittata.} ok!]}
\end{proof}

\begin{remark}
    Let $B$ be a bialgebra, $T \in \mathrm{End}_\Bbbk(B)$ a linear endomorphism, and $a \in \im(p_B)$. Then, there exists $x^i \otimes y_i \in B \boxslash B$ such that $a = x^i\varepsilon(y_i)$ and so $a_1 \otimes T(a_2) = x^i_1 \otimes T(x^i_2)\varepsilon(y_i)$. In this setting, $a_1 \otimes T(a_2)$ equals $x^i \otimes y_i$ if and only if $T$ is as in \cref{prop:pBin}\,\ref{item:pBin3} above.
    % E.g., 
    % \[a_1 \otimes T(a_2) = x^i_1 \otimes T(x^i_2)\varepsilon(y_i) = x^i \otimes y_i\]
    This observation offers also an elementary way of proving that \cref{prop:pBin}\,\ref{item:pBin3} implies that $p_B$ is injective 
    %\rd{[la frase qui a sinistra va rimossa se nella dim rimuoviamo (c)=>(a) {\color{teal}Pi\`u che rimuoveral, l'ho modificata accordingly}]},
    % Because $a_1 \otimes T(a_2) = x^i \otimes y_i$ for every $x^i \otimes y_i$ such that $a = x^i\varepsilon(y_i)$.
    and it characterizes the operator $T$ by means of the identity
    \[a_1 \otimes T(a_2) = p_B^{-1}(a) \in B \boxslash B \qquad \text{for all } a \in \im(p_B).\]
\end{remark}

Dually with respect to \cref{cor:Sprops}, one obtains the following result.

\begin{corollary}%[\ps{extracted 14 07 2025}]
\label{cor:Tantis}
    If $B$ is a bialgebra for which $p_B$ is injective, then for $a,b \in \im(p_B)$ we have
    \begin{equation}\label{eq:Tprops}
    a_1T(a_2)=\varepsilon(a)1, \qquad T(a_1) \otimes T(a_2) = T(a)_2 \otimes T(a)_1, \qquad T(ab) = T(b)T(a),
    \end{equation}
    where $T$ is a linear endomorphism of $B$ as in \cref{prop:pBin}.
\end{corollary}

\begin{invisible}
\begin{proof}
    Since $p_B$ is injective, there is a linear endomorphism $T$ of $B$ for which \eqref{eq:Tcond} holds.
    % such that
    % \begin{equation}\label{eq:Tcond}
    % T\left(x^i\right)\varepsilon\left(y_i\right) = \varepsilon\left(x^i\right)y_i
    % \end{equation}
    % for all $x^i \boxslash y_i \in B \boxslash B$. 
    In view of the fact that $B \boxslash B$ is a left $B \otimes B^\op$-comodule, we know that for every $x^i \boxslash y_i \in B \boxslash B$ we have
    \[\left(x^i_1 \otimes {y_i}_1\right) \otimes \left(x^i_2 \boxslash {y_i}_2\right) \in (B \otimes B) \otimes (B \boxslash B)\]
    and hence we can compute
    \begin{align*}
        \Delta(T(x^i\varepsilon(y_i))) & \stackrel{\eqref{eq:Tcond}}{=} \Delta(\varepsilon(x^i)y_i) = \varepsilon(x^i) {y_i}_1 \otimes {y_i}_2 = \varepsilon(x^i_1) {y_i}_1 \otimes \varepsilon(x^i_2){y_i}_2 \\
        & \stackrel{\eqref{eq:Tcond}}{=} \varepsilon(x^i_1) {y_i}_1 \otimes T(x^i_2)\varepsilon({y_i}_2) = y_i \otimes T(x^i) = \varepsilon(x^i_2){y_i}_2 \otimes T(x^i_1)\varepsilon({y_i}_1) \\
        & \stackrel{\eqref{eq:Tcond}}{=} T(x^i_2)\varepsilon({y_i}_2) \otimes T(x^i_1)\varepsilon({y_i}_1) = T\Big(\big(x^i\varepsilon({y_i})\big)_2\Big) \otimes T\Big(\big(x^i\varepsilon({y_i})\big)_1\Big) \\
        & \stackrel{\phantom{\eqref{eq:Tcond}}}{=} (T \otimes T)\Delta^\op(x^i\varepsilon(y_i)).
    \end{align*}
    This shows that $T(a_1) \otimes T(a_2) = T(a)_2 \otimes T(a)_1$ for all $a \in \im(p_B)$.
    Furthermore, for $x^i\varepsilon(y_i) \in \im(p_B)$ we have
    \[ \left(x^i\varepsilon(y_i)\right)_1 T\Big(\left(x^i\varepsilon(y_i)\right)_2\Big) = x^i_1T(x^i_2)\varepsilon(y_i) \stackrel{\eqref{eq:Tcond}}{=} x^iy_i \stackrel{(*)}{=} \varepsilon(x^iy_i)1,\]
    where in $(*)$ we used the fact that $x^i_1 \otimes {y_i}_1 \otimes x^i_2{y_i}_2 = x^i \otimes y_i \otimes 1$ implies that $x^iy_i = \varepsilon(x^iy_i)1$.
    Along the same line, for all $u^j\varepsilon(v_j),x^i\varepsilon(y_i) \in \im(p_B)$ we have
    \[
         T\left(u^j\varepsilon(v_j)\right) T\left(x^i\varepsilon(y_i)\right) \stackrel{\eqref{eq:Tcond}}{=} \varepsilon(u^j)v_j\varepsilon(x^i)y_i = \varepsilon(x^iu^j)v_jy_i \stackrel{\eqref{eq:Tcond}}{=} T(x^iu^j)\varepsilon(v_jy_i) = T\left(x^iu^j\varepsilon(v_jy_i)\right)
    \]
    where we used that $B\boxslash B$ is an algebra via \eqref{eq:boxslashalg} and hence $x^iu^j\otimes v_jy_i \in  B\boxslash B$. 
    %Summing up, for $a,b \in \im(p_B)$ we have
    %\[a_1T(a_2)=\varepsilon(a)1, \qquad T(a_1) \otimes T(a_2) = T(a)_2 \otimes T(a)_1, \qquad T(ab) = T(b)T(a). \qedhere\]
\end{proof}
\end{invisible}

\subsection{One sided \texorpdfstring{$n$}{}-Hopf algebras}\label{ssec:nHopf}

By seeking sufficient conditions for the surjectivity of $i_B$ or the injectivity of $p_B$, the following definition arises. Since it also play a central role in the proof of our main result, we exhibit a number of examples of it.

Given bialgebras $B,B'$ and $n\in\mathbb{N}$, we denote by $f^{* n} \colon B\to B'$ the $n$-th convolution power of a linear map $f \colon B\to B'$ in the convolution algebra $\mathrm{Hom}_\Bbbk(B,B')$. In particular this notation applies to $f=\id \colon B\to B$.

\begin{definition}
A \emph{left $n$-Hopf algebra} is a bialgebra $B$ with a minimal $n\in\N$ such that there is an $S\in \mathrm{End}_\Bbbk(B)$, called a \emph{left $n$-antipode}, such that $S*\id^{* n+1}=\id^{* n}$. Similarly one defines a \emph{right $n$-Hopf algebra} and a \emph{right $n$-antipode}.
We speak of \emph{$n$-Hopf algebra} in case it is both a left and a right $n$-Hopf algebra. If the same $S$ is both a left $n$-antipode and a right $n$-antipode, we call it a \emph{two-sided $n$-antipode}.
\end{definition}

\begin{example}\label{exa:nHopf}
%\rd{[Da qui in avanti ho sostituito $G$ con $M$ per indicare i monoidi. Non sto a colorare.]}\ps{[okey-dokey]}
Fix $n\in\N$. Consider the monoid presented by the generator $x$ with relation $x^{n+1}=x^n$ i.e. $M=\langle x\mid x^{n+1}=x^n\rangle$.
Then, the monoid algebra $\Bbbk M$ is a $n$-Hopf algebra. Indeed  one has $\id^{* n+1}=\id^{* n}$ so that both $\id$ and $u \circ \varepsilon$ are two-sided $n$-antipodes. To see this we should check the minimality of $n$. Suppose there is $k<n$ and an  another $S\in \mathrm{End}_\Bbbk(B)$ such that $S*\id^{* k+1}=\id^{* k}$. Evaluation at $x$ yields the equality $S(x)x^{k+1}=x^k$. Write $S(x)=\sum_{t\in\N}h_tx^t$ for some $h_t\in\Bbbk$.  Then  $\sum_{t\in\N}h_tx^{t+k+1}=x^k$. Since $\Bbbk M=\Bbbk[X]/\langle X^{n+1}-X^n\rangle$,
where $\langle X^{n+1}-X^n\rangle$ denotes the two-sided ideal generated by $X^{n+1}-X^n$,
we get $\sum_{t\in\N}h_tX^{t+k+1}-X^k\in \langle X^{n+1}-X^n\rangle$, a contradiction, as $k<n$.
\end{example}

\begin{remark}
\begin{enumerate}[label=\arabic*),leftmargin=*,wide=0pt,itemsep=0.1cm]
\item Clearly a (left, right) $0$-Hopf algebra is just a usual (left, right) Hopf algebra.

\item In an $n$-Hopf algebra, a left and a right $n$-antipode  need not to be equal. For instance, in \cref{exa:nHopf} we can consider $\id$ as a left $n$-antipode and $u \circ \varepsilon$ as a right $n$-antipode.

\item %\begin{invisible}\rd{[modified 10/7/24]}\end{invisible}
Uniqueness of left $n$-antipode is not guaranteed. If $B$ is a left $n$-Hopf algebra and $S,S'$ are two left $n$-antipodes, then $S*\id*S'$ is a left $n$-antipode (by minimality of $n$).
\begin{invisible}
$S*\id*S'*\id^{n+1}
=S*\id*\id^{n}
=S*\id^{n+1}=\id^{n}.$
\end{invisible}
In particular, if $S$ is a left $n$-antipode, then $(S*\id)^{* k}*S$ is a left $n$-antipode for all $k \geq 0$ (to see this, for $k>0$, take $S'=(S*\id)^{* k-1}*S$).
Analogous conclusions hold for right $n$-antipodes. \qedhere
\end{enumerate}
\end{remark}

\begin{example}
%\begin{invisible}\ps{[added 5/7/24]}\end{invisible}
Let $M$ be a commutative regular monoid.
Since it is commutative, for every $x \in M$ there exists $x^\dagger$ such that $x^\dagger \cdot x^2 = x$.
If we consider again the monoid bialgebra $B = \Bbbk M$ and we perform a choice of $x^\dagger$ for every $x \in M$, then $B$ is a $1$-Hopf algebra with $1$-antipode $S$ uniquely determined by $S(x) \coloneqq x^\dagger$ for all $x \in M$.
If $M$ is not a group, then $\Bbbk M$ cannot be a Hopf algebra.
For a concrete example, if the groups $H$ and $G$ in \cref{ex:regmon} are abelian, then $M = H \sqcup G$ from there is a commutative regular (even inverse) monoid.

In general, however, since for every $x \in M$ the element $x^\dagger$ such that $x^\dagger \cdot x^2 = x$ is not necessarily unique, the monoid bialgebra $B = \Bbbk M$ may have many $1$-antipodes determined by the rule $S(x) = x^\dagger$. For example, in \cite[Example 3.1.7]{Wehrung} it is exhibited a commutative inverse monoid $M$ with $0$. In this case, every element $m \in M$ satisfies $0\cdot m\cdot 0 = 0$ and hence for every choice of $S(0)$ we have a different 1-antipode.
\end{example}

\begin{lemma}\label{lem:nthS}
Let $m,n\in \N$. If $B$ is both a left $m$-Hopf algebra and a right $n$-Hopf algebra, then $m=n$ i.e.\ it is an $n$-Hopf algebra.
In particular if  a left $m$-Hopf algebra is a right Hopf algebra, then $m=0$ and so it is a Hopf algebra.
\end{lemma}
\begin{proof}
By assumption, there exist $m,n\in\N$ minimal such that there are $S,T\in \mathrm{End}_\Bbbk(B)$ such that $S*\id^{*m+1}=\id^{*m}$
and $\id^{*n+1}*T=\id^{*n}$. If $m<n$, set $t=n-m>0$ so that $S*\id^{*m+1}=\id^{*m}$ and $\id^{*n+1}*T=\id^{*n}$ imply
\[\id^{*m+t}*T = S*\id^{*m+t+1}*T = S*\id^{*m+t} = \id^{*m+t-1},\]
i.e.\ $\id^{*n}*T=\id^{*n-1}$ against the minimality of $n$. Similarly, $n<m$ leads to a contradiction as well, so $m=n$.
\end{proof}

In order to provide further instances of left $n$-Hopf algebras we use  the notion of  \emph{right perfect} ring, which admits several equivalent characterizations, see e.g. \cite[Theorem 28.4]{AndFull}. The one used here requires the subsistence of the descending chain condition on principal left ideals (there is no mistake on side change).

\begin{proposition}\label{lem:perfect}
Let $B$ be a bialgebra. Let $A$ be a sub-algebra of the convolution algebra $\mathrm{End}_\Bbbk(B)$ such that $\id\in A$. If $A$ is right perfect, then there is a left $n$-antipode $S\in A$ for some $n\in\mathbb{N}$. Analogous statement holds with left and right interchanged.
\end{proposition}
\begin{proof}
The descending chain of left ideals $A*\id\supseteq A*\id^{* 2}\supseteq A*\id^{*3}\supseteq \cdots$ stabilizes so that there is $n\in\N$ such that $A*\id^{*n}\subseteq A*\id^{*n+1}$. Thus $\id^{*n}\in A*\id^{*n+1}$ i.e.\ there is $S\in A\subseteq \mathrm{End}_\Bbbk(B)$ such that $\id^{*n}=S*\id^{*n+1}$.  Choose $n$ minimal.
\end{proof}

The following result provides a wide family of examples of $n$-Hopf algebras.

\begin{corollary}
\label{coro:fdS}
%\begin{invisible}\ps{[Modified 03 03 2025]}\end{invisible}%
 Let $B$ be a bialgebra such that  $\id$ is algebraic over $\Bbbk$ in the convolution algebra $\mathrm{End}_\Bbbk(B)$ (this is true, for instance, for any finite-dimensional bialgebra). Then $B$ has a left $n$-antipode $S$ for some $n\in\mathbb{N}$, which additionally satisfies $S*\id=\id*S$. In particular, it is a two-sided $n$-antipode.
\end{corollary}

\begin{proof}
Since $\id$ is algebraic over $\Bbbk$ in $\mathrm{End}_\Bbbk(B)$, then the sub-algebra $\Bbbk[\id]$ of $\mathrm{End}_\Bbbk(B)$ generated by $\id$ is finite-dimensional, whence left Artinian and a fortiori right perfect.
By \cref{lem:perfect}, there is a minimal $n\in\N$ such that there is $S\in \Bbbk[\id]$ which is a left $n$-antipode i.e.\ such that  $S*\id^{*n+1}=\id^{*n}$. Since $\Bbbk[\id]$ is commutative, we have $S*\id=\id*S$ so that $\id^{*n+1}*S=\id^{*n}$. In view of \cref{lem:nthS}, the positive integer $n$ is minimal also with respect to the latter equality, so $S$ is also a right $n$-antipode.
\end{proof}

As an immediate consequence of \cref{lem:nthS} and \cref{coro:fdS}, we recover that a finite-dimensional left (resp.~right) Hopf algebra is indeed an ordinary Hopf algebra. However the original proof in \cite[Theorem 3]{GNT} is much easier.

\begin{proposition}%[\ps{Modified 05 01 2025}]
\label{prop:semiantip}
Let $B$ be a left $n$-Hopf algebra with left $n$-antipode $S$  for some $n\in\N$ (e.g., $B$ is a finite-dimensional bialgebra). Then:
\begin{enumerate}[label=\alph*),ref={\itshape\alph*)},leftmargin=*]
\item\label{item:semiantip1} For all $y\in B$, we have $S(y)\oslash 1=1\oslash y$, whence $i_B \colon B\to B\oslash B$ is surjective.
\item\label{item:semiantip2} For all $x^i \boxslash {y_i} \in B \boxslash B$ we have $S(x^i)\varepsilon({y_i}) = \varepsilon(x^i){y_i}$, whence $p_B \colon B \boxslash B \to B$ is injective.
\end{enumerate}
\end{proposition}

\begin{proof}
Let us prove \ref{item:semiantip1} first. For $y\in B$, we compute
\begin{align*}
S(y)\oslash 1
 &=S(y_1)y_{2_1}y_{2_2}\cdots y_{2_{n+1}}\oslash  y_{3_{n+1}}\cdots y_{3_{2}}y_{3_{1}}
 =S(y_1)\id^{*n+1}(y_{2}) \oslash  y_{3_{n+1}}\cdots y_{3_{2}}y_{3_{1}}\\
  &=(S*\id^{*n+1})(y_{1})\oslash  y_{2_{n+1}}\cdots y_{2_{2}}y_{2_{1}}
  =\id^{*n}(y_{1}) \oslash  y_{2_{n+1}}\cdots y_{2_{2}}y_{2_{1}}\\
  &=y_{1_1}y_{1_2}\cdots y_{1_{n}} \oslash  y_3y_{2_{n}}\cdots y_{2_{2}}y_{2_{1}}
  =1\oslash y.
\end{align*} We conclude by applying \cref{lem:iBsu}\,\ref{item:epi2}.

To prove \ref{item:semiantip2}, let $x^i \otimes {y_i} \in B \boxslash B$ and recall from \eqref{eq:coinvs} that $x^i \otimes {y_i} \otimes 1 = x^i_1 \otimes {y_i}_1 \otimes x^i_2{y_i}_2$. Then we can observe that
    \[
    x^i \otimes {y_i} \otimes 1 = x^i_1 \otimes {y_i}_1 \otimes x^i_2{y_i}_2 = x^i_1 \otimes {y_i}_1 \otimes x^i_2x^i_3{y_i}_3{y_i}_2 = \cdots = x^i_1 \otimes {y_i}_1 \otimes x^i_2x^i_3\cdots x^i_{k}{y_i}_{k} \cdots {y_i}_3{y_i}_2
    \]
    for any $k \geq 2$, so that
    \begin{gather}
    x^i \otimes {y_i} \otimes 1 = x^i_1 \otimes {y_i}_1 \otimes x^i_2x^i_3\cdots x^i_{n+1}{y_i}_{n+1} \cdots {y_i}_3{y_i}_2 \label{eq:uffa} \\
    x^i \otimes {y_i} \otimes 1 = x^i_1 \otimes {y_i}_1 \otimes x^i_2x^i_3\cdots x^i_{n+2}{y_i}_{n+2} \cdots {y_i}_3{y_i}_2 \label{eq:riuffa}
    \end{gather}
    If now we apply $a \otimes b \otimes c \mapsto S(a)\varepsilon(b)c$ to both sides of \eqref{eq:riuffa}, then we get
    \begin{align*}
    S\left(x^i\right)\varepsilon\left({y_i}\right) & = S\left(x^i_1\right) x^i_2x^i_3\cdots x^i_{n+2}{y_i}_{n+1} \cdots {y_i}_2{y_i}_1 = \left(S*\id^{*n+1}\right)(x^i){y_i}_{n+1} \cdots {y_i}_2{y_i}_1 \\
    & = \id^{*n}(x^i){y_i}_{n+1} \cdots {y_i}_2{y_i}_1 = x^i_1x^i_2\cdots x^i_{n}{y_i}_{n+1} \cdots {y_i}_2{y_i}_1.
    \end{align*}
    If we apply instead $a \otimes b \otimes c \mapsto \varepsilon(a)cb$ to both sides of \eqref{eq:uffa} then we get
    \[x^i_1x^i_2\cdots x^i_{n}{y_i}_{n+1} \cdots {y_i}_2{y_i}_1 = \varepsilon(x^i){y_i}.\]
    Putting together these two pieces of information, we conclude that
    $S\left(x^i\right)\varepsilon\left({y_i}\right) = \varepsilon(x^i){y_i}$
    for all $x^i \boxslash {y_i} \in B \boxslash B$ and the claim follows from \cref{prop:pBin}\,\ref{item:pBin2} with $T = S$.

Finally, if $B$ is finite-dimensional, then it is a left $n$-Hopf algebra (see \cref{coro:fdS}).
%If $B$ is finite-dimensional, \cref{coro:fdS} gives $S$.
\end{proof}

% \rd{[Erano giorni che crecavo di dimostrare l'enunciato seguente. Grande!!! ]} \ps{<3 troppo buono, come sempre ;)}
% \ps{
% \begin{proposition}
% \label{prop:nHopfpBin}
%  Let $B$ be a left $n$-Hopf algebra with left $n$-antipode $S$ for some $n \in \mathbb{N}$ \rd{(e.g., $B$ is a finite-dimensional bialgebra)}. Then
% \end{proposition}

% \begin{proof}

% \end{proof}
% }

\begin{remark}\label{rmk:not-nHopf}
Let $B$ be a bialgebra such that $i_B$ is injective (or $p_B$ is surjective) and which is not a right Hopf algebra. Then, by \cref{prop:Frobenius}, the map $i_B$ is not surjective (or $p_B$ is not injective) and hence, in view of \cref{prop:semiantip},  $B$ is not a left $n$-Hopf algebra for any $n\in\N$.
By the foregoing, the $n$-Hopf algebra of \cref{exa:nHopf} for $n\neq 0$ has $i_B$ surjective but not injective and $p_B$ injective but not surjective.
\end{remark}

\begin{corollary}%[\ps{Modified 06 01 2025}]
\label{cor:Bfd}
    Let $B$ be a finite-dimensional bialgebra. If either $B$ is a sub-bialgebra of a bialgebra $C$ such that $i_C$ is injective, or $B$ is a quotient bialgebra of a bialgebra $C$ whose $p_C$ is surjective, then $B$ is a Hopf algebra.
\end{corollary}

\begin{proof}
Since $B$ is a finite-dimensional bialgebra, the map $i_B$ is surjective and the map $p_B$ is injective by \cref{prop:semiantip}. On the one hand, if $B$ embeds into a bialgebra $C$ such that $i_C$ is injective, then $i_B$ is also injective by \cref{lem:iB-inj}\,\ref{ib-inj_item3}, whence bijective. On the other hand, if $B$ is a quotient of a bialgebra $C$ with $p_C$ surjective, then $p_B$ is surjective by \cref{prop:pB-surj}\,\ref{pB-surj_item3}, whence bijective.
In either case, $B$ is a right Hopf algebra by \cref{prop:Frobenius}. As we already recalled, a finite-dimensional right Hopf algebra is an ordinary Hopf algebra.
\end{proof}

As a consequence of \cref{cor:Bfd}, we recover that a finite-dimensional sub-bialgebra of a Hopf algebra is a Hopf sub-algebra, see \cite[Proposition 7.6.1]{Radford-book}, and that a finite-dimensional quotient bialgebra of a Hopf algebra is a Hopf algebra, cf. \cite[Theorem 1]{Nichols-Quot}.

\begin{remark}
\cref{cor:Bfd} is not necessarily true if we drop the assumption of finite dimension, as there are bialgebras $B$ without antipode that embed into (or are quotients of) a Hopf algebra. % $C$ % e.g., $\Bbbk \N \subseteq \Bbbk \Z$ or Nichols, \emph{Quotients of Hopf algebras}, \S3
% (thus, whose $i_C$ is injective).
\end{remark}

% \begin{invisible}
%     Do we know an example of a bi-ideal in a Hopf algebra which is not a Hopf ideal? I am sure they exist, but for example a theorem of Nichols from "Quotients of Hopf algebras" rules out commutative Hopf algebras, pointed Hopf algebras and cocommutative ones, as in those cases bi-ideals are automatically Hopf.
% \end{invisible}

We conclude the subsection with the following result, which will prove to be genuinely useful in the subsequent sections.

\begin{proposition}
\label{lem:n-Hopfconv}
%\begin{invisible}\ps{[Moved 19 02 2025]}\end{invisible}%
    Let $f \colon B\to B'$ be a bialgebra map which is right  (resp.~left) convolution invertible.
\begin{enumerate}[label=\alph*),ref={\itshape \alph*)},leftmargin=*]
    \item\label{item:fHopf1} If $f$ is surjective and  $B'$ is a left (resp.~right) $n$-Hopf algebra, then $B'$ is a Hopf algebra.
    \item\label{item:fHopf2} If $f$ is injective and  $B$ is a left (resp.~right)  $n$-Hopf algebra, then $B$ is a Hopf algebra.
\end{enumerate}
 \end{proposition}

 \begin{proof}
We just prove \ref{item:fHopf1}, as \ref{item:fHopf2} is similar. Let $S'$ be a left $n$-antipode  for $B'$. Then $S'f*f^{*n+1}=(S'*\id^{*n+1}) f=\id^{*n} f=f^{*n}$ so that $S'f*f^{*n+1}=f^{*n}$. Since $f$ is right convolution invertible, we can cancel $f^{*n}$ on the right and conclude that $S'f$ is a left convolution inverse of $f$. As $f$ is then left and right convolution invertible, $S'f$ has to be a two-sided convolution inverse of $f$.
  Thus $(\id*S') f=f*S'f=u_{B'}\varepsilon_B=u_{B'}\varepsilon_{B'}f$ and hence $\id*S'=u_{B'}\varepsilon_{B'}$, as $f$ is surjective. Similarly $S'*\id=u_{B'}\varepsilon_{B'}$ so $B'$ is a Hopf algebra.
The argument is similar when switching right and left.
 \end{proof}

\begin{corollary}
\label{prop:iBsusub}
%\begin{invisible}\ps{[Modified 19 02 2025]}\end{invisible}%
Let $B$ be a bialgebra. Then $B$ is a Hopf algebra in any of the following cases:
\begin{enumerate}[label=\alph*),ref={\itshape \alph*)},leftmargin=*]
\item\label{item:iBsusub0} $i_B$ is injective and $B$ is a quotient of a Hopf algebra;
\item\label{item:iBsusub1} $i_B$ is surjective and $B$ embeds into a Hopf algebra;
\item\label{item:iBsusub2} $p_B$ is injective and $B$ is a quotient of a Hopf algebra;
\item\label{item:iBsusub3} $p_B$ is surjective and $B$ embeds into a Hopf algebra.
\end{enumerate}
\end{corollary}

\begin{proof}
We just prove \ref{item:iBsusub0} and \ref{item:iBsusub2}, as \ref{item:iBsusub1} and \ref{item:iBsusub3} are similar.

Suppose \ref{item:iBsusub0}. By \cref{lem:iB-su}\,\ref{iBsurj.item4}, we get that $i_B$ is also surjective, whence bijective. Thus $B$ is a right Hopf algebra.
Let $f\colon H \to B$ be the quotient map from a Hopf algebra $H$. Then $f$ has $f \circ S_H$ as a convolution inverse. By \cref{lem:n-Hopfconv}\,\ref{item:fHopf1}, we conclude.

\begin{invisible}
\rem{
Suppose \ref{item:iBsusub1}. By \cref{lem:iB-inj}\,\ref{ib-inj_item3}, we get that $i_B$ is also injective whence bijective. Thus $B$ is a right Hopf algebra, again.
Let $f \colon B\to H$ be the embedding into a Hopf algebra. Then $f$ has $S_H \circ f$ as a convolution inverse and we conclude by \cref{lem:n-Hopfconv}\,\ref{item:fHopf2}.
%Since both $i_B$ and $i_H$ are bijective, by \cref{lem:frightH}, we have that $f$ is a right Hopf algebra map i.e. $fS_B=S_Hf$, where $S_B$ is the right antipode of $B$ and $S_H$ the (right) antipode of $H$. Then $f(S_B*\id_B) =(fS_B)*(f\id_B) =(S_Hf)*(\id_Hf) =(S_H*\id_H) f =u_H\varepsilon_H f=fu_B\varepsilon_B$ and hence $S_B*\id_B=u_B\varepsilon_B$ as $f$ is injective. Thus $S_B$ is an antipode.
}
\end{invisible}

Suppose \ref{item:iBsusub2}. \cref{prop:pB-surj}\,\ref{pB-surj_item3} entails that $p_B$ is also surjective and hence bijective, so that $B$ is a right Hopf algebra. Let $f \colon H\to B$ be the projection from a Hopf algebra. Then $f$ has $f \circ S_H$ as a convolution inverse. By \cref{lem:n-Hopfconv}\;\ref{item:fHopf1}, we conclude.
\begin{invisible}
\rem{
Suppose \ref{item:iBsusub3}. Then \cref{prop:pbInj}\,\ref{item:pbInj4} implies that $p_B$ is also injective and so $B$ is a right Hopf algebra. If $f \colon B\to H$ is the embedding into a Hopf algebra, then $f$ has $S_H \circ f$ as a convolution inverse and we conclude by \cref{lem:n-Hopfconv}\,\ref{item:fHopf2}.
}
\end{invisible}
\end{proof}

\subsection{The right perfect case}\label{ssec:Artin}

From now on, when we write \emph{left Artinian bialgebra} or \emph{left Artinian Hopf algebra}, we mean that the underlying algebra is left Artinian. We specify this since the notion of left Artinian coalgebra exists in the literature. In a similar sense, we will refer to right perfect bialgebras or Hopf algebras.

It was proved in \cite{LiuZhang} that a left Artinian Hopf algebra is necessarily finite-dimensional.
%\ps{\sout{so that \cref{prop:semiantip} applies}}.
Similarly, a monoid algebra $\Bbbk M$ is left Artinian if and only if $M$ is finite, see \cite[Corollary of Theorem 3]{Zelmanov} (note that the statement in loc.~cit.~is in the right Artinian case). Although it is likely known, we could not find an instance of a left Artinian bialgebra which is not finite-dimensional in the literature, so we will provide one in \cref{exa:Radual}, by dualizing an example due to Radford.

Now we are going to prove that several conclusions that we have reached in the finite-dimensional case, remain valid for right perfect bialgebras, but first we need the following lemma.

Originally most of the results in the present subsection where stated only in the Artinian case.  The referee enabled us to extend them to the right perfect case by suggesting improvements to the proofs of \cref{lem:leftArtlocHop} and \cref{lem:Skry}.

\begin{lemma}
\label{lem:leftArtlocHop}
%\rd{[extracted on 2025-01-08]}
Let $B$ be a right perfect (e.g., left Artinian) bialgebra and  $C$ a finite-dimensional sub-coalgebra of $B$ with inclusion $\iota_C \colon C\to B$. Then there are $n\in\mathbb{N}$ and a linear map $\sigma_C\colon C\to B$ such that
$\iota_C^{*n}=\sigma_C*\iota_C^{*n+1}$.
\end{lemma}

\begin{proof}
Let us first show that the convolution algebra $A_C \coloneqq \mathrm{Hom}_\Bbbk(C,B)$ is right perfect, because $C$ is finite-dimensional and $B$ is right perfect.
To this aim, regard $A_C$ as a left $B$-module via the algebra map $\varphi \colon B\to A_C, b\mapsto [c\mapsto b\varepsilon_C(c)]$.
In this way, $A_C \cong B^{\dim(C)}$ as left $B$-modules and therefore any cyclic left ideal of $A_C$ is finitely generated as a left $B$-module: if $\{\alpha_1,\ldots,\alpha_n\}$ is a basis of $A_C$ as a free left $B$-module, then $\alpha * \chi = \sum_{i} \varphi(b_i) * \alpha_i * \chi \in \sum_{i} B(\alpha_i*\chi)$ for every $\alpha \in A_C$.
% Indeed, let $\{c^i,i=1,\cdots,n\}$ be a basis for $C$ and write $\Delta(c^i)=\sum_j c^i_j\otimes c^j$ for some $ c^i_j\in C$. Given $\chi\in A_C$, define $\chi_t^j\in A_C$ on the basis by setting $\chi_t^j(c^i):=\delta_{t,i}\chi(c^j)$. Then, for every $\alpha\in A_C$, we have $\alpha*\chi=\sum_{j,t}\varphi(\alpha(c^t_j))*\chi_t^j$.
% \begin{invisible} Indeed
%  $(\alpha*\chi)(c^i)=\sum_j\alpha(c^i_j) \chi(c^j)
% =\sum_j\alpha(c^i_j) \chi_i^j(c^i)
% =\sum_{j,t}\alpha(c^t_j) \chi_t^j(c^i)$ so that $\alpha*\chi=\sum_{j,t}\alpha(c^t_j) \chi_t^j=\sum_{j,t}\varphi(\alpha(c^t_j))*\chi_t^j$. \end{invisible}
% Therefore the cyclic left ideal $A_C*\chi=\sum_{j,t}B \chi_t^j$ is finitely generated as a left $B$-module.
Now, by a result due to Bj\"{o}rk (see e.g. \cite[page 355]{Lam}), $B$ right perfect is equivalent to the fact that any left $B$-module satisfies the DCC on finitely generated submodules. Thus any descending chain of cyclic left ideals of $A_C$ is stationary and hence $A_C$ is right perfect as claimed.
% Let us first observe that the convolution algebra $A_C \coloneqq \mathrm{Hom}_\Bbbk(C,B)$ is left Artinian, because $C$ is finite-dimensional and $B$ is left Artinian (see the proof of \cite[Theorem 2.2]{SkryabinVO}).
% \begin{invisible}
%  Since $C$ is finite-dimensional, the map $\varphi:B\otimes C^*\to A, b\otimes f\mapsto [c\mapsto bf(c)],$ is an algebra map where the domain is the tensor product of the algebra $A$ by the convolution algebra $C^*=\mathrm{Hom}_\Bbbk(C,\Bbbk)$. Since $C^*$ is finite-dimensional
%   and $B$ a left Artinian ring, we get that $B\otimes C^*$, and hence $A$, is a left Artinian ring. This fact is true for any left Artinian algebra $B$ and finite-dimensional algebra $B'$. Indeed $B'\cong\Bbbk^n$ as a vector space, so that $B\otimes B'\cong B\otimes \Bbbk^n\cong B^n$ as a left $B$-module. Thus $B\otimes B'$ is a left Artinian left $B$-module. This $B$-module structure is given by $a(x\otimes y)=ax\otimes y=(a\otimes1)(x\otimes y)$. Thus any left ideal of $B\otimes B'$ is , in particular, a left $B$-submodule of $B\otimes B'$ and hence $B\otimes B'$ is a left Artinian ring.
% \end{invisible}

Consider the convolution powers $\{\iota_C^{*n} \mid n \in \N\}$ of $\iota_C$ in $A_C$. Since $A_C$ is right perfect, the descending chain of left ideals $A_C*\iota_C\supseteq A_C*\iota_C^{*2}\supseteq A_C*\iota_C^{*3}\supseteq \cdots$ stabilizes so that there is $n\in\N$ such that $A_C*\iota_C^{*n}\subseteq A_C*\iota_C^{*n+1}$. Thus $\iota_C^{*n}\in A_C*\iota_C^{*n+1}$ and hence there exists $\sigma_C\in A_C$ such that $\iota_C^{*n}=\sigma_C*\iota_C^{*n+1}$.
\end{proof}

\begin{remark}
     The condition $\iota_C^{*n}=\sigma_C*\iota_C^{*n+1}$ in \cref{lem:leftArtlocHop} can be regarded as a ``local version'' (note that $n$ can change based on the given sub-coalgebra) of the equality $\id^{*n}=S*\id^{*n+1}$ appearing in the definition of a left $n$-antipode. Clearly, not every left $n$-Hopf algebra is right perfect: $\Bbbk[X] \cong U(\Bbbk X)$ is an example of a left $0$-Hopf algebra which is not right perfect. It would be interesting to know whether a right perfect bialgebra is always a $n$-Hopf for a certain $n$ or not. \cref{lem:leftArtlocHop} suggests that the answer may be in the negative.
\end{remark}

We have the following result obtained by slightly modifying the proofs of \cref{lem:perfect} and \cref{prop:semiantip}.

\begin{proposition}
\label{pro:Artinian}
%\rd{[modified on 2025-01-08]}
%\begin{invisible}\rd{[Added on 18/04/2024]}\end{invisible}
If $B$ is a right perfect bialgebra, then $i_B$ is surjective and $p_B$ is injective.
\end{proposition}

\begin{proof}
Let $y\in B$. By the fundamental theorem of coalgebras (see e.g. \cite[Theorem 2.2.3]{Radford-book}) there is a finite-dimensional sub-coalgebra $C$ of $B$ such that $y\in C$.
Then we can apply \cref{lem:leftArtlocHop} to $C$ and compute
\begin{align*}
\sigma_C(y)\oslash 1
 &=\sigma_C(y_1)y_{2_1}y_{2_2}\cdots y_{2_{n+1}}\oslash  y_{3_{n+1}}\cdots y_{3_{2}}y_{3_{1}}
 =\sigma_C(y_1)\iota_C^{*n+1}(y_{2}) \oslash  y_{3_{n+1}}\cdots y_{3_{2}}y_{3_{1}}\\
  &=(\sigma_C*\iota_C^{*n+1})(y_{1})\oslash  y_{2_{n+1}}\cdots y_{2_{2}}y_{2_{1}}
  =\iota_C^{*n}(y_{1}) \oslash  y_{2_{n+1}}\cdots y_{2_{2}}y_{2_{1}}\\
  &=y_{1_1}y_{1_2}\cdots y_{1_{n}} \oslash  y_3y_{2_{n}}\cdots y_{2_{2}}y_{2_{1}}
  =1\oslash y.
\end{align*}
We conclude that $i_B$ is surjective by applying \cref{lem:iBsu}.

Let now $x^i\otimes {y_i}\in B\boxslash B$. Then, by \cite[Theorem 2.2.3]{Radford-book} again, there is a finite-dimensional sub-coalgebra $C$ of $B$ with $x^i, {y_i}\in C$. As in the proof of \cref{prop:semiantip}, we can write
\begin{align*}    \sigma_C\left(x^i\right)\varepsilon\left({y_i}\right) & = \sigma_C\left(x^i_1\right) x^i_2x^i_3\cdots x^i_{n+2}{y_i}_{n+1} \cdots {y_i}_2{y_i}_1 = \left(\sigma_C*\iota_C^{*n+1}\right)(x^i){y_i}_{n+1} \cdots {y_i}_2{y_i}_1 \\
    & = \iota_C^{*n}(x^i){y_i}_{n+1} \cdots {y_i}_2{y_i}_1 = x^i_1x^i_2\cdots x^i_{n}{y_i}_{n+1} \cdots {y_i}_2{y_i}_1=\varepsilon(x^i) {y_i}.
    \end{align*}
Therefore $\tau(x^i)\varepsilon({y_i})=\varepsilon(x^i) {y_i}$ for $\tau=\sigma_C$ and hence $p_B$ is injective in view of \cref{prop:pBin}.
\end{proof}

We conclude this subsection with a variant of \cref{lem:n-Hopfconv}\,\ref{item:fHopf1} and of \cref{prop:iBsusub}\,\ref{item:iBsusub2} for left Artinian bialgebras (\cref{prop:quoArtisHopf}) and with a ``local'' test for the surjectivity of $i_B$ (\cref{prop:localiB}) that takes advantage of what we discussed so far and which will prove useful in what follows. To this aim, recall (see e.g. \cite[Definition 1.3.30]{Rowen-RT1}) that a ring $R$ is \emph{weakly $n$-finite} if for all matrices $A,B$ in $\mathrm{M}_n(R)$ one has that $AB=1$ implies $BA=1$. Moreover $R$ is \emph{weakly finite} if $R$ is weakly $n$-finite for all $n$.

 \begin{lemma}
 \label{lem:Skry}
 %\begin{invisible}\ps{[Moved 19 02 2025]}\end{invisible}%
   Let $B$ be a right perfect (e.g., left Artinian) bialgebra. Then $B$ is weakly finite. As a consequence, if $C$ is a coalgebra and $f \colon C\to B$ a right (or left) convolution invertible coalgebra map, then $f$ is two-sided convolution invertible. Moreover
   the inclusion $f(C)\hookrightarrow B$ is two-sided convolution invertible.
 \end{lemma}

 \begin{proof} 
Since $B$ is right perfect, then $B/J$ is semisimple, where $J$ is the Jacobson radical of $B$, see \cite[28.4]{AndFull}, i.e., $B$ is a semilocal ring. 
 % According to \cite{Skryabin-proj}, a semilocal ring is a ring $R$ such that $R/J(R)$ is semisimple artinian. Since \cite{Skryabin-proj} cites Rowen's \emph{Ring Theory}, it is possible that the definition of semisimple artinian is coming from there (see page 169 of Rowen). In this case, a semisimple artinian ring is a finite direct product of matrix rings over division rings, whence what we would call a semisimple ring.
 % In fact, a semisimple ring is at the same time left and right Artinian.
 % \cite[15.17]{AndFull} states that $R/J(R)$ is semisimple if and only if it is left artinian.
 By \cite[Proposition 2.2]{Skryabin-proj}, this implies that $B$ is weakly finite. 
%  The first assertion is well-known as a left Artinian  ring is left Noetherian (due to Hopkins)
% %see e.g. Hungerford's book on page 443
% and so weakly finite (see e.g. \cite[Theorem 3.2.37]{Rowen-RT1}).
When $B$ is weakly finite, so is the convolution algebra $\mathrm{Hom}_\Bbbk(C,B)$ by \cite[Lemma 7.1]{Skryabin-proj}. In particular, it is weakly $1$-finite and so the  map $f$ is two-sided convolution invertible.
Now the last part of the statement follows by \cite[Lemma 2.2]{Ery-Skry}.
 \end{proof}

\begin{proposition}
\label{prop:quoArtisHopf}
%\begin{invisible}\ps{[Modified 19 02 2025]}\end{invisible}%
 Let $f \colon B\to B'$ be a bialgebra map which is right (or left) convolution invertible. If $f$ is surjective and $B'$ is right perfect, then $B'$ is a Hopf algebra.

 In particular, a right perfect bialgebra which is a quotient of a Hopf algebra is a Hopf algebra.
\end{proposition}

\begin{proof}
By \cref{lem:Skry}, the inclusion $f(B)\hookrightarrow B'$, i.e. $\id_{B'}$, is two-sided convolution invertible.
% By the proof of  \cite[Proposition 7.6.9]{Radford-book}, in order to conclude that $B'$ is Hopf, it is enough to prove that, for any finite-dimensional sub-coalgebra $C'$ of $B'$, the inclusion $C'\to B'$ is convolution invertible. By \cite[Proposition 6.2.2]{Radford-book} we can further assume that the coalgebra $C'$ is simple.
% Since $f$ is surjective and $C'$ is a simple sub-coalgebra of $B'$,  \cite[Proposition 4.1.7(a)]{Radford-book} ensures that $C'\subseteq f(C)$ for some simple (whence finite-dimensional) sub-coalgebra $C$ of $B$. Thus we just have to prove that the inclusion $\iota:f(C)\to B'$ is convolution invertible.
% As in the proof of \cref{pro:Artinian}, the convolution algebra $A'=\mathrm{Hom}_\Bbbk(f(C),B')$ is left Artinian and we can find an $n\in\mathbb{N}$ and $\sigma \in A'$ such that $\iota^{*n}=\sigma*\iota^{*n+1}$. Define the coalgebra map $g:C\to f(C),c\mapsto f(c).$
% Then $(\iota g)^{*n}=\iota^{*n}g
% =(\sigma g)*(\iota^{*n+1}g)
% =(\sigma g)*(\iota g)^{*n+1}$. Since $\iota g=f_{\mid C}$ is right convolution invertible, we can cancel $(\iota g)^{*n}$ on the right and conclude (as in the proof of \cref{lem:n-Hopfconv}) that $\sigma g$ is a two-sided convolution inverse of $\iota g$.
% Thus $(\iota*\sigma)g
% =(\iota g)*(\sigma g)=u_{B'}\varepsilon_C
% =u_{B'}\varepsilon_{f(C)}g.$ Since $g$ is surjective, we get $\iota*\sigma=u_{B'}\varepsilon_{f(C)}$. Similarly $\sigma*\iota=u_{B'}\varepsilon_{f(C)}$.

Concerning the second claim, if $f \colon H \to B$ is a surjective bialgebra morphism from a Hopf algebra $H$ to a right perfect bialgebra $B$, then $f \circ S_H$ is the convolution inverse of $f$.
\end{proof}

% \rd{[Mi sembra che questa parte "locally finite" spezzi un po' il discorso sulla monoid algebra. Non vedo però quale sia il punto migliore in cui ricollocarla.]}\ps{[Cosa ne dici di subito dopo \cref{pro:Artinian}?]} \rd{[Spezzerebbe l'artinanità. Che ne dici di metterli in fondo a sezione 2? Nel caso, dobbiamo ricordarci di modificare la frase subito dopo \cref{pro:Artinian}. ]}\\
%\ps{[Cosa ne dici di subito dopo \cref{ex:regmon} e promuovere conseguentemente il corollario a propositione? In questo caso dovremmo eliminare il caso \ref{item:localiB3}. C'è solo un punto, in \cref{exa:notArtLoc}, in cui lo usiamo esplicitamente, ma che può essere rimpiazzato da \cref{prop:semiantip} più \ref{item:localiB2}.]}\\
Recall also that an algebra $A$ is said to be \emph{locally finite} (or locally finite-dimensional) if every finitely generated sub-algebra is finite-dimensional (equivalently the Gelfand-Kirillov dimension $\mathrm{GKdim}(A)$ of $A$ is zero, see e.g.~\cite[page 14]{Krause-Lenagan}).

\begin{proposition}%[\ps{moved on the 3/4/25}]
\label{prop:localiB}
    %\begin{invisible}\rd{[Added on 19/04/2024]}\end{invisible}
    Let $B$ be a bialgebra which is generated as an algebra by a subset $\{x_i\mid i\in I\}$. If, for every $i\in I$, one has $1\oslash x_i=x'_i\oslash 1$ for some $x'_i\in B$, then $i_B$ is surjective.
    In particular, this happens in any of the following cases:
    \begin{enumerate}[label={\alph*)},ref={\itshape \alph*)},leftmargin=*]
    %\item\label{item:localiB1} for every $i \in I$, $x_i$ is grouplike and left or right invertible;
    \item\label{item:localiB2} for every $i\in I$, $x_i$ belongs to a sub-bialgebra $B_i$ of $B$ with $i_{B_i}$ surjective;
    \item\label{item:localiB3} for every $i \in I$, $x_i$ belongs to a right perfect (e.g., left Artinian) sub-bialgebra $B_i$ of $B$;
    \item\label{item:localiB4} $B$ is locally finite.
    \end{enumerate}
    %Let $B$ be a bialgebra which is generated as an algebra by a subset $\{x_i\mid i\in I\}$ such that, for every $i\in I$, $x_i\in B_i$ for a sub-bialgebra $B_i$ of $B$ with $i_{B_i}$ surjective. Then $i_B$ is surjective.
\end{proposition}

\begin{proof}
By employing the left $B\otimes B$-module structure of $B\oslash B$ from \cref{lem:oslash}, we have
\[ a\oslash bx_i=(a\otimes b)(1\oslash x_i) = (a\otimes b)(x'_i\oslash 1) = ax'_i\oslash b\]
for every $a,b\in B$ and every $i \in I$. By using this equality, we can move to the left every $x_i$ that appears on the far right of an element in $B\oslash B$. Since $B$ is generated as an algebra by the set $\{x_i\mid i\in I\}$, we can take advantage of \cref{lem:iBsu}\,\ref{item:epi4} to conclude that $i_B$ is surjective.

% Suppose now that we are in case \ref{item:localiB1}. If there exists $y_i \in B$ such that $x_iy_i = 1$, then $1 = \varepsilon(x_iy_i) = \varepsilon(x_i)\varepsilon(y_i)$ and hence
% \[1 \oslash x_i = \left( 1\oslash x_i\right)\varepsilon(x_i)\varepsilon(y_i) = \varepsilon(x_i)y_i \oslash x_iy_i = \varepsilon(x_i)y_i \oslash 1.\]

If we are in case \ref{item:localiB2}, then for every $i \in I$ \cref{lem:iBsu}\,\ref{item:epi2} ensures the existence of $S_i \colon B_i\to B_i$ such that $1\oslash x_i=S_i(x_i)\oslash 1$ in $B_i\oslash B_i$. If we denote by $f_i:B_i\to B$ the canonical inclusion, we get $1\oslash x_i=(f_i\oslash f_i)(1\oslash x_i)=(f_i\oslash f_i)(S_i(x_i)\oslash 1) =S_i(x_i)\oslash 1$ in $B\oslash B$.
%     Denote by $f_i:B_i\to B$ the canonical inclusion. By \cref{lem:iBsu},
%     there is $S_i:B_i\to B_i$ such that $1\oslash x_i=S_i(x_i)\oslash 1$ in $B_i\oslash B_i$.
%     Then, by employing the left $B\otimes B$-module structure of $B\oslash B$ given in \cref{lem:oslash}, for every $a,b\in B$, in $B\oslash B$, we have $a\oslash bx_i=(a\otimes b)(1\oslash x_i)
% =(a\otimes b)(f_i\oslash f_i)(1\oslash x_i)
% =(a\otimes b)(f_i\oslash f_i)(S_i(x_i)\oslash 1)
% =(a\otimes b)(S_i(x_i)\oslash 1)
% =aS_i(x_i)\oslash b$. Thus $a\oslash bx_i=aS_i(x_i)\oslash b$ for every $a,b\in B,i\in I$. By using this equality, we can move to the left every $x_i$ that appears on the far right of an element in $B\oslash B$. Since $B$ is generated as an algebra by the set $\{x_i\mid i\in I\}$, this implies $i_B$ is surjective.

Case \ref{item:localiB3} follows from case \ref{item:localiB2} and \cref{pro:Artinian}.

Let then $B$ be a locally finite bialgebra as in \ref{item:localiB4}. By the fundamental theorem for coalgebras, every $x\in B$ belongs to a finite-dimensional sub-coalgebra $C_x$ of $B$. Let $B_x$ be the sub-algebra of $B$ generated by $C_x$. Since $C_x$ is finite-dimensional, then $B_x$ is finitely generated as an algebra whence finite-dimensional as $B$ is locally finite. Since $B_x$ is indeed a sub-bialgebra of $B$, we conclude.
\end{proof}

\noindent It may be interesting to know if there is an analogue of \cref{prop:localiB} for the injectivity of $p_B$.
%\rd{[Sposterei la frase precedente subito dopo \cref{prop:localiB}.]}\ps{[okey-dokey]}

\section{Free and cofree constructions for finite-dimensional and perfect bialgebras}

\subsection{Hopf envelopes, aka free Hopf algebras}\label{ssec:HopfEnv}

% Given a bialgebra $B$, Manin introduced the \emph{Hopf envelope} of $B$ \cite[\S7]{Manin-book}, which can also be called the \emph{free Hopf algebra generated by} $B$. This is a Hopf algebra $\mathrm{H}\left( B\right) $ together with a bialgebra map $\eta _{B}:B\rightarrow \mathrm{H} \left( B\right) $ such that for any Hopf algebra $H$, and for every bialgebra map $f:B\rightarrow H$ there is a unique Hopf algebra map  $g:\mathrm{H}\left( B\right) \rightarrow H$ such that $g\circ \eta _{B}=f$, see e.g.\ \cite[Theorem 2.6.3]{Pareigis}. The general construction of the Hopf envelope of a bialgebra may be particularly unfriendly, as it involves the coproduct of infinitely many copies of $B$, possibly with opposite multiplication and comultiplication.
In the present section, we give an alternative realisation of the Hopf envelope of a finite-dimensional bialgebra and of an Artinian bialgebra as a quotient of the bialgebra itself.

%In \cref{prop:quotquantfree}, we stumbled upon a finite-dimensional bialgebra $B$ whose canonical map $\eta_B:B\to \mathrm{H}(B)$ into the free Hopf algebra results to be surjective. We now show this is not a coincidence.
Given a bialgebra $B$, denote by $\mathrm{H}(B)$ the free Hopf algebra generated by $B$ and by $\eta_B \colon B \to \mathrm{H}(B)$ the associated universal bialgebra morphism. Recall that, in general, this canonical map $\eta_B$ is just an epimorphism of bialgebras, see \cite[Theorem 3.2 (a)]{Chirv}, which is not necessarily surjective, see \cite[Corollary 3.4 (a)]{Chirv}. Thus, our first aim is to determine sufficient conditions under which $\eta_B$ is surjective. First we need the following characterization.

% \ps{
% \begin{proposition}
% \label{prop:pBetasurj}
%     Let $B$ be a bialgebra for which $p_B$ is surjective. Then, the map $\eta_B \colon B \to \mathrm{H}(B)$ is surjective.
% \end{proposition}

% \begin{proof}
%     Consider the bialgebra $H=\im(\eta_B)$ and denote by $\pi \colon B\to H$ the corestriction of $\eta_B$ to its image and by $\sigma \colon H\to \mathrm{H}(B)$ the canonical inclusion.
%     Since $H$ is a quotient of $B$, by \cref{prop:pB-surj}\,\ref{pB-surj_item3}, we get that $p_H$ is surjective, as so is $p_B$.
%     Since $H$ embeds into the Hopf algebra $\mathrm{H}(B)$, \cref{prop:iBsusub}\;\ref{item:iBsusub3} entails that $H$ is in fact a Hopf algebra.  Then, by the universal property of $\mathrm{H}(B)$, there is a unique Hopf algebra map $\gamma:\mathrm{H}(B)\to H$ such that $\gamma\circ \eta_B=\pi.$ Then $\sigma\circ \gamma\circ \eta_B= \sigma \circ\pi=\eta_B$ and hence $\sigma\circ\gamma=\id_{\mathrm{H}(B)}$ so that $\sigma$ is surjective. Thus $H=\mathrm{H}(B).$
% \end{proof}
% }

\begin{lemma}%[\rd{added 2025/06/26}]
\label{lem:HopfimiB}
For a bialgebra $B$, $\eta_B \colon B\to \mathrm{H}(B)$ is surjective if and only if $\im(\eta_B)$ is a Hopf algebra. 
%\rd{[Se ti piace facciamo anche la versione duale e la applichiamo al caso cocommutativo.]}\ps{ok}
\end{lemma}

\begin{proof}
Consider the bialgebra $H=\im(\eta_B)$ and denote by $\pi \colon B\to H$ the corestriction of $\eta_B$ to its image and by $\sigma \colon H\to \mathrm{H}(B)$ the canonical inclusion.
If $H$ is a Hopf algebra, then there is a unique Hopf algebra map $\gamma \colon \mathrm{H}(B)\to H$ such that $\gamma\circ \eta_B=\pi$ by the universal property of $\mathrm{H}(B)$. Hence $\sigma\circ \gamma\circ \eta_B= \sigma \circ\pi=\eta_B$ and so $\sigma\circ\gamma=\id_{\mathrm{H}(B)}$, so that $\sigma$ is surjective. Thus $H=\mathrm{H}(B).$    
\end{proof}

\begin{proposition}%[\ps{Modified 20 02 2025}]
\label{prop:etaSurj}
Let $B$ be a bialgebra. Then, the map $\eta_B:B\to \mathrm{H}(B)$ is surjective in any of the following cases:
\begin{enumerate}[label=\alph*),ref={\itshape \alph*)}, leftmargin=*]
\item\label{item:etaSurj1} $i_B$ is surjective (e.g., when $B$ is finite-dimensional or, more generally, right perfect);
\item\label{item:etaSurj2} $p_B$ is surjective.
\end{enumerate}
\end{proposition}

\begin{proof}
Consider the bialgebra $H=\im(\eta_B)$.
% and denote by $\pi \colon B\to H$ the corestriction of $\eta_B$ to its image and by $\sigma \colon H\to \mathrm{H}(B)$ the canonical inclusion.
Since $H$ is a quotient of $B$, either \cref{lem:iB-su}\,\ref{iBsurj.item4} entails that $i_H$ is surjective or \cref{prop:pB-surj}\,\ref{pB-surj_item3} entails that $p_H$ is surjective, according to whether we are in case \ref{item:etaSurj1} or \ref{item:etaSurj2}, respectively.
In either case, since $H$ embeds into the Hopf algebra $\mathrm{H}(B)$, \cref{prop:iBsusub} entails that $H$ is in fact a Hopf algebra. We conclude by \cref{lem:HopfimiB}. 
%Then, by the universal property of $\mathrm{H}(B)$, there is a unique Hopf algebra map $\gamma:\mathrm{H}(B)\to H$ such that $\gamma\circ \eta_B=\pi.$ Then $\sigma\circ \gamma\circ \eta_B= \sigma \circ\pi=\eta_B$ and hence $\sigma\circ\gamma=\id_{\mathrm{H}(B)}$ so that $\sigma$ is surjective. Thus $H=\mathrm{H}(B).$
\end{proof}

\begin{corollary}
\label{cor:EnvArtfd}
Let $B$ be a left Artinian bialgebra. Then %$\eta_B$ is surjective and
$\mathrm{H}(B)$ is finite-dimensional.
\end{corollary}

\begin{proof}
    It follows from \cref{pro:Artinian} that $i_B$ is surjective, hence we can apply \cref{prop:etaSurj} to conclude that $\eta_B$ is surjective. Thus, $\mathrm{H}(B)$ becomes a quotient ring of the left Artinian ring $B$ and so it is left Artinian itself and we already noticed that a left Artinian Hopf algebra is finite-dimensional (see \cite{LiuZhang}).
\end{proof}
We could not adapt \cref{cor:EnvArtfd} to the right perfect case since we do not know whether a right perfect Hopf algebra is necessarily finite-dimensional.

 We can now realise $\mathrm{H}(B)$ as an explicit quotient of $B$ in case $B$ is a finite-dimensional  or, more generally, right perfect bialgebra. To this aim, consider the following distinguished quotient of it.

 \begin{definition}\label{def:QB}
     Let $B$ be a bialgebra. Set $\qB \coloneqq B/\ker (i_B)B$ and denote by $\qq \colon B\to \qB$ the canonical projection.
 \end{definition}%

 %\rd{[Vista l'importanza di $Q(B)$ forse si potrebbe definirla prima della prop. Idem nel caso duale.]}

\begin{proposition}%[\rd{Added: 2024-11-19}, \ps{modified 25/11/24}]
\label{prop:pconv}
Let $B$ be a bialgebra.
\begin{enumerate}[label=\alph*),ref={\itshape \alph*)},leftmargin=*]
    \item\label{item:pconv1} $\qB$ is a bialgebra and $\qq \colon B \to \qB$ is a bialgebra map.
    \item\label{item:pconv2} For any bialgebra map $f \colon B\to C$ into a bialgebra $C$ with $i_C$ injective, there is a (necessarily unique) bialgebra map $\hat f \colon \qB\to C$ such that $\hat f \circ \qq = f$.
    \item\label{item:pconv3} If $i_B$ is surjective, then $\qq$ is right convolution invertible.
	\item\label{item:pconv4} If $i_{B^\cop}$ is surjective, then $\ker(i_B)$ is a bi-ideal of $B$ so that $Q(B)=B/\mathrm{ker}(i_B)\cong \mathrm{im}(i_B)$.
\end{enumerate}
 \end{proposition}

 %\rd{[E se definissimo il funtore $\mathrm{E}:\Bialg\to\Bialg$ ponendo $\mathrm{E}(B)\coloneqq B/\ker (i_B)B$ e la trasformazione naturale $\eta:\id\to \mathrm{E}$?
%Is it true that $\mathrm{E}\eta=\eta\mathrm{E}$ like for idempotent comonads and localization functors (A.G. Heinicke)? If $U:\Hopf\to \Bialg$ is the forgetful functor, then $\eta U=U$ and it seems that $U$ is a right $U$-coadjoint of $E$, cf. \href{https://ncatlab.org/nlab/show/relative+adjoint+functor}{ncatlab}. Non so.]}
% \ps{[Parliamone a voce]}
% \rd{[VEDERE SOTTO]}

\begin{proof}
    \ref{item:pconv1} Since, in view of \cref{rem:freecom}, $i_B:B\to B\oslash B$ is a coalgebra map, we have that $K \coloneqq \ker (i_B)$ is a coideal of $B$ and hence we can consider the quotient bialgebra $\qB \coloneqq B/\langle K \rangle$ and the canonical projection $\qq \colon B\to \qB$. Note that $K$ is in fact a left ideal as for all $x\in K$ and for all $b\in B$ we have $i_B(bx)=bx\oslash1=(b\otimes 1)(x\oslash 1)=(b\otimes 1)i_B(x)=0$.
    %\rd{[In precedenza abbiamo indicato con cdot l'azione. Userei sempre la giustapposizione.]}\ps{[okydoky]}
    Thus $\langle K \rangle=KB.$
    
    \ref{item:pconv2} Let $f:B\to C$ be a bialgebra map with $i_C$ injective. By \cref{lem:iB-inj}\,\ref{ib-inj_item2}, we have that $K=\ker (i_B)\subseteq \ker (f)$ so that we get a (necessarily unique) bialgebra map $\hat f \colon \qB \to C$ such that $\hat f \circ \qq = f$.
    
    \ref{item:pconv3} % Since $i_B$ is surjective and $\ker (i_B)=K\subseteq \langle K\rangle=\ker (p)$, there is $\hat{p}:B\oslash B\to B'$ such that $\hat{p}\circ i_B=p.$
    %By surjectivity of $i_B$ again,
    Since $i_B$ is surjective, we can apply \cref{lem:iBsu} to obtain an $S\in \mathrm{End}_\Bbbk(B)$ such that $x\oslash y= xS(y)\oslash 1$, for every $x,y\in B$. 
    Moreover, for every $b \in B$, we have $b_1S(b_2)-\varepsilon_B(b)1\in\ker (i_B) \subseteq \ker (\qq)$ by \cref{cor:Sprops}.
    %Thus, we have $b_1S(b_2)\oslash 1=b_1\oslash b_2=\varepsilon_B(b)1 \oslash 1$ and hence $b_1S(b_2)-\varepsilon_B(b)1\in\ker (i_B)=K\subseteq \langle K\rangle=\ker (\qq)$. 
    Therefore, $( \qq * \qq S)(b) = \qq(b_1)\qq S(b_2)= \qq (b_1S(b_2)) = \varepsilon_B(b)\qq(1) = \varepsilon_B(b)1$.
    %\begin{align*} (p*pS)(b) & = p(b_1)pS(b_2) = p(b_1S(b_2)) = \hat{p}i_B(b_1S(b_2)) = \hat{p}(b_1S(b_2)\oslash 1) \\& = \hat{p}(b_1\oslash b_2) = \varepsilon_B(b)\hat{p}(1 \oslash 1) = \varepsilon_B(b)\hat{p}i_B(1) = \varepsilon_B(b)p(1) = \varepsilon_B(b)1.\end{align*}
    Hence $\qq*\qq S = u_{\qB} \circ \varepsilon_B$, so that $\qq$ is right convolution invertible.
    
    \ref{item:pconv4} By applying \cref{lem:iBsu} to $B^\cop$, we obtain a linear endomorphism $\bar{S}$ of $B$ such that $1\otimes b-\bar{S}(b)\otimes 1\in (B\otimes B)\Delta^\cop(B^+)$. By applying the flip map, we get $b\otimes 1-1\otimes \bar{S}(b)\in (B\otimes B)\Delta(B^+)$ so that $b\oslash 1=1\oslash \bar{S}(b)$ for every $b\in B$. Now, for $b\in B$ and $x\in \ker(i_B)$, we get
    \[xb\oslash 1=(x\otimes 1)(b\oslash 1)=(x\otimes 1)(1\oslash \bar{S}(b))=x\oslash \bar{S}(b)=(1\otimes \bar{S}(b))(x\oslash 1)=(1\otimes \bar{S}(b))i_B(x)=0\]
    and hence $xb\in \ker(i_B)$, i.e., $\ker(i_B)$ is a right ideal. Since we already know it is also a left ideal and a two-sided coideal, we conclude.
\end{proof}

\begin{corollary}%[\ps{moved 22 07 2025}]
%[\rd{added 10/07/2025; da ricollocare?}]
\label{lem:HopfidealiB}
Let $I$ be a  bi-ideal of a bialgebra $B$ such that $B/I$ has an antipode. If $I\subseteq \ker(i_B)$, then $\ker(i_B)=I$ and $\mathrm{H}(B) = Q(B) = B/I \cong \im(i_B).$   
\end{corollary}

\begin{proof}
Since $B/I$ has an antipode, we can apply \cref{lem:iB-inj}\,\ref{ib-inj_item2} to the projection $\pi:B\to B/I$ to conclude that $\mathrm{ker}(i_B) \subseteq \mathrm{ker}(\pi)=I$. Thus, $Q(B)=B/\mathrm{ker}(i_B)B=B/IB=B/I\cong \mathrm{im(i_B)}.$ In particular $Q(B)$ is a Hopf algebra so that, by \cref{prop:pconv}\;\ref{item:pconv2}, we conclude $\mathrm{H}(B)=Q(B)$. 
\end{proof}

Remark that, in view of \cref{prop:pconv}\,\ref{item:pconv2}, when $\eta_B \colon B \to \mathrm{H}(B)$ is surjective, then $\mathrm{H}(B)$ is in fact a quotient of $\qB$.

% \ps{[Mi piace questa faccenda \`a la Kelly. Come ti \`e venuta l'idea? \\
% \rd{[Sono incappato casualmente nel link \href{https://ncatlab.org/nlab/show/transfinite+construction+of+free+algebras}{ncatlab}. :)]}

% Potremmo fare che promuovere 4.1 a sezione e mettere tutto insieme.]}
% \rd{[Fatto!]}

\begin{corollary}%[\rd{Added: 2024-11-19}]
\label{coro:pconv}
Let $B$ be a left $n$-Hopf algebra or a right perfect bialgebra. Then the canonical projection $\qq : B\to \qB$ is a (two-sided) convolution invertible bialgebra map.
\end{corollary}

\begin{proof}
    Either by \cref{prop:semiantip} or by \cref{pro:Artinian}, the map $i_B$ is surjective, so that $\qq$ is a right convolution invertible bialgebra map by \cref{prop:pconv}\,\ref{item:pconv3}. If $B$ is a right perfect bialgebra, then we conclude by \cref{lem:Skry}. If $B$ is a left $n$-Hopf algebra instead, then let $S$ be a left $n$-antipode. 
    Since $\qq$ is a bialgebra map, it induces a morphism of convolution algebras $\mathrm{Hom}_\Bbbk(B,B) \to \mathrm{Hom}_\Bbbk(B,\qB)$. Therefore, the identity $S*\id^{* n+1} = \id^{*n}$ in $\mathrm{Hom}_\Bbbk(B,B)$ implies the identity $\qq S*\qq^{* n+1} = \qq^{*n}$ in $\mathrm{Hom}_\Bbbk(B,\qB)$, from which one deduces that $\qq S*\qq = u_{Q(B)} \circ \varepsilon_B$ because $\qq$ is right convolution invertible (see \cref{prop:pconv}\ref{item:pconv3}\,), so that $\qq$ is also left convolution invertible.
    % For $y\in B$, we compute %\rd{[spostare questo conto in \cref{prop:semiantip}?]}
    % \begin{align*}
    % S(y_1)y_2\oslash 1
    %  &=S(y_1)y_{2_1}y_{2_2}\cdots y_{2_{n+1}}\oslash  y_{3_{n}}\cdots y_{3_{2}}y_{3_{1}}
    %  =S(y_1)\id^{*n+1}(y_{2}) \oslash  y_{3_{n}}\cdots y_{3_{2}}y_{3_{1}}\\
    %   &=(S*\id^{*n+1})(y_{1})\oslash  y_{2_{n}}\cdots y_{2_{2}}y_{2_{1}}
    %   =\id^{*n}(y_{1}) \oslash  y_{2_{n}}\cdots y_{2_{2}}y_{2_{1}}\\
    %   &=y_{1_1}y_{1_2}\cdots y_{1_{n}} \oslash  y_{2_{n}}\cdots y_{2_{2}}y_{2_{1}}
    %   =\varepsilon(y)1\oslash 1.
    % \end{align*}
    % As a consequence
    % $S(y_1)y_2-\varepsilon(y)1\in \ker (i_B)\subseteq \ker (\qq)$ so that
    % $(\qq S*\qq)(b)  = \qq S(b_1)\qq(b_2) = \qq(S(b_1)b_2)
    % = \varepsilon_B(b)\qq(1) = \varepsilon_B(b)1$.
    %  Hence $\qq S*\qq = u_{Q(B)} \circ \varepsilon_B$ so that $\qq$ is also left convolution invertible.
\end{proof}

%\rd{[Il risultato seguente vale anche per $B$ Artiniana? Si potrebbe usare \cref{pro:Artinian} e il fatto che il quoziente di un anello Artiniano e' Artiniano. Se funziona, $\mathrm{H}(B)$ diventa Hopf e Artiniana e dunque di dimensione finita. Resterebbe il dubbio che una bialgebra Artiniana possa essere f.d.]}

\begin{theorem}
\label{prop:HBfd}
Let $B$ be a finite-dimensional bialgebra. Then $\mathrm{H}(B)= \qB$.
\end{theorem}

% \ps{
% \[i_B(xS(b)) = xS(b) \oslash 1 = x \oslahs b = (1 \otimes b)(x \oslash 1) = 0\]

% \[B/K \cong B \oslahs B ~\to~ (a \otimes b) \cdot [x] = [axS(b)] \]

% \[x \in K, i_B(S(x)) = 1 \oslash x\]

% \[S \colon B/\ker(j_B) \to B/\ker(i_B)\]
% }

\begin{proof} Set $K \coloneqq \ker (i_B)$. Since $B$ is finite-dimensional, $\qB = B/KB $ is finite-dimensional, too, and so both of them are left $n$-Hopf algebras by \cref{coro:fdS} (possibly for different $n$'s). Thus, \cref{coro:pconv} entails that $\qq \colon B\to \qB$ is a surjective, convolution invertible, bialgebra map to a left $n$-Hopf algebra and hence, by \cref{lem:n-Hopfconv}\,\ref{item:fHopf1}, $\qB$ is a Hopf algebra.
%
% Set $K \coloneqq \ker (i_B)$. Since $B$ is finite-dimensional, $i_B$ is surjective (cf.\ \cref{prop:semiantip}). Thus, by \cref{prop:pconv}, $\qB \coloneqq B/\langle K \rangle =B/KB $ is a bialgebra and the canonical projection $\qq \colon B\to \qB$ is a surjective, right convolution invertible, bialgebra map.
%
%  Since $\qB$ is a quotient of $B$, it is finite-dimensional too and hence it is a left $n$-Hopf algebra for some $n\in\N$, cf.\ \cref{coro:fdS}. By \cref{lem:n-Hopfconv}\;\ref{item:fHopf1}, we get that $\qB$ is a Hopf algebra.
% %Therefore $S'p*p^{n+1}=(S'*\id^{n+1})\circ p=\id^{n}\circ p=p^n$ so that $S'p*p^{n+1}=p^n$. Since we proved that $p*pS=u_{B'}\varepsilon_B$, we arrive at $S'p
% %=S_1p*p^{n+1}*(pS)^{n+1}
% %=p^n*(pS)^{n+1} =pS$. Thus $(\id*S')\circ p=p*S'p=p*pS=u_{B'}\varepsilon_B=u_{B'}\varepsilon_{B'}p$ and hence $\id*S'=u_{B'}\varepsilon_{B'}$ by surjectivity of $p$. Thus $B'$ is also a right Hopf algebra. By \cref{lem:nthS}, $B'$ is a Hopf algebra.
 Moreover, $\qB$ has the desired universal property in view of \cref{prop:pconv}\,\ref{item:pconv2}.
%We now prove that $B'$ has the desired universal property.
%Let $f:B\to H$ be a bialgebra map into a Hopf algebra $H$. By \cref{lem:iB-inj}\,\ref{ib-inj_item2}, we have that $K=\ker (i_B)\subseteq \ker (f)$ so that we get a (necessarily unique) bialgebra map $f':B'\to H$ such that $f'p=f$.
 \end{proof}

% \ps{
% [We may replace \cref{prop:HBfd} with the following and then convert the finite-dimensional case into a corollary (at least until we manage to prove that a left Artinian bialgebra is fd)] } \\
% \rd{[Io pensavo di mettere la nuova dim in invisible e scrivere che si dimostra in modo analogo al caso finito dimensionale usando però i risultati sull'Artinianità. ]}\\
% \ps{[Quanto alla dim non so ancora: quella del caso di dim finita usa esplicitamente le $n$-Hopf e quindi le "motiva", ma la dim nel caso artiniano è più generale]}

\begin{theorem}
\label{prop:HBArt}
Let $B$ be a right perfect bialgebra. Then $\mathrm{H}(B)= \qB$.
\end{theorem}

\begin{proof}
Set $K \coloneqq \ker (i_B)$. Since $B$ is right perfect $i_B$ is surjective (cf.\ \cref{pro:Artinian}).
Thus, by \cref{prop:pconv}\,\ref{item:pconv1} and  \cref{coro:pconv},
the quotient $\qB = B/KB $ is a bialgebra and  the canonical projection $\qq \colon B\to \qB$ is a surjective, convolution invertible, bialgebra map.

 Since $\qB$ is a quotient of $B$, it is right perfect, too, see e.g. \cite[Corollary 24.19]{Lam}. As a consequence, we have a surjective, convolution invertible, bialgebra morphism $\qq \colon B \to \qB$ to a right perfect bialgebra $\qB$. In light of \cref{prop:quoArtisHopf}, $\qB$ is a Hopf algebra.
 Moreover, $\qB$ has the desired universal property in view of \cref{prop:pconv}\,\ref{item:pconv2}.
%We now prove that $B'$ has the desired universal property.
%Let $f \colon B\to H$ be a bialgebra map into a Hopf algebra $H$. By \cref{lem:iB-inj}\,\ref{ib-inj_item2}, we have that $K=\ker (i_B)\subseteq \ker (f)$ so that we get a (necessarily unique) bialgebra map $f':B'\to H$ such that $f'p=f$.
 \end{proof}

The following result shows that the Hopf envelope of $B$ is in fact $B\oslash B$ in the relevant cases.

\begin{proposition}
\label{prop:HBfdArt}
Let $B$ be a finite-dimensional or, more generally, a right perfect bialgebra. Then $B\oslash B\cong \mathrm{H}(B) = B/\ker(i_B)$ via the isomorphism $\widehat{q_B}:B\oslash B\to Q(B),x\oslash y\mapsto q_B(x)Sq_B(y)$.
\end{proposition}

\begin{proof}
If $B$ is finite-dimensional, so is $B^\cop$. If $B$ is right perfect, so is $B^\cop$, as the underlying algebra is the same. In both cases, we get that both $i_{B}$ and $i_{B^\cop}$ are surjective by \cref{prop:semiantip} and \cref{pro:Artinian}. Now, $\ker(i_B)$ is two-sided ideal of $B$ in view of \cref{prop:pconv}\,\ref{item:pconv4}. Hence, by \cref{prop:HBfd} and \cref{prop:HBArt}, we have that $\mathrm{H}(B) = Q(B) = {B}/{\ker(i_B)} \cong \im(i_B) = B\oslash B$, where the isomorphism $\widetilde{\imath}_B \colon Q(B)\to B\oslash B$ is the unique morphism satisfying $\widetilde{\imath}_B \circ q_B = i_B$.
Since $\widehat{q_B}\circ \widetilde{\imath}_B \circ q_B=\widehat{q_B}\circ i_B=q_B$, we get that $\widehat{q_B}\circ \widetilde{\imath}_B = \id$ and so $\widehat{q_B} = {\widetilde{\imath}_B}^{\,-1}$ is invertible.
\end{proof}

Note that both in the finite-dimensional and in the left Artinian case, $\mathrm{H}(B)$ is finite-dimensional (recall \cref{cor:EnvArtfd}).

Now, we continue by studying $\mathrm{H}(B)$ for a cocommutative bialgebra $B$ with $i_B$ surjective. To this aim we first need the following lemma.

\begin{lemma}%[\ps{added 27 03 2025}]
\label{lem:S(K)}
Let $B$ be a bialgebra. If $i_B \colon B \to B \oslash B$ is surjective and
a resulting endomorphism $S$ of $B$ as in \cref{lem:iBsu}\,\ref{item:epi2} obeys
$S(\ker(i_B)) \subseteq \ker(i_B)$, then $\qB$ is a right Hopf algebra with anti-multiplicative and anti-comultiplicative right antipode $\overline{S}$ given by the equality $\overline{S}\circ q_B=q_B\circ S$.
% In particular, $i_{Q(B)}$ is an isomorphism and $Q^\infty(B) = Q(B)$.
\end{lemma}

\begin{proof}
    Set $K \coloneqq \ker(i_B) \subseteq \ker(q_B)$. Since $i_B$ is surjective,  \cref{lem:iBsu}\,\ref{item:epi2} and \cref{cor:Sprops} entail that there exists $S \colon B \to B$ such that $1 \oslash b = S(b) \oslash 1$ for all $b \in B$ and for which \eqref{eq:Sbar} hold. 
    % By using the left $B \otimes B$-module structure on $B \oslash B$, one can show that
    % \[S(ab) \oslash 1_B = 1_B \oslash ab = (1_B \otimes a)(1_B \oslash b) = (1_B\otimes a)(S(b) \oslash 1_B) = (S(b) \otimes 1_B)(1_B \oslash a) = S(b)S(a) \oslash 1_B\]
    % for all $a,b \in B$. Furthermore,
    % $a_1S(a_2) \oslash 1 = a_1 \oslash a_2 = \varepsilon(a)1_B \oslash 1_B$
    % and
    % \begin{multline*}
    % \left(S(a_2) \oslash 1_B\right) \otimes \left(S(a_1) \oslash 1_B\right) = \left(1_B \oslash a_2\right) \otimes \left(1_B \oslash a_1\right) \stackrel{\eqref{def:oslashcoalg}}{=} \Delta(1_B \oslash a) \\
    % = \Delta(S(a) \oslash 1_B) = \left(S(a)_1 \oslash 1_B\right) \otimes \left(S(a)_2 \oslash 1_B \right)
    % \end{multline*}
    % for all $a \in B$ and hence
    % \begin{equation}\label{eq:Sbar}
    % \begin{aligned}
    % S(ab) - S(b)S(a) & \in K \subseteq KB = \ker(q_B), \\
    % a_1S(a_2) - \varepsilon(a)1_B & \in K \subseteq KB = \ker(q_B), \\
    % S(a_2) \otimes S(a_1) - S(a)_1 \otimes S(a)_2 & \in \ker(i_B \otimes i_B) = K \otimes B + B \otimes K \\
    % & \subseteq KB \otimes B + B \otimes KB = \ker(q_B \otimes q_B).
    % \end{aligned}
    % \end{equation}
Moreover, for all $x \in K$ and $b \in B$ we have
    \[S(xb) + KB = S(b)S(x)+KB = KB,\]
    because $S(x) \in K$ by hypothesis, which is a left ideal. Therefore, $S \colon B \to B$ factors uniquely through a linear map $\overline{S} \colon \qB \to \qB$ such that $\overline{S}\circ q_B=q_B\circ S$, i.e.\ $\overline{S}(b + KB) = S(b) + KB$ for all $b \in B$.
    % Since
    % \[(a_1 + KB)\overline{S}(a_2+KB) = (a_1 + KB)(S(a_2) + KB) = a_1S(a_2) + KB = \varepsilon(a)1_B + KB,\]
    % it follows
    Relations \eqref{eq:Sbar} imply that $\overline{S}$ is an anti-multiplicative and anti-comultiplicative right antipode for $\qB$.
    % The last claim follows from \cref{prop:Frobenius} and \cref{thm:Kellyfree}.
\end{proof}

% \ps{[Credo anche che $\overline{S}$ sia anti-comoltiplicativo, ma devo ancora mettermi a fare il conto] \rd{[Se fosse vero allora $i_{Q(B)}$ sarebbe biettivo e quindi $Q(B)=Q^\infty(B)$. ]}
% }

 \begin{proposition}%[\ps{modified proof 3/4/25}]
 %\rd{[added on 2025/03/26]}
 \label{pro:cocom}
 Let $B$ be a cocommutative bialgebra with $i_B$ surjective. Then $\qB$ is a cocommutative Hopf algebra %\rd{[Grande! non c'ero riuscito a provarlo][\ps{Aspetta a cantar vittoria: adesso devo assicurarmi di non aver seminato una castroneria da qualche parte ;) }]}
 and, moreover, $\mathrm{H}(B)=Q(B)\cong B\oslash B$. Thus, $B\oslash B$ carries a unique bialgebra structure such that $i_B:B\to B\oslash B$ is a bialgebra map and it is, in fact, a Hopf algebra with %\ps{\sout{right}}
 antipode $x\oslash y\mapsto y\oslash x$, for every $x,y\in B$.
 % Moreover, $Q(B)\cong B\oslash B$.
  \end{proposition}

\begin{proof}
By \cref{lem:iBsu}\,\ref{item:epi3}, there is an endomorphism $S$ of $B$ such that $x\oslash y=xS(y)\oslash 1$ for every $x,y\in B$.
% %\ps{and by \cref{prop:pconv}\,\ref{item:pconv3} $\qq \colon B \to \qB$ is right convolution invertible}.
Moreover, since $B$ is cocommutative,
% %
% % \ps{on the one hand $\qq$ is convolution invertible, because $\qB$ is cocommutative as well \rd{[se $Q(B)$ non è commutativa non lo vedo!]}, and on the other hand} \ps{[Sorry, hai ragione, ho fatto confusione. Non volevo usare così la faccenda della cocommutatività di $\qB$. Colpa del fatto che ho usato prima la faccenda del $\qq = \qq S^2$ + cocommutatività per costruire l'altro inverso e poi mi sono reso conto dopo che si poteva accorciare ma l'ho fatto male]}
% %
the endomorphism $\sigma \colon B\oslash B\to B\oslash B,x\oslash y\mapsto y\oslash x$ is well-defined.
% Thus, from $1_B \oslash y = S(y) \oslash 1_B$ we conclude that also $y \oslash 1_B = 1_B \oslash S(y)$ and so, by using the left $B \otimes B$-module structure, that
% $ %\begin{equation}\label{eq:Sswitch}
% xy \oslash 1_B = x \oslash S(y)
% $ %\end{equation}
% for all $x,y \in B$.
% %\rd{[Ok averla riallineato, ma rimane la domanda del referee: che ce ne facciamo di questa formula?] {\color{teal} Ops, hai ragione, non mi sono accorto che non la usavamo più perché era stata rimpiazzata da \cref{prop:pconv}\,\ref{item:pconv4}}}
Set $K\coloneqq \ker(i_B)$.
% \ps{For all $a,b \in B$,
%\[i_B(b) = b \oslash 1_B = \sigma\left(1_B \oslash b\right) = \sigma\left(S(b) \oslash 1_B\right) = 1_B \oslash S(b) = S^2(b) \oslash 1_B = i_B\left(S^2(b)\right),\]
% whence $S^2(b) - b \in K$, and, by using the left $B \otimes B$-module structure, we can conclude that
% \begin{equation}\label{eq:}
% ab \oslash 1_B = aS^2(b) \oslash 1_B = a \oslash S(b)
% \end{equation}
% for all $a,b \in B$.}
Given $x\in K$, we have $i_B(S(x))=S(x)\oslash 1
=1\oslash x
=\sigma (x\oslash 1)=\sigma i_B(x)=0$ so that $S(K)\subseteq K$.
By \cref{lem:S(K)}, $Q(B)$ is a right Hopf algebra with anti-multiplicative and anti-comultiplicative right antipode $\overline{S}$ such that $\overline{S}\circ q_B=q_B\circ S$.
%Since $Q(B)=B/K$, there is a unique morphism $S':Q(B)\to Q(B)$ such that $S'\circ q_B=q_B\circ S.$ For $b\in B$ we have
%$q_B(b)_1S'(q_B(b)_2)
%=q_B(b_1)S'(q_B(b_2))
%=q_B(b_1)q_B(S(b_2))
%=q_B(b_1S(b_2))
%=\varphi^{-1}(b_1\oslash b_2)
%=\varphi^{-1}(1\oslash 1)\varepsilon_B (b)
%=\varphi^{-1}i_B(1)\varepsilon_B (b)
%=q_B(1)\varepsilon_B (b)=1_{Q(B)}\varepsilon_{Q(B)}q_B(b)$. Since $q_B$ is surjective, we deduce that $\id_{Q(B)}*S'=u_{Q(B)}\varepsilon_{Q(B)}$ so that $S'$ is a right antipode for $Q(B)$.
Therefore, since $\qB$ is cocommutative as well, the right-hand analogue of \cite[Theorem 3]{GNT} entails that $\qB$ is a Hopf algebra and, since it enjoys the required universal property in view of \cref{prop:pconv}\,\ref{item:pconv2}, we have $\mathrm{H}(B) = Q(B)$.
% \ps{Therefore, \cref{lem:n-Hopfconv}\,\ref{item:fHopf1} entails that $\qB$ is a Hopf algebra and, since it enjoys the required universal property in view of \cref{prop:pconv}\,\ref{item:pconv2}, we have $\mathrm{H}(B) = Q(B)$.}
Furthermore,  since $B^\cop = B$, \cref{prop:pconv}\,\ref{item:pconv4} entails that
% given $x\in K,b\in B$, we have
% % $i_B(xb)=xb\oslash 1
% % =(x\otimes 1)(b\oslash 1)
% % =(x\otimes 1)\sigma(1\oslash b)
% % =(x\otimes 1)\sigma(S(b)\oslash 1)
% % =(x\otimes 1)(1\oslash S(b))
% % =x\oslash S(b)=(1\otimes S(b))(x\oslash 1)
% % =(1\otimes S(b))i_B(x)=0$
% \[
% i_B(xb) = xb \oslash 1_B \stackrel{\eqref{eq:Sswitch}}{=} x \oslash S(b) = \left(1_B \otimes S(b)\right)i_B(x) = 0
% \]
% so that  $xb\in K$ and hence $K$ is a right ideal and, by definition,
\begin{equation}\label{eq:QBisBoslashB}
Q(B) = B/\ker(i_B) \cong \im(i_B)=B\oslash B.
%Q(B)=B/KB=B/K=B/\ker(i_B)\cong \im(i_B)=B\oslash B.
\end{equation}
% Now, \eqref{eq:Sswitch} implies $S(a_1)a_2 \oslash 1_B = S(a_1) \oslash S(a_2)$, which in turn entails that, in $\qB \oslash \qB$,
% \begin{align*}
% \overline{S}\left(\qq(a_1)\right)\left(\qq(a_2)\right) \oslash 1_{\qB} & = \left(\qq \oslash \qq\right)\left(S(a_1)a_2 \oslash 1_B\right) = \left(\qq \oslash \qq \right)\left(S(a_1) \oslash S(a_2)\right) \\
% & = \qq\left(S(a_1)\right) \oslash \qq\left(S(a_2)\right) \stackrel{\eqref{eq:Sbar}}{=} \qq\left(S(a)_2\right) \oslash \qq\left(S(a)_1\right) \\
% & \stackrel{(*)}{=} \left(\qq \oslash \qq \right)\left(S(a)_1 \oslash S(a)_2\right) = \overline{S}\left(\qq(a)\right)_1 \oslash \overline{S}\left(\qq(a)\right)_2 \\
% & \stackrel{(\star)}{=} \varepsilon_{\qB}(\qq(a))1_{\qB} \oslash 1_{\qB}.
% \end{align*}
% where in $(*)$ we used the cocommutativity of $B$ again and in $(\star)$ the counitality of $\overline{S}$.
% Since $i_{\qB}$ is bijective, we conclude that $\overline{S}(\qq(a)_1)\qq(a)_2 = \varepsilon_{\qB}(\qq(a))1_{\qB}$ and so $\overline{S}$ is, in fact, an antipode.
% %By definition $Q(B)=B/KB=B/K=B/\ker(i_B)\cong \im(i_B)=B\oslash B.$
The isomorphism $\varphi:Q(B)\to B\oslash B,b+K\mapsto b\oslash 1$ resulting from \eqref{eq:QBisBoslashB} satisfies $\varphi\circ q_B=i_B.$  Note that $\varphi^{-1}(x\oslash y)=\varphi^{-1}(xS(y)\oslash 1)=\varphi^{-1}i_B(xS(y))=q_B(xS(y))=xS(y)+K$.
Of course $B\oslash B$ inherits a unique bialgebra structure such that $\varphi$ is a bialgebra isomorphism and hence so that $i_B$ is a bialgebra map as $q_B$ is surjective.

Note that
$\varphi(\overline{S}(q_B(b)))
=\varphi(q_B(S(b)))
=i_B(S(b))
=S(b)\oslash 1
=1\oslash b
=\sigma(b\oslash 1)
=\sigma i_B(b)
=\sigma(\varphi(q_B(b)))$ so that $\varphi \circ \overline{S}=\sigma\circ \varphi$ so that $\sigma$ is the %\ps{\sout{right}}
antipode transported on $B\oslash B$. The multiplication is
\begin{align*}
(x\oslash y)(x'\oslash y')
& =(xS(y)\oslash 1)(x'S(y')\oslash 1)
=i_B(xS(y))i_B(x'S(y'))
=i_B(xS(y)x'S(y')) \\
& =xS(y)x'S(y')\oslash 1
=xS(y)x'\oslash y'
\end{align*}
while the identity is $i_B(1)=1\oslash 1$.
\end{proof}

% \begin{corollary}
% \label{cor:cocom}
%  \rd{[added on 2025/03/26]}
% Let $B$ be a pointed cocommutative bialgebra (e.g. a monoid algebra) with $i_B$ surjective.
% Then $\mathrm{H}(B)=Q(B)\cong B\oslash B$ is a pointed Hopf algebra.
% \end{corollary}

% \begin{proof}
% By \cref{pro:cocom} we have that $Q(B)\cong B\oslash B$ is a right Hopf algebra. Since $Q(B)$ is a quotient bialgebra of $B$, it is pointed as well, see \cite[Corollary 5.3.5]{Montgomery}. By the right-hand version of \cite[Theorem 3]{NT}, we have that $Q(B)$ is in fact a Hopf algebra. In view of its universal property we get $\mathrm{H}(B)=Q(B)$.
% \end{proof}

%\rd{[Il risultato precedente si applica ad esempio in \cref{exa:periodmon} e in altri esempi studiati nell'articolo.]}

% \ps{
% \sout{As an example, note that any finite-dimensional cocommutative bialgebra $B$ has $i_B$ surjective, by} \cref{prop:semiantip}, \sout{and hence $B\oslash B\cong \mathrm{H}(B)$. \rd{However we already know this from}} \cref{prop:HBfdArt}.
% }\medskip

Finally, we conclude this section with the study of $\mathrm{H}(B)$ for a commutative bialgebra $B$, because in this case we can provide an alternative description of it based on the constructions developed so far, even when it is not necessarily a quotient of $B$.

%\marginpar{\tiny \ps{Come mai abbiamo denotato la free commutative con $\mathrm{HP}_c$ e non semplicemente $\mathrm{H}_c$?} \rd{Abbiamo adottato la notazione di  \cite[Definition 66]{Takeuchi}.}}
To this aim, recall that the \emph{free commutative Hopf algebra  generated by a commutative bialgebra} $B$ is a commutative Hopf algebra $\mathrm{HP}_c(B)$ together with a bialgebra map $\eta^c_B \colon B\to \mathrm{HP}_c(B)$ such that for any commutative Hopf algebra $H$ and for every bialgebra map $f \colon B\to H$, there is a unique Hopf algebra map $g \colon \mathrm{HP}_c(B)\to H$ such that $g\circ\eta^c_B=f$, cf. \cite[Definition 66]{Takeuchi}.

\begin{remark}\label{rem:HBcomm}
    Let $B$ be a commutative bialgebra, let $X \coloneqq \mathrm{G}(B)$ be the set of grouplike elements in $B$ and let us consider the localisation $B[X^{-1}]$. In view of \cite[Theorem 65]{Takeuchi}, $B[X^{-1}]$ is a Hopf algebra and it is, in fact, the free commutative Hopf algebra $\mathrm{HP}_c(B)$ generated by the commutative bialgebra $B$. \begin{invisible}In \cite[Theorem 65]{Takeuchi}, only a subset $X$ of $\mathrm{G}(B)$ is involved in the construction of the free commutative Hopf algebra. However, as observed in \cite[Proposition 3.3]{Hayashi} in the more general braided context: if $B[X^{-1}]$ has an antipode, then $B[X^{-1}] \cong B[\mathrm{G}(B)^{-1}]$. \end{invisible}%
    We denote by $\eta_B \colon B \to B[X^{-1}]$ the universal map associated with the localisation, which coincides with the structure map of the free commutative Hopf algebra generated by $B$ and with that of the Hopf envelope, in view of the next result.
\end{remark}

\begin{proposition}\label{prop:HPcB}
 Let $B$ be a commutative bialgebra. Then $\mathrm{H}(B)=\mathrm{HP}_c(B)$.
\end{proposition}

\begin{proof}
For a commutative algebra $A$ and a multiplicatively closed subset $X$ of $A$, the commutative ring of fractions $A[X^{-1}]$ (see, e.g., \cite[Chapter 3]{AtiyahMacdonald}) and the non-commutative one (see, e.g., \cite[Chapter II]{St}) coincide. Therefore, $B[\mathrm{G}(B)^{-1}]$ satisfies the universal property of $\mathrm{H}(B)$ in view of, for instance, \cite[Proposition 3.2]{Hayashi}.
\end{proof}

% \ps{\color{teal} [Personalmente, non mi piace aver rimpiazzato una dimostrazione con un riferimento ad un risultato \emph{senza} dimostrazione. Per\`o, \emph{ubi maior, minor cessat}]} \rd{[Almeno una nostra dim. del risultato la conosciamo.]}

\begin{remark}%[\ps{Added 03 07 2025}]
    It is noteworthy that the statement of \cref{prop:HPcB} already appears in \cite[Remark 4.5]{Porst-subcats}, but supported by an incorrect argument. In fact, the bialgebra coproduct of commutative bialgebras is not necessarily commutative any more (e.g., $\Bbbk[X] \sqcup \Bbbk[Y] = \Bbbk \langle X,Y \rangle$, the free noncommutative algebra generated by $X$ and $Y$). Hence, the argument in \cite{Porst-subcats} stating that $\mathrm{H}(B)$ is commutative because image of the bialgebra $\coprod B_n$, which is commutative as any $B_n$ is commutative, is misleading. Indeed $(\coprod B_n)^\op\cong \coprod (B_n^\op)=\coprod B_n$ but this bialgebra isomorphism does not imply that $\coprod B_n$ is commutative. For instance $\Bbbk \langle X,Y \rangle^\op=(\Bbbk[X] \sqcup \Bbbk[Y])^\op\cong\Bbbk[X]^\op \sqcup \Bbbk[Y]^\op=\Bbbk[X] \sqcup \Bbbk[Y]= \Bbbk \langle X,Y \rangle$ but the latter is not commutative.
\end{remark}

\begin{theorem}\label{thm:intcomm}
Let $B$ be a commutative bialgebra.
Then $B \oslash B \cong \mathrm{HP}_c(B) = \mathrm{H}(B)$, the free (commutative) Hopf algebra generated by $B$, via $\widehat{\eta_B}$.
\end{theorem}

\begin{proof}
    Since $B$ is commutative, it follows from \cref{lem:invosl} that $B \oslash B$ is a commutative bialgebra and that $i_B$ is a morphism of bialgebras. Since for every $x \in X$ we have
    \[(x \oslash 1_B)(1_B \oslash x) = x \oslash x = 1_B \oslash 1_B = (1_B \oslash x)(x \oslash 1_B),\]
    because $x$ is grouplike, there exists a unique bialgebra morphism $\tilde{\imath}_B \colon B[X^{-1}] \to B \oslash B$ extending $i_B \colon B \to B \oslash B$, namely $\tilde{\imath}_B(b/x) = b \oslash x$ for all $b \in B$, $x \in X$. That is, $\tilde{\imath}_B \circ \eta_B = i_B$. 
    % In the opposite direction, \cref{rem:freecom} entails the existence of a coalgebra map $\widehat{\eta_B} \colon B \oslash B \to B[X^{-1}]$ such that $\widehat{\eta_B} \circ i_B = \eta_B$, explicitly given by $\widehat{\eta_B}(a \oslash b) = \eta_B(a)S\eta_B(b)$. Observe that
    % \[\widehat{\eta_B}(a \oslash b)\widehat{\eta_B}(a'\oslash b') = \eta_B(a)S\eta_B(b)\eta_B(a')S\eta_B(b') = \eta_B(aa')S\eta_B(bb') = \widehat{\eta_B}(aa'\oslash bb'),\]
    % for all $a,a',b,b' \in B$, and $\widehat{\eta_B}(1_B \oslash 1_B) = 1_B$, that is to say, $\widehat{\eta_B}$ is a bialgebra map. Now, $\widehat{\eta_B} \circ \tilde{\imath}_B$ is an algebra endomorphism of $B[X^{-1}]$ such that $\widehat{\eta_B} \circ\tilde{\imath}_B \circ \eta_B =\widehat{\eta_B} \circ i_B=\eta_B$ and hence, by the universal property of the localisation, $\widehat{\eta_B} \circ \tilde{\imath}_B = \id_{B[X^{-1}]}$.
    In the opposite direction, \cref{rem:freecom}\,\ref{item:freecom5} entails the existence of a bialgebra map $\widehat{\eta_B} \colon B \oslash B \to B[X^{-1}]$ such that $\widehat{\eta_B} \circ i_B = \eta_B$, explicitly given by $\widehat{\eta_B}(a \oslash b) = \eta_B(a)S\eta_B(b)$. On the one hand, we have
    \[(\widehat{\eta_B} \circ \tilde \imath_B)(b/x) = \widehat{\eta_B}(b \oslash x) = \eta_B(b)S\eta_B(x) = \eta_B(b)\eta_B(x)^{-1} = b/x\]
    for every $b/x \in B[X^{-1}]$ and so $\widehat{\eta_B} \circ \tilde{\imath}_B = \id_{B[X^{-1}]}$. On the other hand, we want to show that $ \tilde{\imath}_B \circ \widehat{\eta_B} = \id_{B \oslash B}$, too. To this aim, remark that %\rd{\sout{the bialgebra maps $\eta_B$ and $\tilde\imath_B$ induce a morphism of convolution algebras $\mathrm{End}_\Bbbk(B[X^{-1}]) \to \mathrm{Hom}_\Bbbk(B,B \oslash B)$ and hence}} 
    $i_B = \tilde\imath_B \circ \eta_B$ is convolution invertible with convolution inverse $\tilde\imath_B \circ S \circ \eta_B$. However, as we already observed in the proof of \cref{lem:invosl}, the morphism $i_B^{\dagger} \colon B \to B \oslash B, b \mapsto 1 \oslash b$, is a right convolution inverse of $i_B$ and so $i_B^{\dagger} = \tilde{\imath}_B \circ S \circ \eta_B$.
    %
     %    Now, consider the morphism $i_B^{\dagger} \colon B \to B \oslash B, b \mapsto 1 \oslash b$. On the one hand, we have
     %    \[(i_B * i_B^{\dagger})(b) = b_1 \oslash b_2 = \varepsilon(b)1 \oslash 1,\]
     %    so that $i_B^{\dagger}$ is right convolution inverse of $i_B$ in $\mathfrak{M}(B,B\oslash B)$. 
     % %\rd{[In alcuni punti $j_B$ è diventato $i_B^\dag.$]}\ps{[Sì, perché mi sono reso conto che abbiamo già usato $j_B$ per l'inclusione di $B \boxslash B$ in $B \otimes B$ e, inoltre, così è più coerente con $p_B^\dagger$. Però è possibile che qui e là me ne sia sfuggito qualcuno]}
     %    On the other hand, we have
     %    \[\left(\tilde{\imath}_B \circ S \circ \eta_B\right) * i_B = \left(\tilde{\imath}_B \circ S \circ \eta_B\right) * \left(\tilde{\imath}_B \circ \eta_B\right) = \tilde{\imath}_B \circ (S * \id_{B[X^{-1}]}) \circ \eta_B = u_{B \oslash B} \circ \varepsilon_B\]
     %    so that $i_B$ is convolution invertible and we conclude, therefore, that $i_B^{\dagger} = \tilde{\imath}_B \circ S \circ \eta_B$. 
    Now, for every $a,b\in B$, we have
    \[\left(\tilde{\imath}_B \circ \widehat{\eta_B}\right)(a \oslash b)=\tilde{\imath}_B(\eta_B(a)S\eta_B(b))    =\tilde{\imath}_B\eta_B(a)\tilde{\imath}_BS\eta_B(b)=i_B(a)i_B^{\dagger}(b) = a \oslash  b,\] so that $\tilde{\imath}_B\circ\widehat{\eta_B}=\id_{B\oslash B}.$ Hence $\tilde{\imath}_B$  and $\widehat{\eta_B}$ are mutual inverses and $B \oslash B \cong B[X^{-1}] = \mathrm{H}(B)$.
    % via $\widehat{\eta_B}$.
\end{proof}

\subsection{The dual setting: cofree Hopf algebras}

%\ps{[Added 18/11/2024]}
For the sake of completeness, we address the dual construction of the cofree Hopf algebra as well. For a given bialgebra $B$, the \emph{cofree Hopf algebra} on $B$ is a Hopf algebra $\mathrm{C}(B)$ together with a bialgebra morphism $\epsilon_B \colon \mathrm{C}(B) \to B$ such that for every Hopf algebra $H$ and any bialgebra morphism $f \colon H \to B$, there exists a unique Hopf algebra morphism $\tilde f \colon H \to \mathrm{C}(B)$ such that $\epsilon_B \circ \tilde f = f$. We know, for categorical reasons \cite{Agore}, that it exists. We can dualise many of the results of the previous section in order to obtain a realisation of the cofree Hopf algebra on $B$ as a sub-bialgebra of $B$ itself.

As before, notice that the canonical map $\epsilon_B \colon \mathrm{C}(B)\to B$ is not injective in general, in view of \cite[Corollary 3.4(a)]{Chirv} again. Therefore, we begin by determining conditions under which it is. To this aim, let us state without proof the duals of \cref{lem:HopfimiB} and \cref{prop:etaSurj}.

\begin{lemma}%[\ps{added 27 06 25}]
\label{lem:HopfimEB}
For a bialgebra $B$, $\epsilon_B \colon \mathrm{C}(B) \to B$ is injective if and only if $\im(\epsilon_B)$ is a Hopf algebra.
\end{lemma}

\begin{invisible}
\begin{proof}
Consider the bialgebra $H \coloneqq \im(\epsilon_B)$ and denote by $\pi \colon \mathrm{C}(B)\to H$ the corestriction of $\epsilon_B$ to its image and by $\sigma \colon H\to B$ the canonical inclusion. If $H$ is a Hopf algebra then, by the universal property of $\mathrm{C}(B)$, there is a unique Hopf algebra map $\gamma:H\to\mathrm{C}(B)$ such that $\epsilon_B\circ \gamma=\sigma.$ Then $\epsilon_B\circ \gamma\circ \pi= \sigma \circ\pi=\epsilon_B$ and hence $\gamma\circ \pi=\id_{\mathrm{C}(B)}$ so that $\pi$ is injective. Thus $\epsilon_B=\sigma \circ\pi$ is injective too.
\end{proof}
\end{invisible}

\begin{proposition}%[\ps{Modified 20 02 2025}]
\label{prop:epsiInj}
Let $B$ be a bialgebra. Then, the map $\epsilon_B:\mathrm{C}(B)\to B$ is injective in any of the following cases:
\begin{enumerate}[label=\alph*),ref={\itshape \alph*)},leftmargin=*]
    \item\label{item:epsiInj1} $i_B$ is injective;
    \item\label{item:epsiInj2} $p_B$ is injective (e.g., when $B$ is finite-dimensional or, more generally, right perfect).
\end{enumerate}
\end{proposition}

\begin{invisible}
\begin{proof}
Consider the bialgebra $H \coloneqq \im(\epsilon_B)$. 
%and denote by $\pi \colon \mathrm{C}(B)\to H$ the corestriction of $\epsilon_B$ to its image and by $\sigma \colon H\to B$ the canonical inclusion.
Since $H$ embeds into $B$, either \cref{lem:iB-inj}\,\ref{ib-inj_item3} entails that $i_H$ is injective or \cref{prop:pbInj}\,\ref{item:pbInj4} entails that $p_H$ is injective, according to whether we are in case \ref{item:epsiInj1} or \ref{item:epsiInj2}, respectively.
In either case, since $H$ is a quotient of the Hopf algebra $\mathrm{C}(B)$, \cref{prop:iBsusub} entails that $H$ is in fact a Hopf algebra.
% Since $H$ is a quotient of the Hopf algebra $\mathrm{C}(B)$, from $\pi\circ p_{\mathrm{C}(B)}=p_H\circ (\pi\boxslash \pi)$ we deduce that $p_H$ is surjective. On the other hand, from $\sigma\circ p_H=p_B\circ (\sigma\boxslash\sigma)$ and the injectivity of $p_B$, we get that $p_H$ is also injective whence bijective. By \cref{prop:Frobenius}, $B$ is a right Hopf algebra.
% Since $\mathrm{C}(B)$ is a Hopf algebra $\mathrm{C}(B)$, then $\pi:\mathrm{C}(B)\to H$ is convolution invertible and hence \cref{lem:n-Hopfconv}\;\ref{item:fHopf1} entails that $H$ is in fact a Hopf algebra.
We conclude by \cref{lem:HopfimEB}.
%Then, by the universal property of $\mathrm{C}(B)$, there is a unique Hopf algebra map $\gamma:H\to\mathrm{C}(B)$ such that $\epsilon_B\circ \gamma=\sigma.$ Then $\epsilon_B\circ \gamma\circ \pi= \sigma \circ\pi=\epsilon_B$ and hence $\gamma\circ \pi=\id_{\mathrm{C}(B)}$ so that $\pi$ is injective. Thus $\epsilon_B=\sigma \circ\pi$ is injective too.
\end{proof}
\end{invisible}

The following partially answers the question left open at the end of \cite{Agore}.

 \begin{corollary}\label{cor:CBsub}
     Let $B$ be a bialgebra whose $i_B$ or $p_B$ is injective. Then $\mathrm{C}(B)$ is the largest sub-bialgebra of $B$ that admits an antipode.
 \end{corollary}

%\rd{[modified on 2025-01-08]}
\cref{prop:semiantip} and \cref{pro:Artinian} together with \cref{prop:epsiInj}, entail that $\mathrm{C}(B)$ is a sub-bialgebra of $B$ when the latter bialgebra is finite-dimensional or, more generally, right perfect. This also follows from \cite[Theorem 2.3]{Ery-Skry}, as right perfect implies weakly finite (see \cref{lem:Skry}). In \cref{exa:notweakpBinj} we will provide an instance of a bialgebra $B$ with $p_B$ injective which is not weakly finite.
In the finite-dimensional case, we are going to obtain a more explicit description of $\mathrm{C}(B)$. To this aim, we now provide a dual version of the bialgebra Q(B) and of \cref{prop:pconv}.

\begin{definition}
\label{def:KB}
    Let $B$ be a bialgebra. Set $K(B) \coloneqq \{b \in B \mid \Delta(b) \in \im(p_B) \otimes B\}$ and denote by $k_B \colon K(B)\to B$ the canonical injection.
\end{definition}%

\begin{lemma}%[\ps{Extracted 04 03 2025}]
\label{lem:KBcoalg}
For $B$ a bialgebra, $K(B)$ is a sub-coalgebra. Namely, it is the sum of all sub-coalgebras of $B$ contained in $\im(p_B)$.
\end{lemma}

\begin{proof}
Given a subspace $V$ of a coalgebra $C$, then one easily checks that $V'\coloneqq \{c\in C\mid \Delta(c)\in V\otimes C\}$ is the largest right coideal of $C$ contained in $V$. If, moreover, $V$ is a left coideal, then so is $V'$ and then it is the largest subcoalgebra of $C$ contained in $V$. We conclude by applying this general principle to the case $C=B$ and $V=\im(p_B)$, as $V'=K(B)$. Indeed $\im(p_B)$ is a left coideal as $p_B:B \boxslash B\to B$ is a morphism of left $B$-comodule where the domain is a left $B$-comodule via $\Delta$ while the domain via the structure mentioned at the beginning of the proof of \cref{prop:pBin}. %[Il referee suggeriva di menzionare questa proprietà di $p_B$ e similmente per $i_B$ già nella sezione 2.1.]
\end{proof}

\begin{proposition}
\label{prop:pconvd}
%\rd{[Added on 2024-12-03]}
   Let $B$ be a bialgebra.
   \begin{enumerate}[label=\alph*),ref={\itshape \alph*)},leftmargin=*]
    \item\label{item:pconvd1} $K(B)$ is a sub-bialgebra of $B$, i.e., $k_B \colon K(B) \to B$ is a bialgebra map.
    \item\label{item:pconvd2} If $f \colon A\to B$ is a morphism of bialgebras from a bialgebra $A$ with $p_A$ surjective, then one has $\im(f)\subseteq K(B)$, i.e., there exists a unique bialgebra map $\tilde f \colon A \to K(B)$ such that $f = k_B \circ \tilde f$.
    \item\label{item:pconvd3} If $p_B:B \boxslash B\to B$ is injective, then $k_B$ is right convolution invertible.
    \item\label{item:pconvd4} If $p_{B^\op}$ is injective, then $\im(p_B)$ is a sub-coalgebra of $B$ so that $K(B)=\im(p_B)$.
\end{enumerate}
\end{proposition}

\begin{proof}
We only prove \ref{item:pconvd3} and \ref{item:pconvd4}, for the convenience of the reader.
\begin{invisible}
\rem{
\ref{item:pconvd1} We already know that $K(B)$ is a sub-coalgebra, in view of \cref{lem:KBcoalg}. It remains to check that it is also a sub-algebra. Note that $B \boxslash B$ is a sub-algebra of $B\otimes B^{\mathrm{op}}$ and that $p_B:B \boxslash B\to B$ is an algebra map. Thus $\im(p_B)$ is a sub-algebra of $B$ and hence also $K(B)$ is a sub-algebra of $B$.
%It remains to check that $K(B)$ is a sub-coalgebra.

\ref{item:pconvd2} By \cref{prop:pB-surj}\,\ref{pB-surj_item2}, $\im(f) \subseteq \im(p_B).$ Thus, for every $a\in A$, we have $\Delta (f(a))=f(a_1)\otimes f(a_2)\in \im(f)\otimes B\subseteq \im(p_B)\otimes B$ that is $f(a)\in K(B)$.
}
\end{invisible}

\ref{item:pconvd3} Assume $p_B$ injective. \cref{cor:Tantis} entails that
% and let $x^i\otimes {y_i} \in B \boxslash B$. Then, via \cref{lem:oslash} and \cref{prop:pBin}\,\ref{item:pBin3}, we get
% \[(\id_B*T)p_B(x^i \otimes {y_i}) = x^i_1T(x^i_2)\varepsilon({y_i}) = x^i{y_i} 
% %= \varepsilon(x^i_1)\varepsilon({y_i}_1)x_2^i{y_i}_2 \stackrel{\eqref{eq:coinvs}}{=}
% \varepsilon(x^i)\varepsilon({y_i})1=u_B\varepsilon_B p_B(x^i\otimes {y_i}).\]
% Thus 
$(\id_B*T-u_B\circ\varepsilon_B)\circ p_B=0.$
% If we denote by $q \colon B\to \overline{B}\coloneqq B/\im(p_B)$ the cokernel of $p_B$, then we get a unique map $h:\overline{B}\to B$ such that $h\circ q=\id*T-u\circ\varepsilon$.
% Now, for $b'\in K(B)$, we have that $b'=b'_1\varepsilon(b'_2)\in \im(p_B)\varepsilon(B)\subseteq \im(p_B)\subseteq \ker(q).$
% As a consequence
% $(\id*T-u\circ \varepsilon)(b')=hq(b')=0$
Thus, since $K(B) \subseteq \im(p_B)$, for every $b'\in K(B)$ we have $(\id_B*T - u_B\circ \varepsilon_B)(b') = 0$, i.e.\ $b'_1T(b'_2) = \varepsilon_B(b')1$.
%for every $b'\in K(B).$
This means that the canonical injection $k_B:K(B)\to B$ is right convolution invertible with right inverse $T\circ k_B$.

\ref{item:pconvd4} For every $x^i\otimes {y_i} \in B \boxslash B$ one has $y_i\otimes x^i \in B^\op \boxslash B^\op$. Thus, by \cref{prop:pBin} applied to $B^\op$, there is a linear endomorphism $\bar T$ of $B$ such that $\varepsilon(x^i)\bar{T}(y_i)=x^i\varepsilon(y_i).$ Thus, by employing the left $B\otimes B^\op$-coaction of $B \boxslash B$, for every $x^i\otimes {y_i} \in B \boxslash B$ we get
\begin{align*}
\Delta^\cop(p_B(x^i\otimes {y_i})) & = x^i_2\varepsilon(y_i)\otimes x^i_1 \stackrel{(\star)}{=} \varepsilon(x^i_2)\bar{T}(y_i)\otimes x^i_1 = \bar{T}(y_i)\otimes x^i = \varepsilon(x^i_1)\bar{T}({y_i}_1)\otimes x^i_2\varepsilon\left({y_i}_2\right) \\
& =  (\varepsilon\otimes\bar{T})((x^i\otimes {y_i})_{-1})\otimes p_B((x^i\otimes {y_i})_{0}) \in B\otimes \im(p_B)
\end{align*}
%\begin{align*}
%B\otimes \im(p_B)&\ni (\varepsilon\otimes\bar{T})((x^i\otimes {y_i})_{-1})\otimes p_B((x^i\otimes {y_i})_{0})=\varepsilon(x^i_1)\bar{T}({y_i}_1)\otimes p_B(x^i_2\otimes {y_i}_2)=\bar{T}(y_i)\otimes x^i \\
%& = \varepsilon(x^i_2)\bar{T}(y_i)\otimes x^i_1 \stackrel{(\star)}{=} x^i_2\varepsilon(y_i)\otimes x^i_1=\Delta^\cop(p_B(x^i\otimes {y_i})).
%\end{align*} 
where in $(\star)$ we used that $x^i_1\otimes x^i_2\otimes y_i\in B\otimes (B \boxslash B)$. 
Thus $\Delta(\im(p_B))\subseteq \im(p_B)\otimes B$. Since we already know that  $\Delta(\im(p_B))\subseteq B\otimes\im(p_B) $, we conclude that $\Delta(\im(p_B))\subseteq \im(p_B)\otimes\im(p_B) $ and hence it is a sub-coalgebra. By \cref{lem:KBcoalg}, we get that $K(B) = \im(p_B)$.
\end{proof}

In \cref{lem:KBcoalg} we claimed that $K(B)$ is the sum of all sub-coalgebras of $B$ contained in the algebra $\im(p_B)$ (whence its largest sub-coalgebra). The fact that this sum is a sub-bialgebra could have also being deduced from \cite[Lemma 6]{Porst}.

\begin{corollary}%[\ps{Added 27 02 2025}]
\label{cor:kconv}
Let $B$ be a left $n$-Hopf algebra or a right perfect bialgebra. Then the canonical inclusion $k_B : K(B)\to B$ is a (two-sided) convolution invertible bialgebra map.
\end{corollary}

\begin{proof}
Either by \cref{prop:semiantip}\,\ref{item:semiantip2} or by \cref{pro:Artinian}, $p_B$ is injective and hence $k_B$ is right convolution invertible by \cref{prop:pconvd}\,\ref{item:pconvd3}.
If $B$ is right perfect, then we conclude by \cref{lem:Skry}. If instead it is a left $n$-Hopf algebra, then let us prove that $k_B$ is left convolution invertible as well. If $S$ is a left $n$-antipode for $B$, then we get $Sk_B*k_B^{*n+1} = (S*\id^{*n+1})\circ k_B =\id^{*n} \circ k_B = k_B^{*n}$. Since $k_B$ is right convolution invertible, we may cancel $k_B^{*n}$ on the right to arrive at $Sk_B*k_B = u_B \circ \varepsilon_{K(B)}$, that is $k_B$ is also left convolution invertible.
\end{proof}

We are now ready to provide the dual version of \cref{prop:HBfd} and of the finite-dimensional case of \cref{prop:HBfdArt}.

\begin{theorem}
%\rd{[Added on 2024-12-04]}
\label{prop:CBfd}
Let $B$ be a finite-dimensional bialgebra. Then $\mathrm{C}(B)= K(B) = \im(p_B) \cong B \boxslash B$.
\end{theorem}

\begin{proof} Since $B$ is finite-dimensional, $p_B$ and $p_{B^{\mathrm{op}}}$ are injective (cf.\ \cref{prop:semiantip}). Thus, by \cref{prop:pconvd}, $K(B)= \im(p_B) \cong B \boxslash B$ is a sub-bialgebra of $B$ and the canonical injection $k_B:K(B)\to B$ is a right convolution invertible, bialgebra map.

 Since $K(B)$ is a subspace of $B$, it is finite-dimensional too and hence it is a left $n$-Hopf algebra for some $n\in\N$, cf.\ \cref{coro:fdS}.
 By \cref{lem:n-Hopfconv}\,\ref{item:fHopf2}, we get that $K(B)$ is a Hopf algebra.
 Moreover, $K(B)$ has the desired universal property in view of \cref{prop:pconvd}\,\ref{item:pconvd2}.
 \end{proof}

 \begin{remark}%[\ps{[Added 26 06 25]}]
 Concerning the right perfect case in \cref{prop:HBfdArt} instead, we can observe what follows. If $B$ is left perfect, then $B^\op$ is  right perfect so that $p_{B^\op}$ is injective by \cref{pro:Artinian}. Hence, $K(B)=\im(p_B)$ by \cref{prop:pconvd}. Therefore, if $B$ is perfect (i.e., both left and right perfect), then also $p_B \colon B \boxslash B\to B$ is injective so that it induces the bijection $B \boxslash B\to K(B),\,x^i\otimes y_i\mapsto x^i\varepsilon(y_i)$.
% Now, if $B$ is finite-dimensional than so is $B^\op$ and hence $p_{B^\op}$ is injective by \cref{prop:semiantip}. If $B$ is right Artinian, then $B^\op$ is left Artinian so that $p_{B^\op}$ is injective by \cref{pro:Artinian} again. In both cases $K(B)=\im(p_B)$. 
% Now, if $B$ is finite-dimensional or, more generally, (left and right) Artinian, then also $p_B \colon B \boxslash B\to B$ is injective so that it induces the bijection $B \boxslash B\to K(B),\,x^i\otimes y_i\mapsto x^i\varepsilon(y_i)$. \medskip
 \end{remark}

 Although we are not able to prove that $\mathrm{C}(B)= K(B)$ in case $B$ is a right perfect bialgebra, still in \cref{prop:CBleftArt} we will provide an iterated construction of $\mathrm{C}(B)$ in this latter case.

\begin{remark}%\rd{[added on 2025-01-07]}
\label{rem:suminb}
Let $B$ be a bialgebra. if $C$ is any sub-coalgebra of $B$ such that the inclusion $\iota:C\to B$, is two-sided convolution invertible, then, by \cite[Proposition 6.1.3]{Radford-book}, its convolution inverse, say $\sigma:C\to B$ is an anti-coalgebra map
and hence we can define $\sigma':C\to B\boxslash B,c\mapsto c_1\boxslash\sigma (c_2).$
Then $p_B\sigma'(c)=c_1\varepsilon_B\sigma (c_2) = c_1\varepsilon_C (c_2)=c$ for every $c\in C$ and hence $C\subseteq \im(p_B)$. Since, by \cref{lem:KBcoalg}, $K(B)$ is the sum of all sub-coalgebras of $B$ contained in $\im(p_B)$, we deduce that $C\subseteq K(B)$.  In particular, when $\epsilon_B$ is injective we have $\mathrm{C}(B) \subseteq K(B)$.

Assume now that $B$ is finite-dimensional or, more generally, right perfect. In this case, $p_B$ is injective and hence $\mathrm{C}(B) \subseteq K(B)$, because $\epsilon_B$ is injective (see \cref{prop:epsiInj}), and the canonical injection $k_B:K(B)\to B$ is right convolution invertible by \cref{prop:pconvd}\,\ref{item:pconvd3}. By \cref{lem:Skry}, $k_B$ is two-sided convolution invertible.  Summing up, when $B$ is finite-dimensional or, more generally, right perfect, then $K(B)$ is the largest sub-coalgebra of $B$ whose inclusion in $B$ is two-sided convolution invertible.
This should be compared with the proof of \cite[Theorem 2.3]{Ery-Skry} where $\mathrm{C}(B)$ results to be the the largest sub-coalgebra ``invertible in $B$'', i.e.\ for which the convolution inverse of the inclusion takes values in the sub-coalgebra itself.
\end{remark}

As we did in \cref{ssec:HopfEnv}, we continue this section by studying the particular case of $\mathrm{C}(B)$ for a commutative bialgebra $B$ with $p_B$ injective. To this aim, we begin with a counterpart of \cref{lem:S(K)}.

% \rd{[Sapevo che non avresti resistito alla tentazione del duale. :)]} \ps{[ :D me l'ero scritta sul file delle idee come da giostrare in futuro/rifilare a qualche studentello sbarbatello, ma poi ho finito per fare tutti i conti perché non mi tornavano 0:) ]} \rd{:)}

\begin{lemma}\label{lem:Tim}
Let $B$ be a bialgebra. If $p_B \colon B \boxslash B \to B$ is injective and a resulting endomorphism $T$ of $B$ as in \cref{prop:pBin} obeys $T\left(\im(p_B)\right) \subseteq \im(p_B)$, then $K(B)$ is a right Hopf algebra with anti-multiplicative and anti-comultiplicative right antipode $\overline{T}$ uniquely determined by the equality $k_B \circ \overline{T} = T \circ k_B$.
\end{lemma}

\begin{proof}
Since $p_B$ is injective, \cref{prop:pBin} and \cref{cor:Tantis} entail that there is a linear endomorphism $T$ of $B$ such that 
$
%\begin{equation}\label{eq:Tcond}
T\left(x^i\right)\varepsilon\left(y_i\right) = \varepsilon\left(x^i\right)y_i
%\end{equation}
$ 
for all $x^i \boxslash y_i \in B \boxslash B$ and which satisfies \eqref{eq:Tprops}. 
    Now, suppose that $T\left(\im(p_B)\right) \subseteq \im(p_B)$. For every $a \in K(B)$, we have $a \in \im(p_B)$. %and hence $b = x^i \varepsilon\left(y_i\right)$.
    Moreover, $a_1 \otimes a_2 \in B \otimes \im(p_B)$
    %$x^i_1 \otimes x^i_2\varepsilon(y_i) \in B \otimes \im(p_B)$
    because $\im(p_B)$ is a left coideal, so that
    \[\Delta(T(a)) = T(a)_1 \otimes T(a)_2 = T\left(a_2\right) \otimes T\left(a_1 \right) \in T(\im(p_B)) \otimes T(B) \subseteq \im(p_B) \otimes B.\]
    % \[\Delta(T(b)) = T\left(x^i \varepsilon\left(y_i\right)\right)_1 \otimes T\left(x^i \varepsilon\left(y_i\right)\right)_2 = T\left(x^i_2 \varepsilon\left(y_i\right)\right) \otimes T\left(x^i_1 \right) \in T(\im(p_B)) \otimes T(B) \subseteq \im(p_B) \otimes B.\]
   This implies that $T(K(B)) \subseteq K(B)$. As a consequence, $K(B)$ is a right Hopf algebra with anti-multiplicative and anti-comultiplicative right antipode given by the (co)restriction of $T$.
\end{proof}

\begin{proposition}
    Let $B$ be a commutative bialgebra with $p_B$ injective. Then $K(B)$ is a commutative Hopf algebra and $K(B) = \mathrm{C}(B) \cong B \boxslash B$.
\end{proposition}

\begin{proof}
Dual to \cref{pro:cocom}.
\begin{invisible}
\rem{Since $p_B$ is injective, we have an endomorphism $T$ of $B$ satisfying \eqref{eq:Tcond} above. Since $B$ is commutative, the flip $\sigma \colon B \otimes B \to B \otimes B, x \otimes y \mapsto y \otimes x,$ restricts to $B \boxslash B$.
\begin{comment}
Whenever
\[x^i_1 \otimes {y_i}_1 \otimes x^i_2{y_i}_2 = x^i \otimes y_i \otimes 1,\]
we also have that
\[{y_i}_1 \otimes x^i_1 \otimes {y_i}_2x^i_2 = {y_i}_1 \otimes x^i_1 \otimes x^i_2{y_i}_2 = y_i \otimes x^i \otimes 1.\]
\end{comment}
Then, for $x^i \boxslash y_i \in B \boxslash B$,
\[T\left(x^i\varepsilon(y_i)\right) = T(x^i)\varepsilon(y_i) \stackrel{\eqref{eq:Tcond}}{=} \varepsilon(x^i)y_i = p_B(\sigma (x^i \otimes y_i)) \in \im(p_B).\]
Thus, it follows from \cref{lem:Tim} that $K(B)$ is a commutative right Hopf algebra, whence a Hopf algebra by the right-hand analogue of \cite[Theorem 3]{GNT}. Since it satisfies the desired universal property in view of \cref{prop:pconvd}\,\ref{item:pconvd2}, we conclude that $K(B) = \mathrm{C}(B)$. Finally, since $B^\op = B$, \cref{prop:pconvd}\,\ref{item:pconvd4} entails that $\im(p_B) = K(B)$,
% the fact that if $x^i \boxslash y_i \in B \boxslash B$, then $y_i \boxslash x^i \in B \boxslash B$ as well, entails that
% \begin{align*}
%     \Delta\left(x^i\varepsilon(y_i)\right) & = x^i_1 \otimes x^i_2\varepsilon(y_i) = x^i_1\varepsilon({y_i}_1) \otimes \varepsilon({y_i}_2) x^i_2 \stackrel{\eqref{eq:Tcond}}{=} x^i_1\varepsilon({y_i}_1) \otimes T\left({y_i}_2\right)\varepsilon(x^i_2) \\
%     & = x^i \otimes T\left(y_i\right) = x^i_2 \varepsilon\left({y_i}_2\right) \otimes \varepsilon\left(x^i_1\right)T \left({y_i}_1\right) \in \im(p_B) \otimes B
% \end{align*}
% and so $\im(p_B) = K(B)$,
% \rd{[Visto che $B^\op=B$ possiamo applicare \cref{prop:pconvd} per concludere che $\im(p_B) = K(B)$.]}
implying that $p_B$ induces an isomorphism $B \boxslash B \cong K(B)$.}
\end{invisible}
\end{proof}

% As an example, note that any finite-dimensional commutative bialgebra $B$ has $p_B$ injective, by \cref{prop:semiantip}, and hence $B\boxslash B\cong \mathrm{C}(B)$. \rd{However we already proved this in \cref{prop:CBfd}.}

Finally, we provide an explicit and handy description of the \emph{cofree cocommutative Hopf algebra}. Let $B$ be a cocommutative bialgebra and let us denote by $\mathrm{C}^c(B)$ the cofree cocommutative Hopf algebra on the cocommutative bialgebra $B$, together with the canonical morphism $\epsilon_B^c \colon \mathrm{C}^c(B) \to B$.

\begin{theorem}\label{thm:CcB} 
Let $B$ be a cocommutative bialgebra. Then 
\[\mathrm{C}^c(B) = \left\{x^i \boxslash y_i \in B \boxslash B \mid y_i \otimes x^i \in B \boxslash B\right\} \subseteq B \boxslash B\]
with canonical map $\epsilon_B^c \colon \mathrm{C}^c(B) \to B$ provided by the restriction of $p_B \colon B \boxslash B \to B$.
\end{theorem}

\begin{proof}
    Since $B$ is cocommutative, \cref{rem:freecom} % \cref{lem:invosl} 
    entails that $B \boxslash B$ is a sub-bialgebra of $B \otimes B^{\mathrm{op}}$ and $p_B$ is a bialgebra map. We also have that $(B \boxslash B)^\op$ is a sub-bialgebra of $(B \otimes B^\op)^\op = B^\op \otimes B$ and so, since the flip map $\tau \colon B \otimes B^\op \to B^\op \otimes B$ is an isomorphism of bialgebras, 
    \[B ~\widetilde{\boxslash}~ B \coloneqq (B \boxslash B) \cap \tau^{-1}\big((B \boxslash B)^\op\big) = \left\{x^i \boxslash y_i \in B \boxslash B \mid y_i \otimes x^i \in B \boxslash B\right\}\]
    is still a sub-bialgebra of $B \otimes B^\op$.
    It is, in addition, a Hopf algebra because now the (co)restriction $\tau \colon B ~\widetilde{\boxslash}~ B \to B ~\widetilde{\boxslash}~ B$ of the flip map $\tau$ defines an antipode for it. Indeed, let $x^i \boxslash y_i \in B ~\widetilde{\boxslash}~ B$ (i.e., $y_i \boxslash x^i \in B \boxslash B$, too). On the one hand
    \begin{align*}
    m_{B \widetilde{\boxslash} B}\left(\id_{B \widetilde{\boxslash} B} \otimes \tau\right) & \Delta_{B \widetilde{\boxslash} B}\left(x^i \boxslash y_i\right) = ({x^i}_1 \boxslash {y_i}_1)({y_i}_2 \boxslash {x^i}_2) = {x^i}_1{y_i}_2 \boxslash {x^i}_2{y_i}_1 \\
    & = {x^i}_1{y_i}_1 \boxslash {x^i}_2{y_i}_2 = \Delta\left(x^iy_i\right) = \varepsilon\left(x^iy_i\right)1_B \boxslash 1_B,
    \end{align*}
    where the last equality follows from \cref{lem:oslash}, and on the other hand
    \begin{align*}
    m_{B \widetilde{\boxslash} B}\left(\tau \otimes \id_{B \widetilde{\boxslash} B}\right) & \Delta_{B \widetilde{\boxslash} B}\left(x^i \boxslash y_i\right) = ({y_i}_1 \boxslash {x^i}_1)({x^i}_2 \boxslash {y_i}_2) = {y_i}_1{x^i}_2 \boxslash {y_i}_2{x^i}_1 \\
    & = {y_i}_1{x^i}_1 \boxslash {y_i}_2{x^i}_2 = \Delta\left(y_ix^i\right) = \varepsilon\left(y_ix^i\right)1_B \boxslash 1_B,
    \end{align*}
    where the last equality follows from \cref{lem:oslash} again, but applied to $y_i \boxslash x^i$. Summing up, $B ~\widetilde{\boxslash}~ B$ is a cocommutative Hopf algebra and the restriction of $p_B \colon B \boxslash B \to B$ to $B ~\widetilde{\boxslash}~ B$ is a bialgebra map. We only need to prove that it is universal with respect to this property.

    Let $H$ be a cocommutative Hopf algebra with a bialgebra morphism $f \colon H \to B$ and consider the corresponding bialgebra morphism $\tilde f \colon H \to B \boxslash B, x \mapsto f(x_1) \boxslash f(S(x_2)),$ from \cref{rem:freecom}. 
    The latter lands into $B ~\widetilde{\boxslash}~ B$ because
    \[\tau\tilde f(x) = f\big(S(x_2)\big) \boxslash f(x_1) = f\big(S(x_2)\big) \boxslash f\big(S^2(x_1)\big) = f\big(S(x)_1\big) \boxslash f\big(S\big(S(x)_2\big)\big) = \tilde f\big(S(x)\big)\]%
    % The latter is, in fact, a bialgebra map in this case, since
    % %\rd{[Alla fine mi sembra che la mia idea di usare solo la cocommutatività di $B$ non funzini perché nono riesco a scambiare gli indici nelle posizioni giuste.]}\ps{[Ok]}
    % \begin{align*}
    % \Delta \tilde f(x) & = \Delta(f(x_1) \boxslash f(S(x_2))) = f(x_1)_1 \boxslash f(S(x_2))_1 \otimes f(x_1)_2 \boxslash f(S(x_2))_2 \\
    % & = f(x_1) \boxslash f(S(x_3)_1) \otimes f(x_2) \boxslash f(S(x_3)_2) = f(x_1) \boxslash f(S(x_4)) \otimes f(x_2) \boxslash f(S(x_3)) \\
    % & = f(x_1) \boxslash f(S(x_2)) \otimes f(x_3) \boxslash f(S(x_4)) = \tilde f(x_1) \otimes \tilde f(x_2) = \left(\tilde f \otimes \tilde f\right)\Delta(x)
    % \end{align*}
    % for all $x \in H$, and moreover it lands into $B ~\widetilde{\boxslash}~ B$, since
    % \begin{align*}
    % f(S(x_2))_1 & \otimes f(x_1)_1 \otimes f(S(x_2))_2f(x_1)_2 = f(S(x_4)) \otimes f(x_1) \otimes f(S(x_3))f(x_2) \\
    % & = f(S(x_4)) \otimes f(x_1) \otimes f(S(x_2)x_3) = f(S(x_2)) \otimes f(x_1) \otimes 1_B
    % \end{align*}
    % by cocommutativity of $H$. 
    Now, suppose that $f' \colon H \to B ~\widetilde{\boxslash}~ B$ is a bialgebra map such that $p_B \circ f' = f$. As we observed in the proof of \cref{lem:invosl}, 
	\[f' = (p_B \otimes p_B^\dagger) \circ \Delta \circ f' = (p_B \otimes p_B^\dagger) \circ (f' \otimes f') \circ \Delta = (f \otimes p_B^\dagger f') \circ \Delta\]
    and, since $p_B^\dagger$ is right convolution inverse of $p_B$, then $p_B^\dagger \circ f'$ is right convolution inverse of $p_B \circ f'= f$
	% Now, notice that the bialgebra map $p_B^\dagger \circ f' \colon H \to B^{\mathrm{op}}$ satisfies
	% \[f*(p_B^\dagger \circ f') = (p_B \circ f')*(p_B^\dagger \circ f') = (p_B*p_B^\dagger) \circ f' = u_{B \widetilde{\boxslash} B} \circ \varepsilon_{H}\]
	and therefore $p_B^\dagger \circ f' = f \circ S$, because the latter is already the convolution inverse of $f$. Hence, $f' = \tilde f$ and the proof is finished.
\end{proof}

At the present level of generality, we cannot expect $\mathrm{C}^c(B)$ to be $B \boxslash B$: \cref{exa:notweakpBinj} below will be about $B = \Bbbk M$ with $M = \langle a,b \mid ab=1\rangle$, for which $B \boxslash B = \mathsf{span}_\Bbbk\left\{a^n \otimes b^n \mid n \in \mathbb{N}\right\}$ (see the forthcoming \cref{lem:KkM}), but $\mathrm{C}^c(B) = \mathrm{C}(B) = \Bbbk M^\times \cong \Bbbk$, where the first equality follows from the fact that $\mathrm{C}(B)$ is cocommutative and the second one from \cref{thm:monoidbialgebra}.
\begin{invisible}
In fact, an arbitrary element of $M$ has the form $b^pa^q$ for $p,q \in \mathbb{N}$. If $b^pa^q$ and $b^ra^s$ in $M$ are such that
\[1 = b^pa^qb^ra^s = \begin{cases} b^{p+r-q}a^s & r \geq q \\ b^pa^{q-r+s} & r < q \end{cases}\]
then either $p+r-q=0=s$ or $p=0=q-r+s$, according to which case we are in, and in both cases we conclude that $p = 0 = |r-q| = 0 = s$, i.e., $p=0=s$ and $r=q$. Furthermore, the very same argument shows that if an element $b^pa^q$ is invertible, then necessarily $p=0$ (in order to have a right inverse) and $q=0$ (in order to have a left inverse), that is, $b^pa^q=1$. In addition, concerning the equality $\mathrm{C}^c(B) = \mathrm{C}(B) \cong \Bbbk$, notice that the canonical map $\epsilon^c_B \colon \mathrm{C}^c(B) \to B$ has to factor uniquely through the canonical map $\epsilon_B \colon \mathrm{C}(B) \to B$ and vice versa, since $\mathrm{C}(B)$ is cocommutative. By the two universal properties, the two factorisations are each other inverses.
\end{invisible}

%\ps{In fact, the second last statement is true in general, as the next result shows.}

It would be interesting to know whether an analogue of \cref{prop:HPcB} holds in the cocommutative case and under which conditions we may have an analogue of \cref{thm:intcomm}.

\subsection{The free bialgebra with \texorpdfstring{$i_B$}{iB} injective and the Hopf envelope of a \texorpdfstring{$n$}{n}-Hopf algebra} \hfill

%\ps{AGGIUNGERE COMMENTO MANCANZA SIMMETRIA}

\begin{invisible}\marginpar{\tiny\ps{In AdjMon or in AdjYD?}}
[\ps{DA SPOSTARE IN AdjMon:}] The aforementioned construction could be understood as an Hopf analogue of the  "right cancellative congruence" on a monoid (see AdjYD), that is the smallest congruence $\mathcal{R}$ on a monoid $G$ such that $G/\mathcal{R}$ is a right cancellative monoid.
Indeed if we apply the construction to a monoid algebra $B=\Bbbk G$, then $\tilde{B}$ will be a quotient of it, whence a monoid algebra of the form $\tilde{B}=\Bbbk L$ with $L=\pi_0(G)$. Moreover, we will see that, for a monoid algebra $\Bbbk C$, the map $i_{\Bbbk C}$ is injective if and only if $C$ is right cancellative (see AdjMon). Thus, given a monoid map $f:G\to C$  into a right cancellative monoid, then $\Bbbk f:\Bbbk G\to \Bbbk H$ factors through $\Bbbk L$ and hence $f$ factors through $L$. [From cf. AdjYD, Prop 0.8, it seems that $G$ is left adjustable iff $B_\infty=B_1$, with notations as in the following remark.]
\end{invisible}

Let $B$ be a left $n$-Hopf algebra. \cref{prop:semiantip} ensures that the morphism $i_B$ is surjective and so \cref{prop:etaSurj} entails that $\mathrm{H}(B)$ can be realised as a quotient of $B$ itself. However, the constructions and results from \S\ref{ssec:HopfEnv} do not seem sufficient to explicitly construct $\mathrm{H}(B)$ in this case. Here we address this difficulty and, to this aim, we include the construction of the \emph{free bialgebra $Q^\infty(B)$, with $i_{Q^\infty(B)}:Q^\infty(B)\to Q^\infty(B)\oslash Q^\infty(B)$ injective, generated by a bialgebra} $B$.
Our approach draws heavily on Kelly's work \cite{Kelly1}. Explicitly recall  that a \emph{well-pointed endofunctor} $(F,\sigma)$ on a category $\Cc$ is an endofunctor $F:\Cc\to\Cc$ together with a natural transformation $\sigma:\id\to F$ such that $F\sigma=\sigma F:F\to F^2$. An \emph{$F$-algebra} is an object $C\in\Cc$ together with an action $\mu:F(C)\to C$ satisfying $\mu\circ\sigma_C=\id_C$. With an \emph{$F$-algebra morphism} $f:(C,\mu)\to (C',\mu')$ defined as a map $f:C\to C'$ for which $f\circ \mu=\mu'\circ F(f)$, the $F$-algebras and their morphisms form a category $F\text{-}\Alg$. Kelly investigates conditions under which the obvious forgetful functor $U:F\text{-}\Alg\to \Cc$ has a left adjoint. Our aim is to apply this machinery to a suitable well-pointed endofunctor $(Q,q)$ arising from the bialgebra introduced in \cref{def:QB}.

Recall that a subcategory $\mathcal{D}$ of a category $\mathcal{C}$ is a \emph{replete} subcategory if for any object $D$ of $\mathcal{D}$ and any isomorphism $f \colon D \to C$ in $\mathcal{C}$, both $C$ and $f$ are also in $\mathcal{D}$.

\begin{proposition}
\label{prop:wpe}
%\rd{[Added on 2024/11/29]}
With notations as in \cref{def:QB}, the assignment $B\mapsto Q(B)$ defines a well-pointed endofunctor $(Q,q)$ on $\Bialg$. Moreover $Q\text{-}\Alg$ may be identified with the full replete subcategory of $\Bialg$ consisting of those bialgebras $B$ for which $i_B$ is injective.
\end{proposition}

\begin{proof}
By \cref{lem:iB-inj}\,\ref{ib-inj_item2}, given a bialgebra map $f \colon B\to C$
%and $k\in\ker(i_B)$, we have $i_Cf(k)=(f\oslash f)i_B(k)=0$ so that
we have that $f(\ker(i_B))\subseteq \ker(i_C)$ and hence $f(\ker(i_B)B)=f(\ker(i_B))f(B)\subseteq \ker(i_C) C$. As a consequence there is a (necessarily unique) bialgebra map $Q(f):Q(B)\to Q(C)$ such that $Q(f)\circ q_B=q_C\circ f$. This defines an endofunctor $Q:\Bialg\to\Bialg$ and a natural transformation $q:\id\to Q$ with component the projection $q_B$.
By naturality of $q$, we have $Q(q_B)\circ q_B=q_{Q(B)}\circ q_B$ so that, since $q_B$ is surjective, we get $Q(q_B)=q_{Q(B)}$ i.e. $Qq=qQ$. We have so proved that $(Q,q)$ is a well-pointed endofunctor on $\Bialg$.
By \cite[Proposition 5.2]{Kelly1}, the category $Q\text{-}\Alg$ may be identified with the full replete subcategory of $\Bialg$ determined by those bialgebras $B$ such that $q_B$ is bijective.
\begin{invisible}
We report here Kelly's idea. A $Q$-algebra is a bialgebra $B$ together with an action $\mu:Q(B)\to B$ satisfying $\mu\circ q_B=\id_B$.
Moreover, by naturality of $q$, we have
$q_B\circ \mu=Q(\mu)\circ q_{Q(B)}
=Q(\mu)\circ Q(q_{B})
=Q(\mu\circ q_{B})=\id_{Q(B)}.$ Thus the existence of $\mu$ is equivalent to  require that $q_B$ is bijective. Moreover, the condition $f\circ \mu=\mu'\circ Q(f)$ defining a $Q$-algebra map $f:(B,\mu)\to (B',\mu')$ is equivalent to  $f\circ q_B^{-1}=q_{B'}^{-1}\circ Q(f)$ which is always true by naturality of $q.$
\end{invisible}
Moreover, since $q_B$ is surjective and $\ker(q_B)=\ker(i_B)B$, $q_B$ is bijective if, and only if, it is injective if, and only if, $i_B$ is injective.
\end{proof}

%\rd{[Modified on 2024/12/03]} \\
Now a reflector for the forgetful functor $U:Q\text{-}\Alg\to \Bialg$ is constructed as the colimit of the so-called \emph{free-algebra sequence for $B$}, if it exists (\cite[\S5.2]{Kelly1}).  This sequence is defined recursively as follows,
\begin{equation}\label{eq:dirsystQ}
\xymatrix@C=1.5cm{B\ar[r]^{q_B}&Q(B)\ar[r]^{q_{Q(B)}}&Q^2(B)\ar[r]^{q_{Q^2(B)}}&\cdots}
\end{equation}
where $Q^0=\id$ and $Q^{n+1}=QQ^n$ with connecting-map equal to $qQ^n:Q^n\to QQ^n$.

For the limit-ordinal $\infty$, we define $Q^\infty \coloneqq \underrightarrow{\lim }_{n}Q^n$ with, as connecting-maps $\pi^\infty_n \colon Q^n\to Q^\infty$, the generators of the colimit-cone.

Recall that a full subcategory $\mathcal{D}$ of a category $\mathcal{C}$ is \emph{reflective} if the inclusion functor has a left adjoint and it is \emph{epi-reflective} when, in addition, the unit of the adjunction is an epimorphism.

\begin{theorem} %\rd{[Added on 2024-12-02]}
\label{thm:Kellyfree}
The natural transformation $q Q^\infty:Q^\infty\to QQ^\infty$ is a natural isomorphism. As a consequence
the free-algebra sequence for $B$ converges (in the sense of \cite[\S5.2]{Kelly1}) at $\infty$ and $Q^\infty(B)$ is the free $Q$-algebra on $\Bialg$, the reflection of $B$ into $Q\text{-}\Alg$ being the connecting-map component $\pi_0^\infty:B\to Q^\infty(B)$. Thus $Q\text{-}\Alg$ results to be an epi-reflective subcategory of $\Bialg$.
\end{theorem}

\begin{proof}
Recall, first of all, that the direct limit in vector spaces of a system of algebras or bialgebras is again an algebra or a bialgebra, respectively. In particular, the bialgebra $Q^{\infty}(B)$ can be realized as the colimit of \eqref{eq:dirsystQ} computed in vector spaces. To prove the theorem,
we just have to prove that the component $q_{Q^\infty(B)}:Q^\infty(B)\to QQ^\infty(B)$ of the natural transformation
$q Q^\infty$ is an isomorphism so that
the free-algebra sequence for $B$ converges at $\infty$. Indeed, if this is the case, then \cite[Proposition 5.3]{Kelly1} states that $Q^\infty(B)$ results to be the free $Q$-algebra on $\Bialg$, the reflection of $B$ into $Q\text{-}\Alg$ being the connecting-map component $\pi_0^\infty:B\to Q^\infty(B)$. Thus $Q\text{-}\Alg$ results to be a reflective subcategory of $\Bialg$, in fact an epi-reflective one as we are going to show that $\pi_0^\infty$ is  surjective.
Now, let us simplify the notation by setting $B^n\coloneq Q^n(B)$ and $\pi_n^{n+1}\coloneq q_{Q^n(B)}$ so that the free-algebra sequence rewrites more compactly as
\[\xymatrix@C=1.5cm{B^0\ar[r]^{\pi_0^1}&B^1\ar[r]^{\pi_1^2}&B^2\ar[r]^{\pi_2^3}&\cdots}\]
We also set $B^\infty:=Q^\infty(B)=\underrightarrow{\lim }_{n}Q^n(B)=\underrightarrow{\lim }_{n}B^n$, computed in vector spaces, we denote by
$\pi_n^\infty:B^n\to B^\infty$ the component of the connecting-map and by $i_{B^\infty} \colon B^\infty \to B^\infty \oslash B^\infty$ the corresponding canonical morphism.
By definition of $Q$, if we set $K^n:=\ker(i_{B^n})$, then we have $B^{n+1}=B^n/\langle K^n\rangle$ for every $n\geq 0$ and $K^n \subseteq \ker(\pi_n^{n+1})$. Set $K^\infty \coloneqq \ker (i_{B^\infty})$.
Since each $\pi_n^{n+1}$ is a bialgebra map by \cref{prop:pconv}\,\ref{item:pconv1}, and so $i_{B^{n+1}} \circ \pi_n^{n+1} = (\pi_n^{n+1}\oslash \pi_n^{n+1}) \circ i_{B^n}$ by \cref{lem:foslf}, we have a direct system of vector spaces $K^0 \xrightarrow{0} K^1 \xrightarrow{0} K^2 \xrightarrow{0} \cdots$ which fits the forthcoming diagram \eqref{eq:bigcomplex}. We will check that $K^\infty \cong \underrightarrow{\lim }_{n} K^n$, the latter being $0$ as all the maps of the system are $0$.
\begin{equation}\label{eq:bigcomplex}
\begin{gathered}
\xymatrix@R=15pt@C=40pt{K^0\ar@{^(->}[d]\ar[r]^{0}&K^1\ar@{^(->}[d]\ar[r]^{0}&K^2\ar@{^(->}[d]\ar[r]^{0}&\cdots&K^\infty\ar@{^(->}[d]\\
B^0\ar[d]^{i_{B^0}}\ar[r]^{\pi_0^1}&B^1\ar[d]^{i_{B^1}}\ar[r]^{\pi_1^2}&B^2\ar[d]^{i_{B^2}}\ar[r]^{\pi_2^3}&\cdots&B^\infty\ar[d]^{i_{B^\infty}}\\
B^0\oslash B^0\ar[r]^{\pi_0^1\oslash\pi_0^1}&B^1\oslash B^1\ar[r]^{\pi_1^2\oslash\pi_1^2}&B^2\oslash B^2\ar[r]^-{\pi_2^3\oslash\pi_2^3}&\cdots&B^\infty\oslash B^\infty
}
\end{gathered}
\end{equation}

Since passing to the direct limit preserves exactness (see e.g. \cite[Proposition 3, Page 287]{BourbakiI}), from the exact sequences $0\to K^n\to B^n\overset{i_{B^n}}{\to} B^n\oslash B^n$ for $n \geq 0$, we get the exact sequence
\begin{equation}\label{eq:dirlim}
0\to \underrightarrow{\lim }_{n} K^n\to \underrightarrow{\lim }_{n}B^n\xrightarrow{\underrightarrow{\lim }_{n}i_{B^n}} \underrightarrow{\lim }_{n} (B^n\oslash B^n).
\end{equation}
In view of \cite[Proposition 7, Page 290]{BourbakiI},
\begin{align*}
\underrightarrow{\lim }_{n} (B^n\oslash B^n)
 & \cong \underrightarrow{\lim }_{n} ((B^n\otimes B^n)\otimes_{B^n}\Bbbk)
\cong
( \underrightarrow{\lim }_{n}(B^n\otimes B^n))\otimes_{\underrightarrow{\lim }_{n}B^n}\Bbbk \\
& \cong
( (\underrightarrow{\lim }_{n} B^n)\otimes (\underrightarrow{\lim }_{n} B^n))\otimes_{\underrightarrow{\lim }_{n}B^n}\Bbbk =
(B^\infty\otimes B^\infty)\otimes_{B^\infty}\Bbbk
\cong
B^\infty\oslash B^\infty
\end{align*}
Hence, \eqref{eq:dirlim} rereads as
\[0\to \underrightarrow{\lim }_{n} K^n\to B^\infty \xrightarrow{i_{B^\infty}} B^\infty\oslash B^\infty.\]
Thus, $K^\infty \cong \underrightarrow{\lim }_{n} K^n$
% Now take $z\in \underrightarrow{\lim }_{n} K^n.$ Then there is $z_n\in K^n$, for some $n$, such that $z=\pi_n^\infty(z_n)$. Thus $z=\pi_n^\infty(z_n)=\pi_{n+1}^\infty\pi_n^{n+1}(z_n)=0$.
and we have so proved that $i_{B^\infty}$ is injective.
Equivalently, %$K^\infty=0$ and hence
$q_{B^\infty}$ is bijective.

Finally, note that $\pi_0^\infty:B\to B^\infty$ is surjective. Indeed, since for every $n \geq 1$, $B^n = \pi_0^n(B)$, one gets that $B^\infty = \bigcup_{n \in \mathbb{N}} \pi^\infty_n(B^n) = \bigcup_{n\in\mathbb{N}} \pi_n^\infty(\pi_0^n(B))=\pi_0^\infty(B)$ (see e.g.\ \cite[\S I.10.3, Page 120]{BourbakiI}).
\end{proof}

\begin{comment}
\begin{invisible}\\
Let now $f:B\to C$ be a bialgebra map such that $i_C$ is injective.  By \cref{lem:iB-inj}(2), we have that $K^0=\ker (i_B)\subseteq \ker (f)$ so that we get a unique bialgebra map $f_1:B^1\to C$ such that $f_1\pi_0^1=f$. Similarly $K^1=\ker (i_{B^1})\subseteq \ker (f_1)$ so that we get a unique bialgebra map $f_2:B^2\to C$ such that $f_2\pi_1^2=f_1$. Going on this way we get $f_n:B^n\to C,$ such that $f_n\pi_{n-1}^n=f_{n-1}$.
\[\xymatrix{B^0\ar[r]^{\pi_0^1}\ar[d]_{f}&B^1\ar[r]^{\pi_1^2}\ar[dl]|{f_1}&B^2\ar[r]^{\pi_2^3}\ar[dll]|{f_2}&\cdots&B^\infty\ar[dllll]|{f_\infty}\\
C}\]
Then there is a unique morphism $f_\infty:B^\infty\to C$ such that $f_\infty\pi_n^\infty=f_n$. In particular $f_\infty\pi_0^\infty=f_0=f$ and $f_\infty$ is uniquely determined by this equality as $\pi_0^\infty:B\to B^\infty$ is surjective. Thus $(B^\infty,\pi_0^\infty)$ is the \emph{free bialgebra $B^\infty$ with $i_{B^\infty}$ injective generated by the bialgebra} $B$.\end{invisible}
\end{comment}

\begin{corollary}
    %\rd{[Added on 2024-12-17]}
\label{thm:Kellyisu}
Let $B$ be a bialgebra with $i_B$ surjective. Then the map $i_{Q^\infty(B)}$ is bijective i.e.\ $Q^\infty(B)$ is a right Hopf algebra with anti-multiplicative and anti-comultiplicative right antipode.
\end{corollary}

\begin{proof}
 By \cref{thm:Kellyfree},  the map $i_{Q^\infty(B)}$ is injective. Moreover, since $i_B$ surjective so is $i_{Q^\infty(B)}$ in view of \cref{lem:iB-su}\,\ref{iBsurj.item4}. Thus $i_{Q^\infty(B)}$ is bijective and we conclude by \cref{prop:Frobenius}.
\end{proof}

\begin{proposition}
%[\rd{added on 2024-11-22}]
\label{prop:envleftn}
Let $B$ be a left $n$-Hopf algebra. Then $\mathrm{H}(B)=Q^\infty(B)$.
\end{proposition}

\begin{proof}
We keep some notation from the proof of \cref{thm:Kellyfree}. By \cref{coro:pconv}, the map $\pi_0^1:B\to Q(B)$ is convolution invertible thus, in particular, left convolution invertible. Since the map $\pi_0^\infty:B\to Q^\infty(B)$ can be written as $\pi_0^\infty=\pi_1^\infty\circ \pi_0^1$, we get that $\pi_0^\infty$ is left convolution invertible, too.
Since $B$ is a left $n$-Hopf algebra, $i_B$ is surjective by \cref{prop:semiantip} and so \cref{thm:Kellyisu} entails that $Q^\infty(B)$ is a right Hopf algebra.
% By \cref{lem:iB-su}\,\ref{iBsurj.item4}, also $i_{Q^\infty(B)}$ is surjective as $Q^\infty(B)$ is a quotient of $B$.
% Since, by \cref{thm:Kellyfree}, the map $i_{Q^\infty(B)}$ is injective, we get it is bijective and hence, by \cref{prop:Frobenius}, the bialgebra $Q^\infty(B)$ is in fact a right Hopf algebra.
Therefore, we get that $Q^\infty(B)$ is a Hopf algebra by \cref{lem:n-Hopfconv}. The universal property of $Q^\infty(B)$ now entails that $\mathrm{H}(B)=Q^\infty(B)$.
\end{proof}

\begin{remark}\label{rem:symmconstruction}
    %\rd{[added on 2024-12-18]}
   %  \ps{[Honestly, I am not completely convinced by this remark: is it aimed at providing a different existence result for $H(B)$? But at the price of mentioning the symmetric constructions]}
   % \rd{[L'idea è mostrare che $H(B)$ può essere ottenuta in modo iterato applicando infinite volte $Q$ e $Q'$.]}
Let $B$ be a bialgebra with $i_B$ surjective. In \cref{prop:etaSurj} we proved that the map $\eta_B:B\to \mathrm{H}(B)$ is surjective by using the existence of $\mathrm{H}(B)$. In fact, the bialgebra  $\mathrm{H}(B)$ can be iteratively constructed by means of the machinery we have developed so far. Indeed, by \cref{thm:Kellyisu} the iteratively constructed bialgebra $Q^\infty(B)$ is a right Hopf algebra i.e. a right $0$-Hopf algebra. Thus we can apply the symmetric version of \cref{prop:envleftn} to conclude that $\mathrm{H}(Q^\infty(B))=\tilde{Q}^\infty Q^\infty(B)$ where $\tilde{Q}(B)=B/B\ker(i'_B)$. Here $i'_B: B \to  B\obackslash B, b\mapsto 1_B \obackslash b$ and $B\obackslash B =(_\bullet B\otimes _\bullet B)/B^+(_\bullet B\otimes _\bullet B)$. Now the universal property of $Q^\infty(B)$ entails that, for any bialgebra map $f:B\to H$ with $H$ a Hopf algebra, one has a unique bialgebra map $f':Q^\infty(B)\to B$ such that $f'\circ \pi_0^\infty =f$. Therefore the similar property of $\tilde{Q}$ implies there is a unique bialgebra map $f'':\tilde{Q}^\infty Q^\infty(B)\to H$
such that $f''\circ (\pi')_0^\infty=f'$ where $(\pi')_0^\infty:Q^\infty(B)\to \tilde{Q}^\infty Q^\infty(B)$ is the canonical projection.
Thus $\mathrm{H}(B)=\mathrm{H}(Q^\infty(B))=\tilde{Q}^\infty Q^\infty(B)$.
\end{remark}

\begin{remark}
Note that $Q^\infty(B)$ and $\mathrm{H}(B)$ are not always isomorphic for an arbitrary bialgebra $B$. For instance, if we start from a bialgebra $B$ with $i_B$ is injective, then $Q(B)=B/\ker(i_B)B\cong B$ so that $Q^\infty(B)\cong B$. Thus, if $B$ is not a Hopf algebra, we get $Q^\infty(B)\ncong \mathrm{H}(B)$ in this case. For instance, the monoid algebra $B=\Bbbk \mathbb{N}$, which is not a Hopf algebra as $\mathbb{N}$ is not a group, embeds into the group algebra $H=\Bbbk \mathbb{Z}$ so that $i_B$ is injective, by \cref{lem:iB-inj}. Thus $Q^\infty(\Bbbk \mathbb{N})\ncong \mathrm{H}(\Bbbk \mathbb{N}).$
\end{remark}

We are not aware of any example of a left $n$-Hopf for which $Q^\infty(B)\neq Q(B)$.

\begin{invisible}
\rd{[Can we find an example of a left $n$-Hopf where $Q^\infty(B)\neq Q(B)$? L'algebra di monoide $B=\Bbbk G$ è fuori gioco perchè $i_B$ suriettiva implica left adjustable (cf. prop:eureka in AdjMon) e questo sembra implicare (da scrivere bene) $i_{Q(B)}$ is injective whence $Q^\infty(B) = Q(B)$  (cf. pro:ideaFerri in AdjMon) \ps{[Non la trovo]}. \rd{[Era AdjYD. A occhio dovrebbe funzionare così: Sappiamo che $i_{B}:B=\Bbbk G\to B\otimes B=\Bbbk S$ dove $S=\{x\otimes y\mid x,y\in G\}$ (AdjMon:prop:oslashKG). Quindi $i_{B}=\Bbbk i_G$ dove $i_G:G\to S,g\mapsto g\oslash 1$ è quella in [prop:eureka]. La relazione d'equivalenza associata a questa funzione è $\{(x,y)\in G\times G \mid i_G(x)=i_G(y)\}=\mathcal{W}$. Quindi $\overline{i_G}:G/\mathcal{W}\to S,[x]\mapsto i_G(x),$ è una biiezione. Per [pro:ideaFerri] $\mathcal{W}=\eta$ e dunque $G/\mathcal{W}$ è un monoide cancellativo a destra e quindi lo è anche $S$. Ora vorrei dedurre che $i_B$ è morfismo di bialgebre (così il ker è un ideale destro) e che $Q(B)=B/ker(i_B)B=B/ker(i_B)\cong Im(i_B)=\Bbbk S$ con $S$ monoide cancellativo e quindi usare [AdjMon:prop:iBinjKG]. ]}\medskip\newline
Probably the answer is yes: we should consider the finite dual $(\Bbbk M)^\circ$. We have to check that $\mathrm{Q}((\Bbbk M)^\circ)=(\mathrm{K}(\Bbbk M))^\circ =(\Bbbk M^\ell)^\circ$ and $\mathrm{H}((\Bbbk M)^\circ)=(\mathrm{C}(\Bbbk M))^\circ =(\Bbbk M^\times)^\circ.$
]}\medskip\newline
\end{invisible}

\subsection{The cofree Hopf algebra of a perfect bialgebra and of a \texorpdfstring{$n$}{n}-Hopf algebra}\label{ssec:itercofree}
%\rd{[added on 2024-12-09]}
The following result is a straightforward counterpart of \cref{prop:wpe} in the dual setting.

\begin{proposition}
\label{prop:wpedual}
%\rd{[Added on 2024/11/29]}
With notations as in \cref{def:KB}, the assignment $B\mapsto K(B)$ defines a well-copointed endofunctor $(K,k)$ on $\Bialg$. Moreover the category $K\text{-}\Coalg$ of
$K$-coalgebras and their morphisms may be identified with the full replete subcategory of $\Bialg$ consisting of those bialgebras $B$ for which $K(B)=B$ or, equivalently,  for which $p_B$ is surjective.
\end{proposition}

\begin{invisible} We include here the proof of what is claimed above.
\begin{proof}
Given a bialgebra map $f:A\to B$ and an element $a\in K(A)$, we have $\Delta_B(f(a))
=(f\otimes f)\Delta_A(a)
\in (f\otimes f)(\im(p_A)\otimes A)
\subseteq \im(f\circ p_A)\otimes B
=\im(p_B\circ (f\boxslash f))\otimes B
\subseteq \im(p_B)\otimes B
$ so that $f(a)\in K(B)$ and hence the restriction of $f$ yields a (necessarily unique) bialgebra map $K(f):K(A)\to K(B)$ such that $k_B\circ K(f)=f\circ k_A$. Thus we get a functor $K:\Bialg\to \Bialg$ and a natural transformation $k:K\to \id$ which is injective on components. The naturality yields $k_B\circ K(k_B)=k_B\circ k_{K(B)}$ and hence, by injectivity of $k_B$, we arrive at $K(k_B)=k_{K(B)}$ so that $Kk=kK$. We have so proved that $(K,k)$ is a well-copointed endofunctor. A $K$-coalgebra is then an object $A$ together with a coaction $\rho: A\to K(A)$ satisfying $k_A\circ \rho =\id_A$ but, since $k_A$ is injective, this is equivalent to ask that $k_A$ is surjective i.e. $K(A)=A$. Since $K(A)\subseteq \im(p_A)\subseteq A$,
the latter equality is equivalent to surjectivity of $p_A$. Therefore the category $K\textbf{-}\Coalg$ of $K$-coalgebras and their morphisms may be identified with the full replete subcategory of $\Bialg$ consisting of those bialgebras $B$ for which $p_B$ is surjective.
\end{proof}
\end{invisible}

Because of a lack of symmetry due to the fact that the inverse limit is not right exact (see below), we are led to restrict our attention to the full subcategory $\Bialg^{\mathsf{in}}$ of $\Bialg$ consisting of those bialgebras $B$ whose associated $p_B \colon B \boxslash B \to B$ is injective.

\begin{corollary}%[\ps{Added 13/03/2025}]
    The well-copointed endofunctor $(K,k)$ from \cref{prop:wpedual} induces, by restriction and corestriction, a well-copointed endofunctor $(K,k)$ on the full subcategory $\Bialg^{\mathsf{in}}$ of bialgebras $B$ with $p_B$ injective. The corresponding category of $K$-coalgebras may be identified with the full replete subcategory of right Hopf algebras whose right antipode is an anti-bialgebra map.
\end{corollary}

\begin{proof}
    In view of \cref{prop:pbInj}\,\ref{item:pbInj4} and of \cref{prop:pconvd}\,\ref{item:pconvd1}, if $B$ is a bialgebra whose associated $p_B$ is injective, then the same holds for $K(B)$. A $K$-coalgebra is then a bialgebra $B$ whose $p_B$ is at the same time injective (by hypothesis) and surjective (by \cref{prop:wpedual}), whence bijective. We conclude by \cref{prop:Frobenius} and \cref{prop:preserve_right_ants}.
\end{proof}

In order to provide a right adjoint for the forgetful functor from $\Bialg^{\mathsf{in}}$ to the category of right Hopf algebras whose right antipode is an anti-bialgebra map, we consider what we may call the cofree-coalgebra sequence for $B$, namely
\[\xymatrix@C=1.5cm{\cdots\ar[r]^-{k_{K^2(B)}}&K^2(B)\ar[r]^-{k_{K(B)}}&K(B)\ar[r]^-{k_B}& B.}\]
Even if the inverse limit of an inverse system of surjective maps is not necessarily surjective (see \cite[Remark (2), page 285]{BourbakiI}), we can still adapt the proof we used in the dual setting to check that the above sequence converges at $\infty$.

As above, we define $K^\infty \coloneqq \underleftarrow{\lim }_{n}K^n$ with, as connecting-maps $\sigma_\infty^n \colon K^\infty\to K^n$, the generators of the limit-cone and recall that a full subcategory $\mathcal{D}$ of a category $\mathcal{C}$ is \emph{coreflective} if the inclusion functor has a right adjoint and it is \emph{mono-coreflective} when, in addition, the unit of the adjunction is a monomorphism (see e.g. \cite{Herrlich-Strecker}).

\begin{theorem}%[\ps{Adjusted 13/03/2025}]
\label{thm:Kellycofree}
Consider the well-copointed endofunctor $(K,k)$ on the category $\Bialg^{\mathsf{in}}$. The natural transformation $k K^\infty \colon K^\infty\to KK^\infty$ is a natural isomorphism. As a consequence, the cofree-coalgebra sequence for $B$ converges at $\infty$ and $K^\infty(B)$ is the cofree $K$-coalgebra on $\Bialg^{\mathsf{in}}$, the coreflection being the connecting-map component $\sigma^0_\infty \colon K^\infty(B)\to B$. Thus, right Hopf algebras whose right antipode is an anti-bialgebra morphisms form a mono-coreflective subcategory of $\Bialg^{\mathsf{in}}$.
\end{theorem}

\begin{proof}
As we did in the dual setting, we simplify the notation by defining $B^n\coloneqq K^n(B)$ and $\sigma_{n+1}^n \coloneqq k_{B^n}$.
Note that the maps $\sigma^n_{n+1}$ are indeed inclusions, so that the inverse limit of this inverse system is $B^\infty \coloneqq K^\infty(B) = \underleftarrow{\lim }_{n}K^n(B) = \bigcap_{n\in\mathbb{N}} K^n(B)$. As observed in \cite[page 167]{Radford-book}, the intersection of a family of sub-bialgebras is a sub-bialgebra so that $K^\infty(B)$ is a sub-bialgebra of $B$. Denote by  $\sigma^n_\infty \colon B^\infty \to B^n$ the component of the connecting-map and by $p_{B^\infty} \colon B^\infty \boxslash B^\infty \to B^\infty$ the corresponding canonical morphism. Set also $\overline{B^\infty} \coloneqq \mathrm{coker} (p_{B^\infty})$.

For every $n\geq 0$, denote by $q_n \colon B^n\to \overline{B^n} = B^n/\im(p_{B^n})$ the cokernel of the map $p_{B^n}$.
As claimed in \cref{lem:KBcoalg}, $B^{n+1} = K(B^n)$ is the biggest sub-coalgebra of $B^n$ contained in $\im(p_{B^n})$. In particular, $q_n \circ \sigma^n_{n+1} = 0$ for every $n \geq 0$.

Since each $\sigma^n_{n+1}$ is a bialgebra map by \cref{prop:pconvd}\,\ref{item:pconvd1}, and so $p_{B^{n}} \circ (\sigma^n_{n+1}\boxslash \sigma^n_{n+1}) = \sigma^n_{n+1} \circ p_{B^{n+1}}$ by \cref{lem:foslf}, we have the inverse system of vector spaces $\cdots \xrightarrow{0} \overline{B^2} \xrightarrow{0} \overline{B^1} \xrightarrow{0} \overline{B^0}$ which fits the following diagram
\[
\begin{gathered}
\xymatrix@R=15pt@C=40pt{
 & & 0 \ar[d] & 0 \ar[d] & 0 \ar[d] \\
B^\infty\boxslash B^\infty\ar[d]^{p_{B^\infty}}&
\cdots\ar[r]^-{\sigma_3^2\boxslash \sigma_3^2}&B^2\boxslash B^2\ar[d]^{p_{B^2}}\ar[r]^-{\sigma_2^1\boxslash \sigma_2^1}&B^1\boxslash B^1\ar[d]^{p_{B^1}}\ar[r]^-{\sigma_1^0\boxslash \sigma_1^0}& B^0\boxslash B^0\ar[d]^{p_{B^0}}\\
B^\infty\ar@{->>}[d]^{q_\infty}&\cdots\ar[r]^-{\sigma_3^2}&B^2\ar@{->>}[d]^{q_2}\ar[r]^-{\sigma_2^1}&B^1\ar@{->>}[d]^{q_1}\ar[r]^-{\sigma_1^0}& B^0\ar@{->>}[d]^{q_0}\\
\overline{B^\infty}&\cdots\ar[r]^-{0}&\overline{B^2}\ar[r]^-{0}&\overline{B^1}\ar[r]^-{0}& \overline{B^0} .
}
\end{gathered}
\]
Since for every $n \geq 0$ we have an exact sequence
\[0\to B^n\boxslash B^n \to B^n\otimes B^n \xrightarrow{\gamma_n} B^n\otimes B^n\otimes B^n,\]
where $\gamma_n(x\otimes y) = x_1\otimes y_1\otimes x_2y_2-x\otimes y\otimes 1$, by \cite[Proposition 1, page 285]{BourbakiI} we conclude that
\[0\to \underleftarrow{\lim }_{n}(B^n\boxslash B^n)
\to \underleftarrow{\lim }_{n}(B^n\otimes B^n) \xrightarrow{\underleftarrow{\lim }_{n}\gamma_n} \underleftarrow{\lim }_{n}(B^n\otimes B^n\otimes B^n)\]
is exact. In addition, since the $\sigma_{n+1}^n \otimes \sigma_{n+1}^n$ are still injective, \cite[Exercise 1.2.8]{Radford-book} allows us to deduce that
\[\underleftarrow{\lim }_{n}(B^n\otimes B^n) = \bigcap_{n \in \mathbb{N}}(B^n\otimes B^n) = \left(\bigcap_{n \in \mathbb{N}}B^n\right)\otimes \left(\bigcap_{n \in \mathbb{N}}B^n\right) = B^\infty \otimes B^\infty\]
and hence that
\[0\to \bigcap_{n\in\mathbb{N}}(B^n\boxslash B^n)
\to B^\infty\otimes B^\infty \xrightarrow{\gamma_\infty} B^\infty\otimes B^\infty\otimes B^\infty\]
is exact, which means that $\underleftarrow{\lim }_{n}(B^n\boxslash B^n) = \ker(\gamma_\infty)= B^\infty\boxslash B^\infty$.
Moreover, since for every $n \geq 0$ the sequence
\[0 \to B^n\boxslash B^n \xrightarrow{p_{B^n}} B^n \xrightarrow{q_n} \overline{B^n} \to 0\]
is exact, too, \cite[Proposition 1, page 285]{BourbakiI} again entails that
\[0 \to \underleftarrow{\lim }_{n}\left(B^n\boxslash B^n\right) \xrightarrow{\underleftarrow{\lim }_{n}p_{B^n}} \underleftarrow{\lim }_{n}B^n \xrightarrow{\underleftarrow{\lim }_{n}q_n} \underleftarrow{\lim }_{n}\overline{B^n}\]
is exact, i.e., that
\[0 \to B^\infty\boxslash B^\infty \xrightarrow{p_{B^\infty}} B^\infty \xrightarrow{\underleftarrow{\lim }_{n}q_n} \underleftarrow{\lim }_{n}\overline{B^n}\]
is exact. As, in this very particular case, $\underleftarrow{\lim }_{n}\overline{B^n} = 0$ and $\underleftarrow{\lim }_{n} q_n = 0$, we conclude that $p_{B^\infty}$ is surjective as well, i.e.\ it is bijective and $K^\infty(B)$ is a $K$-coalgebra.

Let now $f\colon C\to B$ be a bialgebra map such that $p_C$ is bijective. By \cref{prop:pconvd}\,\ref{item:pconvd2}, we have that $\im(f) \subseteq K(B)$.
% By \cref{prop:pB-surj}\,\ref{pB-surj_item2}, we have that $\im(f) \subseteq \im(p_B)$, and since $\im(f)$ is a sub-coalgebra of $B$, we even have that $\im(f) \subseteq K(B)$, because $K(B)$ is the biggest sub-coalgebra contained in $\im(p_B)$. 
Thus, there is a (necessarily, unique) bialgebra map $f_1 \colon C \to K(B) = B^1$ such that $\sigma_1^0 \circ f_1 = f$. Similarly, 
% $\im(f_1) \subseteq \im(p_{B^1})$ and hence 
$\im(f_1) \subseteq K(B^1) = B^2$, so that we get a unique bialgebra map $f_2 \colon C \to B^2$ such that $\sigma_2^1 \circ f_2 = f_1$. Carrying on this way we get $f_n \colon C \to B^n$ such that $\sigma^{n-1}_{n} \circ f_n = f_{n-1}$:
\[
\xymatrix{
B^\infty & \cdots \ar[r] & B^2 \ar[r]^{\sigma_2^1} & B^1 \ar[r]^{\sigma_1^0} & B^0  \\
 & & & & C \ar[u]_-{f} \ar[ul]|-{f_1} \ar[ull]|-{f_2} \ar[ullll]|-{f_\infty}
 }
\]
Then, there is a unique morphism $f_\infty \colon C \to B^\infty$ such that $\sigma_\infty^n \circ f_\infty = f_n$. In particular $\sigma_\infty^0 \circ f_\infty = f_0 = f$ and $f_\infty$ is uniquely determined by this equality as $\sigma^0_\infty \colon B^\infty \to B$ is injective. Thus $(B^\infty,\sigma^0_\infty)$ is the \emph{cofree bialgebra $B^\infty$ with $p_{B^\infty}$ bijective generated by the bialgebra} $B$.
\end{proof}

\begin{theorem}
%\rd{[Added on 2025-01-08]}
\label{prop:CBleftArt}
Let $B$ be a right perfect bialgebra. Then $\mathrm{C}(B)= K^\infty(B)$.
\end{theorem}

\begin{proof}
We keep some notation from the proof of \cref{thm:Kellycofree}.
Since $B$ is right perfect, $p_B$ is injective (cf.\ \cref{pro:Artinian}). Thus, by \cref{prop:pconvd} and \cref{cor:kconv}, $K(B)$ is a sub-bialgebra of $B$ and the canonical injection $k_B:K(B)\to B$ is a two-sided convolution invertible bialgebra map.
%is a right convolution invertible, bialgebra map. Therefore, by \cref{lem:Skry}, the map $k_B$ is two-sided convolution invertible.
Since the inclusion $\sigma^0_\infty:K^\infty(B)\to B$ can be written as $\sigma^0_\infty=\sigma^1_\infty\circ \sigma^0_1=\sigma^1_\infty\circ k_B$, we get that $\sigma^0_\infty$ is two-sided convolution invertible, too. By \cref{thm:Kellycofree}, $K^\infty(B)$ is a right Hopf algebra.
By \cref{lem:n-Hopfconv}\,\ref{item:fHopf2}, we get that $K^\infty(B)$ is a Hopf algebra, and hence,  by its universal property, that $\mathrm{C}(B)=K^\infty(B)$.
 \end{proof}

We now state without proof (the proof being dual) the dual version if \cref{prop:envleftn}.

\begin{proposition}
%\rd{[added on 2024-12-17]}
\label{prop:envleftndual}
Let $B$ be a left $n$-Hopf algebra. Then $\mathrm{C}(B)=K^\infty(B)$.
\end{proposition}

\begin{invisible}
\begin{proof}
 By \cref{prop:semiantip}, the map $p_B$ is injective and so, by \cref{thm:Kellycofree}, the bialgebra $K^\infty(B)$ is a right Hopf algebra.
 %By \cref{thm:Kellypinj}, the map $p_{K^\infty(B)}$ is surjective. It is also injective by \cref{prop:pbInj}\;\ref{item:pbInj4}, whence bijective.  By \cref{prop:Frobenius}, the bialgebra $K^\infty(B)$ is in fact a right Hopf algebra.
 Furthermore, as above, \cref{cor:kconv} entails that the canonical inclusion $k_B \colon K(B) \to B$ is two-sided convolution invertible and so is the inclusion $K^\infty(B) \to B$, too.
%Let us prove that $\sigma\coloneq\sigma_\infty^0:K^\infty(B)\to B$ is left convolution invertible.
%If $S'$ is a right antipode for $B'\coloneq K^\infty(B)$ and $S$ a left $n$-antipode for $B$, then we get $\sigma*\sigma S'=\sigma (\id*S')=\sigma u'\varepsilon'=u\varepsilon'$ and
%$S\sigma*\sigma^{*n+1}=(S*\id^{*n+1})\sigma =\id^{*n}\sigma=\sigma^{*n}$. Thus we may cancel $\sigma^{*n}$ on the right to arrive at $S\sigma*\sigma=u\varepsilon'$ that is $\sigma$ is also left convolution invertible.
Therefore, by \cref{lem:n-Hopfconv} we conclude that $K^\infty(B)$ is a Hopf algebra, and hence,  by its universal property, that $\mathrm{C}(B)=K^\infty(B)$.
\end{proof}
\end{invisible}

\begin{remark}
\label{rmk:Ktilde}
%\rd{[added on 2024-12-18]}
%\ps{[Same comment as to \cref{rem:symmconstruction}]}
Let $B$ be a bialgebra with $p_B$ injective. In \cref{prop:epsiInj} we proved that the map $\epsilon_B:\mathrm{C}(B)\to B$ is injective by using the existence of $\mathrm{C}(B)$. In fact, the bialgebra  $\mathrm{C}(B)$ can be iteratively constructed by means of the machinery we have developed so far. Indeed, by \cref{thm:Kellycofree} the iteratively constructed bialgebra $K^\infty(B)$ is a right Hopf algebra i.e. a right $0$-Hopf algebra. Thus we can apply the symmetric version of \cref{prop:envleftndual} to conclude that $\mathrm{C}(K^\infty(B))=\tilde{K}^\infty K^\infty(B)$ where $\tilde{K}(B)=\{b\in B\mid \Delta(b)\in B\otimes \im(p'_B)\}$. Here $p'_B: B\boxbackslash B \to  B, x^i\otimes {y_i}\mapsto\varepsilon(x^i) {y_i}$ and $B\boxbackslash B = \prescript{\mathrm{co}B}{}{(\prescript{\bullet}{}{B} \otimes \prescript{\bullet}{}{B})}$. Now \cref{prop:pconvd}\,\ref{item:pconvd2} entails that, for any any bialgebra map $f:H\to B$ with $H$ a Hopf algebra, one has $\im(f)\subseteq K^\infty(B)$. Therefore the similar property of $\tilde{K}$ implies $\im(f)\subseteq \tilde{K}^\infty K^\infty(B)$. Thus $\mathrm{C}(B)=\mathrm{C}(K^\infty(B))=\tilde{K}^\infty K^\infty(B)$.
\end{remark}

\section{Some concrete examples}
\label{ssec:example}

In this section we collect some examples where we explicitly compute the Hopf algebras $\mathrm{H}(B)$ and $\mathrm{C}(B)$ for some concrete bialgebras $B$.

%\begin{invisible}[\rd{Added 19/06/24}]\end{invisible}
%The last example we are going  to investigate is given by a finite-dimensional bialgebra $B$ with $i_B$ surjective but not injective.

% To this aim we turn back to the quantum plane $\Bbbk _{q}\left[ x,y\right] ,$ with $yx=qxy$ and coalgebra
% structure given by $\Delta \left( x\right) =x\otimes x,\Delta \left(
% y\right) =x\otimes y+y\otimes 1.$
% Now  $I=\left\langle x^{3}-x,y^{2}\right\rangle $ is a bi-ideal in case $q=-1$.

\subsection{A non-commutative, non-cocommutative, finite-dimensional example}
\label{ssec:quotquant}

The first example we want to study is the bialgebra
\begin{equation*}
B=\Bbbk \left\langle x,y\mid yx=-xy,x^{3}=x,y^{2}=0\right\rangle
\end{equation*}
which is $6$-dimensional with basis $\left\{ 1,x,x^{2},y,xy,x^{2}y\right\} $. Its coalgebra structure is uniquely determined by $\Delta \left( x\right) =x\otimes x$ and $\Delta \left(y\right) =x\otimes y+y\otimes 1$, and the interested reader may recognize in it the bialgebra quotient of the quantum plane $\Bbbk _{-1}\left[ x,y\right]$ by the bi-ideal $I=\left\langle x^{3}-x,y^{2}\right\rangle $.
% quotient bialgebra $\Bbbk _{-1}\left[ x,y\right]/I$ i.e.
% \begin{equation*}
% B=\Bbbk \left\langle x,y\mid yx=-xy,x^{3}=x,y^{2}\right\rangle
% \end{equation*}%
% which is $6$-dimensional with basis $\left\{ 1,x,x^{2},y,xy,x^{2}y\right\} .$
%
\begin{invisible}
Indeed, note that
\begin{align*}
\Delta \left( y^{2}\right) & =\left( x\otimes y+y\otimes 1\right)
\left( x\otimes y+y\otimes 1\right) =x^{2}\otimes y^{2}+xy\otimes
y+yx\otimes y+y^{2}\otimes 1 \\
& =x^{2}\otimes y^{2}+xy\otimes y+qxy\otimes y+y^{2}\otimes 1=x^{2}\otimes
y^{2}+\left( q+1\right) xy\otimes y+y^{2}\otimes 1.
\end{align*}
If $q=-1$ this element belongs to $I\otimes A+A\otimes I.$ Moreover $\Delta
\left( x^{3}-x\right) =x^{3}\otimes x^{3}-x\otimes x=\left( x^{3}-x\right)
\otimes x^{3}+x\otimes \left( x^{3}-x\right) \in I\otimes A+A\otimes I.$
Finally, $\varepsilon \left( y^{2}\right) =0=\varepsilon \left(
x^{3}-x\right) .$
\end{invisible}
%
%In what follows $\left[ z\right] _{2}$ denotes  the congruence class of $z\in \mathbb{Z}$ modulo $2.$
%
It is also the quotient of the Ore extension $\Bbbk M[y,\varphi,\delta]$
by the ideal $\langle y^2\rangle$, with monoid $M=\langle x\mid x^{3}=x\rangle$, algebra endomorphism $\varphi \colon x\mapsto -x$ of $\Bbbk M$ and $\varphi$-derivation $\delta=0$, see \cite[\S 5.6]{DNC}.

%[Moved 22 07 2025] 
Note that $B$ is finite-dimensional so that $i_{B}$ is surjective (by \cref{prop:semiantip}). It is not injective, otherwise $B$ should be a right Hopf algebra by \cref{prop:Frobenius} and so $1=\varepsilon \left( x\right) =xS^{r}\left( x\right) $ would make $x$ right invertible, which is not the case.

Denote by $H_{4}=\Bbbk \left\langle x,y\mid yx=-xy,x^{2}=1,y^{2}=0\right\rangle $ the Sweedler's $4$-dimensional Hopf algebra.

\begin{proposition}%[\ps{Modified 03 03 2025}]
\label{prop:quotquant}
For $B=\Bbbk \left\langle x,y\mid
yx=-xy,x^{3}=x,y^{2}=0\right\rangle $ we have $B\oslash B \cong H_{4}$ as coalgebras and $H_{4} = \mathrm{H}(B)$.
\end{proposition}

\begin{proof}
%\ps{[Modified 22 07 2025]} %\ps{[Modified 03 03 2025]}
As $B$ is finite-dimensional, the coalgebra map $\widehat{q_B} \colon B\oslash B\to Q(B),x\oslash y\mapsto q_B(x)Sq_B(y)$ is an isomorphism by \cref{prop:HBfdArt}. Furthermore, in $B \oslash B$ we have
\begin{equation*}
x^{2}\oslash 1=x^{2}\oslash 1\varepsilon \left( x\right) =x^{2}x\oslash
x=x^{3}\oslash x=x\oslash x=1\oslash 1,
\end{equation*}
so that $x^2-1 \in \ker(i_B)$. Now, since $x^2$ is central in $B$, $\langle x^2-1 \rangle \subseteq B(x^2-1) \subseteq \ker(i_B)$, and since $\frac{B}{\langle x^2-1\rangle}\cong H_{4}$ as bialgebras, $\langle x^2-1\rangle$ is a bi-ideal. Hence, \cref{lem:HopfidealiB} entails that $\mathrm{H}(B) = Q(B) = \frac{B}{\langle x^2-1\rangle} \cong H_4$. 
\end{proof}

\begin{remark}
Keeping the notations as in the proof of \cref{prop:quotquant}, denote by $\pi$ the composition $B\overset{q_B}{\to} Q(B)\overset{\xi}{\to} H_4$, where $\xi$ is the obvious Hopf algebra isomorphism, i.e., $\pi \colon B\rightarrow H_{4},$ $\pi \left( x^{m}y^{n}\right)
=x^{m}y^{n}$. Then we can transport on $B\oslash B$ the Hopf algebra structure of $H_{4}$ in such a way that $\widehat{\pi }=\xi\circ \widehat{q_B }:B\oslash B\to H_4$ becomes a Hopf algebra map.
\begin{invisible}
Indeed, since $\xi$ is a Hopf algebra map, we have
 $\xi\widehat{q_B }(x\oslash y)=\xi(q_B(x)S_{Q(B)}q_B(y))=\xi q_B(x)\xi S_{Q(B)}q_B(y)=\xi q_B(x)S_{H_4}\xi q_B(y)
 =\widehat{\xi q_B }(x\oslash y)$ and hence $\xi\widehat{q_B }=\widehat{\xi q_B }=\widehat{\pi }.$  
\end{invisible}

Still, for $\mathrm{char}%
\left( \Bbbk \right) \neq 2,$ the canonical projection $p:B\otimes
B\rightarrow B\oslash B$ is not a bialgebra map otherwise $\alpha \coloneqq \widehat{%
\pi }\circ p:B\otimes B\rightarrow H_{4}$ would be an algebra map,
and this is not possible because $B \otimes 1$ centralises $1 \otimes B$ in $B \otimes B$, while both subalgebras are mapped onto the whole $H_4$, which is non-commutative. For example,
\[\alpha((1 \otimes x)(y \otimes 1)) = \alpha(y \otimes x) = \pi(y)S_{H_4}\pi(x) = yx \neq xy = \alpha(1 \otimes x)\alpha(y \otimes 1).\]
%
% \begin{align*}
% \alpha \left( \left( 1\otimes 1-x\otimes x\right) \left( y\otimes 1\right)
% \right)  &=\alpha ( y\otimes 1-xy\otimes x) =\pi \left( y\right)
% -\pi \left( xy\right) S_{H_4}\pi \left( x\right) =y-xyx=y+x^{2}y=2y \\
% \alpha ( 1\otimes 1-x\otimes x) \alpha( y\otimes 1)
% &=\left( 1-\pi \left( x\right) S_{H_4}\pi \left( x\right) \right) \left( \pi
% \left( y\right) \right) =\left( 1-x^{2}\right) y=0.  \qedhere
% \end{align*}
\end{remark}

In this case we are also able to explicitly describe the $n$-Hopf algebra structure of $B$.

\begin{proposition}
The bialgebra $B=\Bbbk \left\langle x,y\mid
yx=-xy,x^{3}=x,y^{2}=0\right\rangle $ has a two-sided invertible $1$-antipode $S:B\rightarrow B$ which is an anti-algebra map and it is defined by setting $S\left(
x\right) =x$ and $S\left( y\right) =\left( 1-x-x^{2}\right) y.$
\end{proposition}

\begin{proof} %\ps{[Modified 03 03 2025]}
Since $B$ is finite-dimensional, $B$ has two sided $n$-antipode such that $S\ast \mathrm{Id}=\mathrm{Id}\ast S$ for some $n$, by \cref{coro:fdS}. Furthermore, $n>0$ because, as we already observed above, $B$ is not a right Hopf algebra.
%, as we already observed in the proof of \cref{prop:quotquant}.
We want to prove that $n=1.$
Since $\Delta(x) = x \otimes x$ and $\Delta(y) = x \otimes y + y \otimes 1$, a moment of thought shall convince the reader that $\id^{*n}(x) = x^n$ and  $\id^{*n}(y) = \sum_{k=0}^{n-1}x^ky$ for all $n \geq 1$.
\begin{invisible}
Indeed,
\[\id^{*(n+1)}(y) = x \cdot \id^{*n}(y) + y\]
and we can conclude by induction.
\end{invisible}
In particular,
\begin{gather*}
\id^{*3}(x) = x^3 = x, \qquad \id^{*3}(y) = x^2y + xy + y, \qquad \id^{*5}(x) = x^5 = x \\
\text{and} \qquad \id^{*5}(y) = x^4y+x^3y+x^2y+xy+y = 2x^2y+2xy+y
\end{gather*}
entail that
\[\left(\left(2 \cdot \id - \id^{*3}\right)*\id^{*2}\right)(x) = x \qquad \text{and} \qquad \left(\left(2 \cdot \id - \id^{*3}\right)*\id^{*2}\right)(y) = y.\]
Set $S \coloneqq 2 \cdot \id - \id^{*3}$. It satisfies
\[S(1) =1, \qquad S(x) = x \qquad \text{and} \qquad S(y) = (1-x-x^2)y.\]
One may check directly that the unique algebra map $T':\Bbbk \left\langle X,Y\right\rangle \rightarrow B^{\mathrm{op}}$ satisfying $T'\left( X\right) =x$ and $T'\left( Y\right) =\left( 1-x-x^{2}\right) y$ factors uniquely through an algebra map $T \colon B \to B^{\mathrm{op}}$.
\begin{invisible}
Indeed,
\begin{align*}
T'\left( YX\right) &=T'\left( X\right) T'\left( Y\right) =x\left(
1-x^{2}-x\right) y=\left( 1-x^{2}-x\right) xy=-\left( 1-x^{2}-x\right)
yx \\
& =-T'\left( Y\right) T'\left( X\right) =T'\left( -XY\right) \\
T'\left( X^{3}\right) &=T'\left( X\right) ^{3}=x^{3}=x=T'\left( X\right) \\
T'\left( Y^{2}\right) &=T'\left( Y\right) ^{2}=\left( 1-x^{2}-x\right)
y\left( 1-x^{2}-x\right) y\in By^{2}=0
\end{align*}

Thus $T'$ induces an algebra map $T:B\rightarrow B^{\mathrm{op}}$ such that $T\left( x\right) =x$ and $T\left( y\right) =\left( 1-x^{2}-x\right) y.$
\end{invisible}
Since the latter satisfies
\begin{align*}
T\left( x^{2}\right) & = T\left( x\right) ^{2}=x^{2} = 2x^2 - x^6 = S(x^2), \\
T\left( xy\right) & = T\left( y\right) T\left( x\right) = \left(1-x-x^{2}\right) yx = \left( -x + x^{2} + x^3\right)y = x^2y = 2xy - (2xy - x^2y) = S(xy),  \\
T\left( x^{2}y\right) & = T\left( xy\right) T\left( x\right) = x^{2}yx = -xy = 2x^2y - (xy + 2x^2y) = S(x^2y),
\end{align*}%
\begin{invisible}
    indeed,
    \begin{align*}
    \id^{*3}(xy) & = \id^{*2}(x^2)xy + \id^{*2}(xy)x = x^5y + (x^3y + xyx)x = xy-x^2y+xy = 2xy - x^2y, \\
    \id^{*3}(x^2y) & = \id^{*2}(x^3)x^2y + \id^{*2}(x^2y)x^2 = x^8y + (x^5y + x^2yx^2)x^2 = x^2y + xy + x^2y = xy + 2x^2y
    \end{align*}%
\end{invisible}%
we conclude that $S = 2 \cdot \id - \id^{*3}$ is an anti-algebra map and a left $1$-antipode which additionally satisfies $S * \id = \id * S$, constructed as in the proof of \cref{coro:fdS}.
To prove that $S$ is invertible, too, it suffices to observe that the above computations show that it is surjective. More precisely,
\[
S^2(y) = S(y)(1-x^2-x) = (1-x^2-x)y(1-x^2-x) = (1-x^2-x)(1-x^2+x)y = (1-2x^2)y
\]
\begin{invisible}
indeed $(1-x^2)^2-x^2 = 1 - 2x^2 + x^4 - x^2 = 1 - 2x^2 + x^2 - x^2 = 1-2x^2$,
\end{invisible}
entails that
$S^4(y) = S^2(1-2x^2)S^2(y) = (1-2x^2)^2y = y$
\begin{invisible}
    indeed $(1-2x^2)^2 = 1 - 4x^2 + 4x^4 = 1-4x^2+4x^2 = 1$
\end{invisible}
and hence $S^4 = \id$.
\end{proof}

% \rd{[added on 2024-12-09]}
We now deal with $\mathrm{C}(B)$.

\begin{proposition}
\label{prop:Cquotquant}
For $B=\Bbbk \left\langle x,y\mid
yx=-xy,x^{3}=x,y^{2}=0\right\rangle $ we have $\mathrm{C}(B) = \Bbbk$.
\end{proposition}

\begin{proof}
By \cref{prop:CBfd}, $\mathrm{C}(B)=K(B)$ and hence it is a sub-bialgebra of the pointed bialgebra $B$, whence pointed as well.
\begin{invisible}
We observed that $B$ has basis $\left\{ 1,x,x^{2},y,xy,x^{2}y\right\} $. Thus, the coalgebra structure of $B$ entails that $\Bbbk G(B)=\Bbbk+\Bbbk x+\Bbbk x^2\subseteq B$ is an exhaustive coalgebra filtration so that the coradical $B_0$ of $B$ is contained in $\Bbbk G(B)$ and hence $B_0=\Bbbk G(B)$.
%Therefore, by \cite[Proposition 3.4.3(d)]{Radford-book}, the coradical is $\mathrm{C}(B)_0=B_0\cap \mathrm{C}(B)=\Bbbk G(B)\cap \mathrm{C}(B)\subseteq\Bbbk G(B)$.
Now $\mathrm{C}(B)_0$ is the direct sum of the simple sub-coalgebras of $\mathrm{C}(B)$. Since any simple sub-coalgebra of $\mathrm{C}(B)$ is a simple sub-coalgebra of $B$ and $B_0=\Bbbk G(B)$, it is one-dimensional.
\end{invisible} Since $\mathrm{C}(B)$ is a Hopf algebra, its coradical is spanned by invertible grouplike elements. There is only one such element in $B$, namely $1$, so that the coradical of $\mathrm{C}(B)$ is $\Bbbk$.
Then, if we denote, as usual, by $\mathrm{C}(B)_n$ the $n$-th term of the coradical filtration of $\mathrm{C}(B)$, we have that $\mathrm{C}(B)_1=\Bbbk +P(\mathrm{C}(B))$, in view of \cite[Lemma 5.3.2]{Montgomery}. Now it is easy to check that $P(B)=\{0\}$.
\begin{invisible}
 Let $z\in P\left( B\right) .$ Since $\left\{ 1,x,x^{2},y,xy,x^{2}y\right\} $
is a basis for $B,$ we can write $%
z=k_{0}1+k_{1}x+k_{2}x^{2}+h_{0}y+h_{1}xy+h_{2}x^{2}y$ for some $%
k_{i},h_{i}\in \Bbbk $. Recalling that $\Delta \left( x\right) =x\otimes x$
and $\Delta \left( y\right) =x\otimes y+y\otimes 1,$ we compute%
\begin{align*}
\Delta \left( z\right)  &=k_{0}1\otimes 1+k_{1}x\otimes x+k_{2}x^{2}\otimes
x^{2}+h_{0}\left( x\otimes y+y\otimes 1\right) +h_{1}\left( x^{2}\otimes
xy+xy\otimes x\right) +h_{2}\left( x^{3}\otimes x^{2}y+x^{2}y\otimes
x^{2}\right)  \\
=&k_{0}1\otimes 1+k_{1}x\otimes x+k_{2}x^{2}\otimes x^{2}+h_{0}x\otimes
y+h_{0}y\otimes 1+h_{1}x^{2}\otimes xy+h_{1}xy\otimes x+h_{2}x\otimes
x^{2}y+h_{2}x^{2}y\otimes x^{2} \\
=&\left( k_{0}1+h_{0}y\right) \otimes 1+\left( k_{1}x+h_{1}xy\right)
\otimes x+\left( k_{2}x^{2}+h_{2}x^{2}y\right) \otimes x^{2}+h_{0}x\otimes
y+h_{1}x^{2}\otimes xy+h_{2}x\otimes x^{2}y.
\end{align*}%
This element must be compared with
\begin{eqnarray*}
z\otimes 1+1\otimes z &=&z\otimes 1+k_{0}1\otimes 1+k_{1}1\otimes
x+k_{2}1\otimes x^{2}+h_{0}1\otimes y+h_{1}1\otimes xy+h_{2}1\otimes x^{2}y
\\
&=&\left( z+k_{0}1\right) \otimes 1+k_{1}1\otimes x+k_{2}1\otimes
x^{2}+h_{0}1\otimes y+h_{1}1\otimes xy+h_{2}1\otimes x^{2}y.
\end{eqnarray*}%
By matching the terms of the form $\left( \ast \right) \otimes 1$ we get the
condition $k_{0}1+h_{0}y=z+k_{0}1$ so that $z=h_{0}y$. Then, by matching the
terms of the form  $\left( \ast \right) \otimes y,$ we get the condition $%
h_{0}x=h_{0}1$ so that $h_{0}=0$ and hence $z=0.$
\end{invisible}
Since $P(\mathrm{C}(B))\subseteq P(B)$, we get that $\mathrm{C}(B)_1=\Bbbk =\mathrm{C}(B)_0$ and hence $\mathrm{C}(B)=\Bbbk$.
\end{proof}

%\marginpar{\tiny\ps{I reorganized a bit the material in this subsection.}}
\subsection{The monoid bialgebra case and additional results on \texorpdfstring{$i_B$}{iB} and \texorpdfstring{$p_B$}{pB}}\begin{invisible}[\ps{Added 04 03 2025 - the results herein are adaptations of the already existing ones}]\end{invisible}%
Let $M$ be a multiplicative monoid and let $B = \Bbbk M$ be the associated monoid bialgebra. By uniqueness of the left adjoint, the Hopf envelope $\mathrm{H}(\Bbbk M)$ is the group algebra of the enveloping group $G(M)$ of $M$, called the \emph{universal enveloping group} in \cite[\S4.11]{Bergman-book} (see e.g.\ \cite[Example 4.1]{Chirv}).
%It is already known (see e.g.\ \cite[Example 4.1]{Chirv}) that its Hopf envelope $\mathrm{H}(\Bbbk M)$ is the group algebra of the enveloping group $G(M)$ of $M$, \ps{called the \emph{universal enveloping group} in \cite[\S4.11]{Bergman-book}:
\begin{invisible}
Since the grouplike-elements functors $\mathcal{G} \colon \Hopf \to \Grp$ and $\mathcal{G} \colon \Bialg \to \Mon$ are right adjoints of the linearization functors $\Bbbk - \colon \Grp \to \Hopf$ and $\Bbbk - \colon \Mon \to \Bialg$, respectively, and since the underlying functors $\mathcal{U} \colon \Grp \to \Mon$ and $\mathcal{V} \colon \Hopf \to \Bialg$ clearly satisfy $\mathcal{U} \circ \mathcal{G} = \mathcal{G} \circ \mathcal{V}$, we conclude that $\mathrm{H}\left(\Bbbk M\right) \cong \Bbbk G(M)$ by uniqueness of the left adjoint.
\end{invisible}
Hence, we start by computing $\mathrm{C}(B)$. To this aim, denote by $M^\ell \coloneqq \{g\in M \mid \exists h\in M, gh=1\}$ the sub-monoid of $M$ consisting of the left units in $M$ and recall the iterative construction from \cref{ssec:itercofree}.

\begin{lemma}%[\ps{Extracted 04 03 2025}]
\label{lem:KkM}
For $B = \Bbbk M$, $B\boxslash B = \mathsf{span}_{\Bbbk}\{g\otimes h\mid g,h\in M,gh=1\}$ and $\im(p_B) = \Bbbk M^\ell = K(B)$.
\begin{invisible}\rd{[Non so se è rilevante, ma $B\boxslash B=\Bbbk N$ dove $N=\{g\otimes h\mid g,h\in M,gh=1\}$ è un monoide col prodotto $(g\otimes h)(g'\otimes h')=gg'\otimes h'h$, cioè un sottomonoide del monoide moltilicativo $B\otimes B^\op$.]}\end{invisible}
\end{lemma}

\begin{proof}
 From the definition, it follows easily that $B\boxslash B$ is spanned, as a vector space, by the set $\{g\otimes h\mid g,h\in M,gh=1\}.$
 \begin{invisible}
 Let $\sum_{g,h}k_{g,h}g\otimes h\in B\boxslash B$ with $k_{g,h}\neq 0$. Then $\sum_{g,h}k_{g,h}g\otimes h\otimes gh=\sum_{g,h}k_{g,h}g\otimes h\otimes 1$ so that, by applying $g^*\otimes h^*\otimes \id$, where $g^*(h)=\delta_{g,h}$, we get $k_{g,h} gh=k_{g,h}1$ and hence $gh=1$.
  \end{invisible}
 Thus $\im(p_B)=\Bbbk M^\ell$. Therefore, $K(B)=\Bbbk M^\ell$ in view of \cref{lem:KBcoalg}, as $\Bbbk M^\ell$ is a subcoalgebra of $B$.
 \begin{invisible}
 The inclusion $K(B)\subseteq \im(p_B)$ is always true.
 % If $\sum k_gg\in K(B)$ with $k_g\neq0$, then $\sum k_gg\otimes g\in \Bbbk M^\ell\otimes \Bbbk M$ and hence $g\in M^\ell$.
  If $g\in M^\ell$, then $\Delta(g)=g\otimes g\in \Bbbk M^\ell\otimes \Bbbk M$ and hence $\im(p_B)\subseteq K(B).$
 \end{invisible}
\end{proof}

\begin{remark}
If $M$ is finite, then any element in $M^\ell$ is also left invertible and hence $M^\ell=M^\times\coloneq\{g\in M\mid \exists h\in M, gh=1=hg\}$ is a group, which confirms that $K(B)$ is a Hopf algebra in the finite-dimensional case.
\begin{invisible}
 Given $g\in M^\ell$, then $\{g^n\mid n\in\mathbb{N}\}$ is finite, so there are $a>b$ such that $g^a=g^b$. Since $g$ right invertible, this equality implies we can cancel $g^b$ on the right to get $g^{a-b}=1$ whence $g$ is two-sided invertible.
 \end{invisible}
\end{remark}

\begin{theorem}
\label{thm:monoidbialgebra}
    For the monoid bialgebra $B = \Bbbk M$, $\mathrm{C}(B) = K^2(B) = \Bbbk M^\times$ and $\mathrm{C}(B) = \mathrm{C}^c(B)$.
\end{theorem}

\begin{proof}
By \cref{lem:KkM}, $K(B) = \Bbbk M^\ell$. For an arbitrary monoid $M$, one has $(M^\ell)^\ell=M^\times$ so that $K^2(B)=K(K(B))=K(\Bbbk M^\ell)=\Bbbk (M^\ell)^\ell=\Bbbk M^\times$, which is a Hopf algebra.
 \begin{invisible}
If $g\in (M^\ell)^\ell$, then there is $g'\in M^\ell$ such that $gg'=1$ but $g'\in M^\ell$ implies there is $g''\in M$ such that $g'g''=1$ so that $g'$ is two-sided invertible and hence its left and right inverse coincide, i.e. $g=g''$, and so $g\in M^\times$. Thus $(M^\ell)^\ell\subseteq M^\times$. The other inclusion is true.
 \end{invisible}
 As a consequence $K^2(B)=\mathrm{C}(B)$, by \cref{prop:pconvd}\,\ref{item:pconvd2}. The last claim follows by cocommutativity of $\mathrm{C}(B)$.
\end{proof}

\begin{remark}\label{rem:epsiMinj}
     For $B = \Bbbk M$, the canonical map $\epsilon_B:\mathrm{C}(B)\to B$ is given by the inclusion $K^2(B)\subseteq B$, so that it is injective. Moreover, $\mathrm{C}(B) = K(B)$ if and only if $M^\ell = M^\times$.
\end{remark}

Since monoid bialgebras offer a particularly rich family of examples, we can take advantage of them to provide additional evidence or to exhibit counterexamples to the converses of certain claims from the previous sections. To this aim, let us start with a complete characterization of injectivity and surjectivity of $p_B$ and of injectivity of $i_B$, followed by a few additional results on the surjectivity of $i_B$ for a monoid bialgebra $B = \Bbbk M$.

\begin{proposition}
\label{lem:pM}
Let $B=\Bbbk M$. Then
\begin{enumerate}[label=\alph*),ref={\itshape \alph*)},leftmargin=*]
\item\label{item:pM1} $p_B$ is surjective if, and only if, the monoid $M$ is a group;
\item\label{item:pM2} $p_B$ is injective if, and only if, any element in $M$ has at most one right inverse;
\item\label{item:pM3} $i_B$ is injective if, and only if, $M$ is right cancellative.
\end{enumerate}
% \rd{[Sposterei qui la prima frase dell'enunciato di \cref{cor:localiM}, per coerenza.\\ Forse potremmo promuovere questo lemma a proposizione?]}\ps{[Ok per la promozione, ma la prima frase del corollario non ha lo stesso sapore di queste. Non è questo il motivo per cui non affermiamo di avere un risultato esaustivo anche per la suriettività?]} \rd{[Convinto!]}
\end{proposition}

\begin{proof}
\ref{item:pM1} Since $\im(p_B)=\Bbbk M^\ell$ by \cref{lem:KkM}, the surjectivity of $p_B$ is equivalent to $M=M^\ell$ i.e.\ to $M$ being a group.

\ref{item:pM2} If $p_B$ is injective and we take $a,b,c\in M$ with $ab=1=ac$, then $a\otimes b\in B\boxslash B\ni a\otimes c$ and we have $p_B(a\otimes b)=a=p_B(a\otimes c)$ so that $a\otimes b=a\otimes c$ and hence $b=c$. Thus an element in $M$ has at most one right inverse. Conversely, suppose that every $a\in M^\ell$ has a unique right inverse in $M$, say $\tau(a)$. This induces a linear map $\tau \colon \im(p_B)\to B$ which in turn allows us to define a linear map $f \colon \im(p_B) \to B\boxslash B$ by setting $f(a) \coloneqq a\boxslash \tau(a)$. The latter turns out to be injective, since $p_B \circ f =s$ where $s:\im(p_B)\hookrightarrow B$ is the inclusion, and surjective, in view of \cref{lem:KkM}. Whence $p_B=s\circ f^{-1}$ and so $p_B$ itself is injective, as claimed.
%If $z\in \ker(p_B)$ then $z\in B\boxslash B$ and hence we can write it in the form $z=\sum_{a\in M^\ell}k_aa\boxslash \tau(a)=f(\sum_{a\in M}k_aa)=f(p_B(z))=0$ so that  $p_B$ is injective as claimed.

 \ref{item:pM3} Suppose that $i_B$ is injective and that $ac = bc$ in $M$. Then, by taking advantage of the left $B\otimes B$-module structure of $B\oslash B$ from \cref{lem:oslash},
 \[i_B(a) = a \oslash 1_B = ac \oslash c %= (1_B \otimes c)(ac \oslash 1_B) = (1_B \otimes c)(bc \oslash 1_B) 
 = bc \oslash c = b\oslash 1_B = i_B(b),\]
  from which we conclude that $a=b$. Conversely, suppose that $M$ is right cancellative. If we can prove that for $a,b \in M$, $i_B(a) = i_B(b)$ implies $a=b$, then it will follow that $i_B$ is injective: %\rd{[E' così chiaro?]}{\color{green}[In un certo senso sì: il motivo è quello che segue i due punti. Se una mappa lineare manda una base in una famiglia libera, allora deve essere iniettiva. In dimensione finita, è il teorema di nullità più rango: $\dim(V) = \dim(\ker(f)) + \dim(\im(f))$. Nell'altro file lo abbiamo considerato così evidente che non abbiamo neanche aggiunto la spiegazione, ma siccome ci ho messo un po' a recuperarla, ho pensato valesse la pena aggiungerla]}:
 in such a case, the elements $a \oslash 1_B$ in $B \oslash B$ would be distinct grouplike elements, and so linearly independent.
Since $M$ is right cancellative, the Kronecker delta obeys the identity $\delta_{ac,bc}=\delta_{a,b}$ for every $a,b,c\in M$. As a consequence, the linear map $f:B\otimes B\to\Bbbk$, given on the basis by $f(a\otimes b)\coloneqq\delta_{a,b}$, for every $a,b\in M$,
induces a linear map $\overline{f}:B\oslash B\to \Bbbk$.

Now, if $a,b \in M$ are such that $i_B(a) = i_B(b)$, then
\[a \oslash b = (1_B \otimes b)(a \oslash 1_B) = (1_B \otimes b)(b \oslash 1_B) = b \oslash b = 1_B \oslash 1_B,\]
so that $\delta_{a,b}=\overline{f}(a\oslash b)=\overline{f}(1_B\oslash 1_B)=\delta_{1_B,1_B}=1$ and hence $a=b$ completing the proof.
%i.e., $1_B \otimes 1_B - a \otimes b \in (B \otimes B)B^+$. By \ref{item:techstep2}, $a = b$ and the proof is complete.
\end{proof}

\begin{remark}
%\rd{[added on 2025-04-03]}
 Let $B=\Bbbk M$ be a monoid algebra. By construction $Q^\infty(B)$  is a quotient bialgebra of $B$ and hence it is a monoid algebra $\Bbbk C$ where the monoid $C$ is a homomorphic image of $M$. By \cref{thm:Kellyfree}, $Q^\infty(B)=\Bbbk C$ has $i_{Q^\infty(B)}$ injective and hence by \cref{lem:pM}, $C$ is right cancellative. The universal property of $Q^\infty(B)$ together with \cref{lem:pM} now entail that $C$ is the \emph{maximal right cancellative monoid homomorphic image}\footnote{We are grateful to Davide Ferri for bringing this notion to our attention.} of $M$, see \cite[bottom of page 133]{Masat}. Thus the construction of $Q^\infty(B)$ can be regarded as a bialgebra counterpart of the construction of the maximal right cancellative monoid homomorphic image.
\end{remark}

It is well-known that, for a general monoid $M$, the canonical monoid morphism $M \to G(M)$ to the enveloping group is not surjective. E.g., $\N \subseteq \Z$ is not. The following result shows that, instead, this is the case for certain monoids which are generated by suitable elements. In particular, it entails that these monoids admit a \emph{maximal group homomorphic image}, see \cite[page 18]{CP}.

%It is well-known that a semigroup needs not to admit a \emph{maximal group homomorphic image}, see \cite[page 18]{CP}. Still any inverse semigroup possesses a maximal group homomorphic image, see \cite[\S 7.7]{CP2}. In the following result we show this is also true for monoids generated by suitable elements.

\begin{proposition}%[\ps{Updated 01 04 2025}]
%\rd{[E' un corollario di \cref{prop:localiB} che però è lontana. Lo farei diventare una prop.]}\ps{[assolutamente d'accordo]}
\label{cor:localiM}
%\rd{[Updated on 2025/03/29]}
Let $B = \Bbbk M$ for a monoid $M$. Then $i_B$ is surjective if and only if for every $x \in M$, there exists $y \in M$ such that $1 \oslash x = y \oslash 1$. In particular, if $M$ is generated as a monoid by a subset $\{x_i\mid i \in I \}$ such that $1\oslash x_i=y_i\oslash 1$ for some $y_i\in M$, then $\mathrm{H}(B)=Q(B)=B/\ker(i_B)$ where $\ker(i_B) = \langle y_ix_i-1,x_iy_i-1\mid i\in I\rangle$ and the enveloping group $G(M)$ of $M$ can be realised as the multiplicative group generated by $\left\{x_i+\ker(i_B)\mid i\in I\right\}$ in $B/\ker(i_B)$.
%Then $i_{B}$ is surjective and $\mathrm{H}(B)=Q(B)=B/\ker(i_B)$ where $\ker(i_B) = \langle y_ix_i-1,x_iy_i-1\mid i\in I\rangle$. Moreover, the enveloping group $G(M)$ of $M$ can be realised as the multiplicative group generated by $\left\{x_i+\ker(i_B)\mid i\in I\right\}$ in $B/\ker(i_B)$}.
%$\mathrm{H}(B) = \Bbbk G(M)$ is the group algebra on the enveloping group $G(M) \cong \langle x_i+\ker(i_B)\mid i\in I\rangle$ of $M$.
%In particular, t
The above assumption holds true if $y_ix_iz_i=z_i$ for some $y_i,z_i\in M$ for every $i \in I$.
%\ps{\sout{(e.g. $x_i\in\mathcal{P}(y_i)$, or $x_i\in\mathcal{J}_1$, or $x_i$ is regular, or $x_i$ is periodic, for every $i \in I$).}}
\end{proposition}

\begin{proof}
%[Per includere anche Esempio 4.16, ho preso la condizione di lem:eureka di Adjmon.]
%[Dimostro qui quanto dicevo sopra sulla suriettività, casomai volessimo intergrare:\\
If for every $x \in M$ we have that $1\oslash x = y\oslash 1$ for some $y\in M$, then $i_B$ is surjective by \cref{prop:localiB}. Conversely, suppose that $i_B$ is surjective. For every $x\in M$ there is $w\in B$ such that $1\oslash x = i_B(w) = w\oslash 1$. Write $w$ as a linear combination $w = \sum_{g\in M}k_gg$ with finitely many $k_g\in \Bbbk$ non-zero.
%(and not all of them zero \rd{[non è superfluo?]}\ps{[Sì, lo penso anch'io. L'avevo salvato dalla formulazione precedente per sottolineare che $1 \oslash x$ non può essere $0$]}).
Then $1\oslash x=\sum_{g\in M}k_gg\oslash 1.$ By linear independence of grouplike elements, there is $y\in M$ with $1\oslash x=y\oslash 1$. This proves the first claim.

%The surjectivity of $i_B$ is guaranteed by \cref{prop:localiB}.
If $b\oslash 1 = 1\oslash a$, then, by employing the ``flip map'' $\sigma \colon B\oslash B\to B\oslash B$ considered in the proof of \cref{pro:cocom}, we have
\[ba\oslash 1 = (b\otimes 1)(a\oslash 1) = (b\otimes 1)\sigma (1\oslash a) = (b\otimes 1)\sigma (b\oslash 1) = (b\otimes 1)(1\oslash b) = b\oslash b = 1\oslash 1\]
so that $ba-1\in K\coloneqq \ker(i_B)$, but we also have that $a\oslash 1=\sigma(1\oslash a)=\sigma(b\oslash 1) = 1\oslash b$ and hence, by symmetry, $ab-1\in K$, too.
% \begin{align*}
%  ba\oslash 1&=(b\otimes 1)(a\oslash 1)=(b\otimes 1)\sigma (1\oslash a)=(b\otimes 1)\sigma (b\oslash 1)
% =(b\otimes 1)(1\oslash b)=b\oslash b=1\oslash 1,\\
% ab\oslash 1&=(a\otimes 1)(b\oslash 1)=(a\otimes 1)(1\oslash a)=a\oslash a=1\oslash 1,
% \end{align*}
% so that $ba-1,ab-1\in K\coloneqq \ker(i_B)$.
% %[Per applicare concretamente la costruzione ho ridotto all'osso le relazioni che imponiamo in $J$.]\ps{[Ci avevo pensato anch'io]}
In particular, we have that $y_ix_i-1,x_iy_i-1\in K$.
Since $K$ is a two-sided ideal  of $B$ (by \cref{pro:cocom}, too), we have that $J\coloneqq\langle y_ix_i-1,x_iy_i-1\mid i\in I\rangle\subseteq K$. 

Now, the bialgebra structure is such that $J$ is a bi-ideal and $B/J$ is spanned, as a vector space, by words in the alphabet $\{x_i + J\mid i \in I\}$ which are grouplike and invertible. Thus, $B/J$ is a Hopf algebra. By \cref{lem:HopfidealiB}, 
% \rd{[Qui si potrebbe accorciare usando \cref{lem:HopfidealiB} tenendo però la dim che $B/J$ è Hopf.]}
%
% and since $J$ is a bi-ideal, we also know that $B/J$ is a bialgebra, projecting onto $B/K = \mathrm{H}(B)$ (see \cref{pro:cocom} again). Now, $B/J$ is spanned, as a vector space, by words in the alphabet $\{x_i + J\mid i \in I\}$. The bialgebra structure is such that any of these words is grouplike and invertible. Thus, it is a Hopf algebra
% %\rd{[Mi hai anticipato! :)]}
% and hence we can apply \cref{lem:iB-inj}\,\ref{ib-inj_item2} to the projection $\pi:B\to B/J$ to conclude that $K \subseteq J$. As a consequence, 
$\mathrm{H}(B)=Q(B)=B/K=B/J$ and $\mathrm{H}(B)=\Bbbk G$ where $G=G(\mathrm{H}(B))$ is the multiplicative group generated by $\{x_i + J\mid i \in I\}$ inside $B/J$. The universal property of $\mathrm{H}(B)$ entails that $G$ is the enveloping group of the monoid $M$.

Concerning the last assertion, if $bac=c$, then $1\oslash a=c\oslash ac
=bac\oslash ac=b\oslash 1$ and the main assumption is satisfied.
%\ps{\sout{The e.g. part of the statement holds by Lemma 4.12.}}%
%Suppose that $M$ is generated by elements $\{x_i \mid i \in I\}$ which are left or right invertible in $M$. Then $B$ is generated, as an algebra, by the same set. If there exists $y_i \in M$ such that $x_iy_i = 1$, then
%$1 \oslash x_i =  y_i \oslash x_iy_i = y_i \oslash 1.$
%If instead there exists $z_i \in M$ such that $z_ix_i = 1$, then
%$1 \oslash x_i = z_ix_i\oslash x_i = z_i \oslash 1.$
%In both cases, we conclude by \cref{prop:localiB}.
\end{proof}

 We say that an element $a$ in a semigroup $N$ is \emph{periodic} if $a^i=a^{i+p}$ for some positive integers $i,p$. A monoid $M$ generated by periodic elements, such as a locally finite monoid as in \cref{exa:lfmon} or a monoid as in \cref{exa:notArtLoc} and \cref{exa:periodmon} below, always satisfies $M^\ell = M^\times$ and hence $\mathrm{C}(\Bbbk M) = K(\Bbbk M)$ in view of \cref{rem:epsiMinj}. %[Forse ne farei un remark!]}
 \begin{invisible}
 Indeed, suppose that an element $z$ satisfies $zy=1$. Let $x$ be the left-most letter in $z$, i.e. $z = xz'$. Then $1 = zy = xz'y$ entails that $x$ is right invertible. But it is also periodic, whence $x^{p+k} = x^p$ for some $p,k$, from which $x^k = x^{p+k}(z'y)^{p} = x^p(z'y)^{p} = 1$ and so $z'y = x^{k-1} = x^{-1}$. Write then $z = x^tuz''$ where $u$ is another letter. Then $uz''yx^t = x^{-t}x^t = 1$ and so we can repeat the argument to show that $u$ is invertible, too. After a finite number of steps, we conclude that $z$ is a product of powers of invertible generators and so it is invertible itself.
 \end{invisible}
 Recall from \cite[page 19]{CP} that if $i,p$ are such that $i+p$ is the minimal positive integer for which $a^{i+p}$ equals $a^i$, then we call them the \emph{index} and the \emph{period} of $a$ respectively, and from \cite[\S1.9]{CP} that an element $b$ in $N$ is called \emph{regular} if there exists an element $a\in N$ such that $bab = b$, see \cref{ex:regmon}. In the latter case, we say that $a$ is a \emph{pseudoinverse} of $b$. We denote by $\mathcal{P}(b)$ the set of pseudoinverses of $b$.
Recall also the \emph{Green's relation} $  \mathcal{J}$ on a monoid $M$ defined, for $a,b\in M$, by setting $a\mathcal{J}b$ if $MaM= MbM$, see \cite[\S 2.1]{CP}. Denote by $\mathcal{J}_a$ the equivalence class of $a$.
%An element $a$ is a semigroup is called \emph{periodic} if $a^i=a^{i+p}$ for some  positive integers $i,p$ called the \emph{index} and the \emph{period} of $a$, see \cite[page 19]{CP}.

\begin{lemma}
%\rd{[added on 2024-03-28]}
\label{lem:eurekaenv}
Let $M$ be a monoid. Any pseudoinverse $a \in M$ obeys the identity $bac=c$ for some $b,c\in M$ as in \cref{cor:localiM}. This happens in particular if $a\in\mathcal{J}_1$, or $a$ is regular, or $a$ is periodic.
\end{lemma}

\begin{proof}
If $a\in\mathcal{P}(b)$, then $bab=b$ so that  $bac=c$ for $c=b$.

We have that $a\in\mathcal{J}_1$ if and only if $MaM=M1M=M$, if and only if there are $a',a''\in M$ with $1=a'aa''$. In this case, $b\coloneqq a''a'$ satisfies $bab=b$ and so $a\in\mathcal{P}(b)$.

If $a$ is regular with pseudoinverse $b$, then $a$ is the pseudoinverse of $bab$, %\sout{$(bab)a(bab) = b((aba)ba)b = b(aba)b = bab$ and so $a$ is a pseudoinverse}
see \cite[Lemma 1.14]{CP}.

Suppose that $a$ is periodic with index $i$ and period $p$. Then $a^{i+p} = a^i$. Pick a positive integer $k$ such that $kp-1\geq i$, e.g.~$k\coloneqq\lceil\frac{i+1}{p}\rceil$ by using the ceiling function. Then
\[a^{kp-1}aa^{kp-1} = a^{2kp-1}=a^{kp-1-i}a^{i+kp}=a^{kp-1-i}a^{i}=a^{kp-1}\]
and so $a$ in the pseudoinverse of $a^{kp-1}$.
%\rd{[Ho perso un bel po' di tempo per non riuscire a farlo.]}\ps{[o.O Non ho capito]} \rd{[TRADUCO: Nonostante i vari tentativi di dimostrare che periodico implichi pseudoinverso, io non ero riuscito. Tu invece si. Complimenti! :)]}\ps{[Grazie <3 In realtà, la mia è solo testardaggine 0:) ;) ]}
%If $a$ is periodic. i.e. $a^i=a^{i+p}$, then $a ^{p-1}aa ^{i}
%=a ^{i+p}=a ^{i}$ so that $bac=c$ for $b=a ^{p-1},c=a ^{i}$.
\end{proof}

%\rd{[Potremmo migliorare l'enunciato precedente accorciandone al contempo la dimnostrazione nel modo seguente}. \ps{[mi piace!]} \rd{[Fatto].}

As remarked after \cref{cor:CBsub}, for a bialgebra $B$ a sufficient condition to claim that $\mathrm{C}(B)$ is a sub-bialgebra of $B$ itself is to be weakly finite. In fact, under the latter hypothesis, in \cite[Theorem 2.3]{Ery-Skry} $\mathrm{C}(B)$ is constructed as a suitable sub-bialgebra of $B$, even if this construction is not explicit.
We now provide an instance of a bialgebra $B$ with $p_B$ injective which is not weakly finite.

\begin{example}%[\ps{Merged with the former 4.14 on 02 04 2025}]
\label{exa:notweakpBinj}
%\rd{[Added on 2025/03/24]}
Let $M$ be a monoid with $ M^\times\subsetneqq M^\ell$. Then there are $g,h\in M$ such tat $gh=1$ but $hg\neq 1$ and hence $B=\Bbbk M$ is not weakly finite. Still, by \cref{lem:pM}, $p_B$ is injective if any element in $M$ has at most one right inverse.

For instance, take the free monoid $M=\langle a,b\mid ab=1 \rangle$ generated by $\{a,b\}$ and subject to the relations $ab=1$. Then $M^\times=\{1\}$ and $M^\ell=\langle a\rangle$ so that  $p_B$ is injective, but the bialgebra $B=\Bbbk M$ is not weakly finite (this bialgebra is the so-called \emph{Jacobson algebra}, which can be regarded as the Leavitt path algebras of the Toeplitz graph, see \cite[Example 7]{Abrams}).
\begin{invisible}
The elements of $M$ are of the form $b^ia^j$. If $b^ia^j\in M^\ell$, then there is $b^sa^t$ such that $b^ia^jb^sa^t=1$. Since $b$ is not invertible, we must have $i=0$ (otherwise $b(b^{i-1}a^jb^sa^t)=1$) and hence $M^\ell\subseteq \langle a\rangle$. The other inclusion is clear as $a^jb^j=1$. Thus $M^\ell=\langle a\rangle$. Now, if $g\in M^\times$ then $g\in M^\ell$ and so $g=a^j$ for some $j$. Then $1=g^{-1}g=g^{-1}a^j$ so that $j=0$ as $a$ is not invertible. Hence $M^\times=\{1\}$. Let us see that any element in $M$ has at most one right inverse. If $g\in M$ has a right inverse $h$, then $g\in M^\ell$ and so $g=a^j$ as seen above. Write $h=b^sa^t$. From $gh=1$, i.e. $a^jb^sa^t=1$ we deduce that $t=0$ (as $a$ is not invertible). Thus $a^jb^s=1$. If $j>s$, then  $a^{j-s}=a^{j-s}a^sb^s=a^jb^s=1$, a contradiction as $a$ is not invertible. Similarly $j<s$ yields $b^{s-j}=a^jb^jb^{s-j}
=a^jb^s=1$, a contradiction as $b$ is not invertible. Thus $j=s$ and hence $h=b^sa^t=b^j$ is uniquely determined.
\end{invisible}
In this case we also have that $i_B$ is surjective by \cref{cor:localiM}. Notice that neither $p_B$ is surjective (still by \cref{lem:pM}), nor $i_B$ is injective, since $B$ is not a right Hopf algebra ($b$ does not have a right inverse).
  Since $a\in \mathcal{P}(b)$ and $b\in \mathcal{P}(a)$, by \cref{cor:localiM} we have $\ker(i_B)
  =\langle ba-1\rangle$, $\mathrm{H}(B)=Q(B)=B/\ker(i_B)\cong\Bbbk \mathbb{Z}$ and, by \cref{thm:monoidbialgebra}, $\mathrm{C}(B) =\Bbbk M^\times= \Bbbk$.

More generally, for $M=\langle a_1,\ldots,a_n\mid a_1\cdots a_n=1 \rangle$, $B = \Bbbk M$ has $i_{B}$ surjective for any $n\geq 1$.
 Since $a_{i}\in \mathcal{P}(a_{i+1}\cdots a_na_1\cdots a_{i-1})$, by \cref{cor:localiM} we have \[\ker(i_B)
  =\langle a_{i+1}a_{i+2}\cdots a_na_1\cdots a_{i-1}a_i - 1 \mid i=1,\ldots,n-1\rangle\] and $\mathrm{H}(B)=Q(B)=B/\ker(i_B)\cong  \Bbbk G_{n-1}$ where $G_{n-1}$ is the free group on $n-1$ elements.
% \rd{Since $a_{1+i}a_{2+i}\cdots a_na_1\cdots a_{i}\oslash 1
% =a_{1+i}a_{2+i}\cdots a_na_1\cdots a_{i}a_{1+i}a_{2+i}\cdots a_n\oslash a_{1+i}a_{2+i}\cdots a_n
% =a_{1+i}a_{2+i}\cdots a_n\oslash a_{1+i}a_{2+i}\cdots a_n=1\oslash 1$ we have that $a_{1+i}a_{2+i}\cdots a_na_1\cdots a_{i}-1\in K$. Since $K$ is a two-sided ideal we get
% $I\coloneqq\langle a_{1+i}a_{2+i}\cdots a_na_1\cdots a_{i}-1\mid i=0,1,\ldots,n-1\rangle\subseteq K$. The other inclusion is true as $B/I$ is a Hopf algebra and hence we get $\mathrm{H}(B)=Q(B)=B/I\cong\Bbbk \mathbb{Z}^{n-1}.$
% }
\end{example}

In fact, \cref{prop:epsiInj} states that the injectivity of $i_B$ and of $p_B$ are both sufficient conditions for the injectivity of $\epsilon_B \colon \mathrm{C}(B) \to B$. Let us show that they are not necessary.

\begin{example}\label{ex:converses}
Let $M$ be a monoid having an $a\in M$ with at least two different right inverses $b\neq c$, e.g.\ the free monoid $M=\langle a,b,c\mid ab=1=ac \rangle$ generated by $\{a,b,c\}$ and subject to the relations $ab=1=ac$. Then for $B = \Bbbk M$, $p_B$ is neither injective nor surjective by \cref{lem:pM}, but still the canonical map $\epsilon_B:\mathrm{C}(B)\to B$ is injective as we remarked in \cref{rem:epsiMinj}. Thus the converse of \cref{prop:epsiInj}\,\ref{item:epsiInj2} is not true in general. Note also that $i_B(b)=b\oslash 1=b\oslash ab=1\oslash a= c\oslash ac=c\oslash 1=i_B(c)$.
Thus $i_B$ is not injective and so also the converse of \cref{prop:epsiInj}\,\ref{item:epsiInj1} fails to be true.
\begin{invisible}
From AdjMon we know that $i_B$, for $B=\Bbbk M$, is injective exactly when the monoid $M$ is right cancellative. In this case, if $ab=1$, then $bab=b$ and hence $ba=1$ so that $M^\ell=M^\times$. In particular any element in $M$ has at most one right inverse, so that $i_B$ injective implies $p_B$ injective in this context. Instead the converse implication is not true.
For instance, take $M=\langle a,b,c\mid ba=1=ca \rangle$. Then $M$ is not right cancellative and any element in $M$ has, apparently, at most one right inverse (we should have $M^\ell=\langle b,c\rangle$). If this is correct, we get that $i_B$ is not injective while $p_B$ is injective. Is $i_B$ surjective?
\end{invisible}
Nevertheless, $i_B$ is surjective by \cref{cor:localiM}.

  Since $a\in \mathcal{P}(b)\cap \mathcal{P}(c)$ and $b\in \mathcal{P}(a)\ni c$, by \cref{cor:localiM}, we have $\ker(i_B)
  =\langle ba-1,ca-1\rangle$ and $\mathrm{H}(B)=Q(B)=B/\ker(i_B)\cong\Bbbk \mathbb{Z}$. %
% \rd{Since $ba\oslash 1=bab\oslash b=b\oslash b=1\oslash 1$ and similarly $ca\oslash 1=1\oslash 1$, we get that $ba-1,ca-1\in K\coloneqq\ker(i_B)$. Since $K$ is a two-sided ideal, we get $I\coloneqq\langle ba-1,ca-1\rangle\subseteq K$ and hence we get the equality being $\frac{B}{I}$ a Hopf algebra. Summing up $\mathrm{H}(B)=Q(B)=\frac{B}{I}\cong\Bbbk \mathbb{Z}.$
% }
Furthermore, $M^\times=\{1\}$ and $M^\ell=\langle a\rangle$, so that $\mathrm{C}(B) =\Bbbk M^\times= \Bbbk$ by \cref{thm:monoidbialgebra}.
\begin{invisible}
The elements of $M$ are of the form $xa^j$ where $x$ is a word in $b,c$. If $xa^j\in M^\ell$, then there is $ya^t$ such that $xa^jya^t=1$. If $x$ starts with $b$, then $b$ would be invertible, which is not the case. If $x$ starts with $c$, then $c$ would be invertible, which is not the case either. Hence $x = 1$ and $M^\ell\subseteq \langle a\rangle$. The other inclusion is clear as $a^jb^j=1$. Thus $M^\ell=\langle a\rangle$. Now, if $g\in M^\times$ then $g\in M^\ell$ and so $g=a^j$ for some $j$. Then $1=g^{-1}g=g^{-1}a^j$ so that $j=0$ as $a$ is not invertible. Hence $M^\times=\{1\}$.
\end{invisible}
\end{example}

%\subsection{Additional results on the surjectivity of \texorpdfstring{$i_B$}{iB} and monoid bialgebras}

For many bialgebras $B$, such as when $B$ is right perfect or left $n$-Hopf, both the surjectivity of $i_B$ and the injectivity of $p_B$ hold. The two properties are, however, unrelated in general.

\begin{example}%[\ps{added 2025-01-05}]
\label{ex:iBpB}
\cref{ex:converses} exhibits a bialgebra with $i_B$ surjective and $p_B$ not injective.
    % Here is another example on the same line:} let $M = \langle x_0, x_1, x_2 ,\ldots \mid  x_0x_k = 1,\, \forall\,\ps{k \geq 1}\rangle$ be the free monoid generated by $\{x_n \mid n \in \N\}$ and subject to the relations $x_0x_k = 1$ for all \ps{$k \geq 1$}.
    % \cref{lem:pM} entails that for the monoid bialgebra $B = \Bbbk M$, the map $p_B$ is not injective, as $x_0$ has infinitely many right inverses. However, $i_B$ is surjective \ps{by \cref{cor:localiM}}.
    % %: for every $n \in \N$, the sub-bialgebra $B_n$ of $B$ generated by $x_0$ and $x_n$ is isomorphic to the monoid bialgebra from \cref{exa:notweakpBinj} and hence it has $i_{B_n}$ surjective.}
    % % \ps{Then, the following computation}
    % % \[
    % % \ps{1 \oslash x_n^a x_0^b = x_n^b \oslash x_n^ax_0^bx_n^b = x_n^b \oslash x_n^a = x_n^bx_0^ax_n^a \oslash x_n^a = x_n^bx_0^a \oslash 1}
    % % \]
    % % \ps{for all $a,b \geq 0$ shows that $i_{B_n}$ is surjective (see \cref{lem:iBsu}\,\ref{item:epi4}) and so $i_B$ is surjective, too.}
    % %Since for every $m,n \in \N$ we have
    % %\[x_m \oslash x_n = x_mx_0x_n \oslash x_n = x_mx_0 \oslash 1.\]
    \begin{invisible}
    Particular case of Lemma 2.21 in AdjMon.
    \end{invisible}
%
%\rd{[In realtà qui stiamo usando la successiva \cref{prop:localiB} per concludere che $i_B$ è suriettiva in base al comportamento sui generatori.]}
%\rd{[Nota che $B_n$ è il monoide $M$ di \cref{exa:notweakpBinj}. \ps{[Cacchio, è vero, non me n'ero accorto in effetti]} Mi accorgo ora che basta che la bialgebra sia generata da elementi $x_i\in B$ tali che $1\oslash x_i=x'_i\oslash 1$ per qualche $x'_i\in B$, nella Prop precedente. Questo permetterebbe di fare il conto per i soli generatori qui sopra.]\ps{[Vero, me n'ero accorto, ma non volendo aggiungere un altro risultato che fosse una sfumatura di \cref{prop:localiB}, ho preferito adattare l'esempio alla proposizione e lasciare al lettore trarre le sue conclusioni]}} \rd{[Inrealtà basterebbe fare una poccola modifica all'enunciato...faccio una prova]}
On the other hand, consider the monoid bialgebra $B = \Bbbk \N$. It has $p_B$ injective because in $\N$ every element has at most one right inverse. However, it cannot have $i_B$ surjective because it embeds into $\Bbbk \Z$, which is a Hopf algebra, and hence $B$ has $i_B$ already injective by \cref{lem:iB-inj}\,\ref{ib-inj_item3}: if $i_B$ were also surjective, then $\Bbbk \N$ would have been a commutative right Hopf algebra (by \cref{prop:Frobenius}) and so a Hopf algebra (see \cite[Theorem 3]{GNT}), which is clearly not the case.
\end{example}

Being left Artinian, or even right perfect, is a sufficient but not necessary condition for $i_B$ to be surjective.

\begin{example}
\label{exa:lfmon}
%\begin{invisible}\rd{[Added on 29/04/2024]} \end{invisible}
A monoid $M$ is called a \emph{locally finite monoid} if every finite subset generates a finite sub-monoid.
 Clearly the monoid algebra $B=\Bbbk M$ of such a monoid is a locally finite algebra, see e.g. \cite[Remark 2.4]{Ikeda}  in the group case.
\begin{invisible}
Let $x_1,x_2,\cdots x_n\in \Bbbk M$. Write
$x_i=\sum_{j=1}^m k_{ij} g_{ij}$. Then $H:=\langle g_{ij}\mid 1\leq i\leq n,1\leq j\leq m\rangle$ is finite by assumption. Clearly each $x_i\in \Bbbk H$ which is finite-dimensional. Hence the subablgebra generated by $x_1,x_2,\cdots x_n$ is finite-dimensional and $\Bbbk M$ is locally finite.

In fact, in \href{https://math.stackexchange.com/questions/4908076/the-relation-between-locally-finite-algebras-and-group-algebras-of-locally-finit}{stack} it is noticed that $M$ is a locally finite group iff the group algebra $\Bbbk M$ is a locally finite algebra.
\end{invisible}
In view of \cref{prop:localiB}\,\ref{item:localiB4}, we get that $i_B$ is surjective. However, as we already observed, the monoid algebra $B$ is not left Artinian unless $M$ is finite.

Now, let $M$ be a monoid which is the union of an increasing chain of finite sub-monoids $(M_i)_{i\in I}$. Then $M$ is a locally finite monoid, cf.\ \cite[Lemma 1.A.9]{Kegel-Wehrfritz} in the group case.
\begin{invisible}
Let $g_1,g_2,\cdots ,g_n\in M$. If $M=\cup_{i\in I} M_i$ with $(M_i)_{i\in I}$ a chain of finite sub-monoids, then for each $t$ there is $i_t\in I$ such that $g_t\in M_{i_t}$. Let $M_N \coloneqq \max\{M_{i_t}\mid t = 1,\ldots, n\}$. Then $g_1,g_2,\cdots ,g_n\in M_N$ and $M_N = M_{i_1}\cup\cdots\cup M_{i_n}$. Hence $\langle g_1,g_2,\cdots ,g_n\rangle \subseteq M_N$  which is finite.
\end{invisible}
Therefore the argument above applies.

For instance, let $\{x_i\mid i\in \N\}$ be a set of variables. For every $i\in \N$ take two positive integers  $a_i\neq b_i$ and consider the commutative monoid $M=\langle x_i\mid i,j\in \N,x_i\!^{a_i}=x_i\!^{b_i} ,x_ix_j=x_jx_i\rangle$ and the monoid algebra $B=\Bbbk M$.
Then $M$ is the union of the increasing chain of finite sub-monoids $M_n=\langle x_i\mid 0\leq i,j\leq n,x_i\!^{a_i}=x_i\!^{b_i} ,x_ix_j=x_jx_i\rangle$. Therefore, $B$ is locally finite, $i_B$ is surjective, but $B$ is not a right perfect ring as the descending chain of cyclic ideals $(x_0)\supset (x_1x_0)\supset (x_2x_1x_0)\supset \cdots$ does not stabilize.
Set $n_i\coloneqq |a_i-b_i|$. Note that if $u$ is periodic with index $i$ and period $p$, i.e.\ $u^i=u^{i+p}$, then $u^{p-1}uu^{i} = u^{i+p} = u^{i}$ so that $buc=c$ for $b=u ^{p-1}$, $c=u ^{i}$. Thus, since $x_i$ is periodic, by \cref{cor:localiM} we have $\ker(i_B)
  =\langle x_i^{n_i}-1\mid i\in\N\rangle$ and $\mathrm{H}(B)=Q(B)=B/\ker(i_B)\cong\Bbbk G$  where $G\coloneqq \langle x_i\mid i,j\in \N,x_i\!^{n_i}=1 ,x_ix_j=x_jx_i\rangle$.
%  \rd{Let's assume for simplicity $a_i>b_i$ and set $n_i\coloneqq a_i-b_i\in\N$.
% Since $x_i\!^{n_i}\oslash 1=x_i\!^{n_i}x_i\!^{b_i}\oslash x_i\!^{b_i}=x_i\!^{a_i}\oslash x_i\!^{b_i}=x_i\!^{b_i}\oslash x_i\!^{b_i}=1\oslash 1$, we get that $x_i\!^{n_i}-1\in K\coloneqq\ker(i_B)$. Since $K$ is a two-sided ideal, we get $I\coloneqq\langle x_i\!^{n_i}-1\mid i\in\N\rangle\subseteq K$ and hence we get the equality being $\frac{B}{I}$ a Hopf algebra. Summing up $\mathrm{H}(B)=Q(B)=\frac{B}{I}\cong\Bbbk G$ where $G\coloneqq \langle x_i\mid i,j\in \N,x_i\!^{n_i}=1 ,x_ix_j=x_jx_i\rangle$.
% }
\end{example}

Similarly, being locally finite is a sufficient but not necessary condition for having $i_B$ surjective.

\begin{example}%rd{[Added on 19/04/2024]}
\label{exa:notArtLoc}
Let $\{x_i\mid i\in I\}$ be a set of variables  and for every $i\in I$ take  two positive integers  $a_i\neq b_i$.
Consider the monoid $M=\langle x_i\mid i\in I,x_i\!^{a_i}=x_i\!^{b_i} \rangle$ generated by the periodic elements $\{x_i\mid i\in I\}$ and the monoid algebra $B=\Bbbk M$. Since $x_i$ belongs to the sub-bialgebra it generates, which is isomorphic to $ \Bbbk\langle x_i \mid x_i^{a_i}=x_i^{b_i}\rangle$ and hence is a finite-dimensional sub-bialgebra of $B$,  \cref{prop:localiB}\,\ref{item:localiB3} allows us to conclude that $i_B$ is surjective.
Set $n_i\coloneqq |a_i-b_i|$.
As in \cref{exa:lfmon}, we get $\ker(i_B)=\langle x_i\!^{n_i}-1\mid i\in I\rangle$ and hence $\mathrm{H}(B)=Q(B)=B/\ker(i_B)\cong\Bbbk G$ where $G\coloneqq \langle x_i\mid i\in I,x_i\!^{n_i}=1 \rangle$.

Assume that $I$ has at least two elements  $i\neq j$.
Then $B$ is not locally finite as the sub-algebra $\Bbbk\langle x_i,x_j\rangle$ of $B$ generated by $x_i,x_j$ is not finite-dimensional. The monoid algebra $B$ is not even a right perfect
ring.
%If e.g. $I=\N$, then $B$ is not even a left Artinian ring as the descending chain of left ideals $Bx_0\supset Bx_1x_0\supset Bx_2x_1x_0\supset \cdots$ does not stabilize.
\end{example}

\begin{example}%[\rd{added on 2024-11-22}]
\label{exa:periodmon}
Let us dig deeper into a particular instance of \cref{exa:notArtLoc}.
Given $m,n\in\mathbb{N}$, consider the finite monoid $M=\langle x\mid x^{m+n}=x^{n} \rangle$ and the monoid algebra $B=\Bbbk M$.
We know that $\mathrm{H}(B)=Q(B)= B/\ker(i_B)\cong \Bbbk C_m$  where $C_m\coloneqq\langle x\mid x^{m}=1 \rangle$ is the cyclic group of order $m$.
%\rd{Since $x$ is periodic, by \cref{cor:localiM}, we have $\ker(i_B)
%  =\langle x^{m}-1\mid \rangle$ and $\mathrm{H}(B)=Q(B)= \frac{B}{\ker(i_B)}\cong H\coloneqq\Bbbk C_m$  where $C_m\coloneqq\langle x\mid x^{m}=1 \rangle$ is the cyclic group.}

% Note that
% $x^{m}\oslash 1
% =x^{m}x^{n}\oslash x^{n}
% =x^{m+n}\oslash x^{n}
% =x^{n}\oslash x^{n}
% =1\oslash 1$ so that $x^{m}-1\in\ker(i_B).$
% Consider the cyclic group $C_m\coloneqq\langle x\mid x^{m}=1 \rangle$. Set $H\coloneqq\Bbbk C_m$ and let $\pi:B\to H,\;x\mapsto x,$  be the canonical projection. By the foregoing $\ker(\pi)=\langle x^{m}-1\rangle\subseteq\ker(i_B)=\ker(\qq)$ where $\qq:B\to \qB = B/\ker(i_B)B$ is the canonical projection and we note that $\ker(i_B)B=B\ker(i_B)=\ker(i_B)$ as $B$ is commutative.
% Since $\ker(\pi)\subseteq \ker(\qq)$ and $\pi$ is surjective, there is a bialgebra map $\alpha:H\to \qB$ such that $\alpha\circ \pi=\qq. $
% By \cref{prop:HBfd}, we have $\qB=\mathrm{H}(B)$ so that, being $H$ a Hopf algebra and by the universal property of $\mathrm{H}(B)$, there is a bialgebra map $\beta:\qB\to H$ such that $\beta\circ \qq=\pi.$  Since $\qq$ and $\pi$ are surjective, we deduce that $\alpha$ and $\beta$ are mutual inverses. As a consequence $\ker(i_B)=\ker(\qq)=\ker(\pi)=\langle x^{m}-1\rangle$ and $\mathrm{H}(B)= \qB =\frac{B}{\langle x^{m}-1\rangle}\cong B\oslash B\cong H$.

Note that, since $\id^{* m+n}=\id^{*n}$, $S=\id^{* m-1}$ is a two sided $n$-antipode whenever $m>1$. The minimality of $n$ can be checked as we did in \cref{exa:nHopf}.
In the particular case when $m=3$ and $n=2$, we get $M=\langle x\mid x^{5}=x^{2} \rangle$ and $S=\id^{* 2}$. Thus $S(x^k)=x^{2k}$ for all $k\in\mathbb{N}$ and hence $S^2(x)=S(x^2)=x^4$, $S^3(x)=S(x^4)=x^8=x^2=S(x)$ so that $S$ is neither injective nor surjective.

%\rd{[added on 2024-12-09]}
Let us turn to $\mathrm{C}(B)$. By the foregoing, $\mathrm{C}(B)=K(B)$ and hence it is a sub-bialgebra of $B$ whence pointed as well.
\begin{invisible}
By \cite[Proposition 3.4.3(d)]{Radford-book}, the coradical is $\mathrm{C}(B)_0=B_0\cap \mathrm{C}(B)=B\cap \mathrm{C}(B)=\mathrm{C}(B)$. Thus $\mathrm{C}(B)$ is direct sum of its simple sub-coalgebras. Since any simple sub-coalgebra of $\mathrm{C}(B)$ is a simple sub-coalgebra of $B$, then it is one-dimensional.
\end{invisible} Since $\mathrm{C}(B)$ is a Hopf algebra, it is then spanned by invertible grouplike elements. If $n\neq0$, i.e. if $B$ is not natively a Hopf algebra, then there is only one such an element in $B$, namely $1$. In this case we get $\mathrm{C}(B)=\Bbbk$.
Another way to arrive at the same conclusion is to use \cref{lem:KkM}. Indeed $K(B)=\Bbbk M^\ell=\Bbbk$ as $M^\ell=\{1\}$ for $n\neq 0$.
\end{example}

\subsection{Further examples}

\begin{example}
\label{exa:Radford}
% \rd{[added on 2024-12-06]}
Following \cite[Example 3]{Radford-simple}, consider a coalgebra $C\neq 0$. This coalgebra has an associative multiplication defined by $\varepsilon\otimes C:C\otimes C\to C$ with respect to which $\Delta$ and $\varepsilon$ are multiplicative.
The coalgebra structure of $C$ extends uniquely to a bialgebra structure on the algebra $B=\Bbbk\cdot 1+C$ obtained by adjoining a unity to $C$. We claim that $\mathrm{H}(B) = \Bbbk = \mathrm{C}(B)$.

%\rd{[added on 2024-12-17]
% \marginpar{\tiny\ps{With the previous formulation, I could not see why it was not a left Hopf algebra, whence I slightly changed the order of the claims.} \rd{ok!}}
Let us first check that $B$ is not a right Hopf algebra.
If $S$ is a right antipode, then for every $c \in C$ we have $\varepsilon_B(c)1=c_1S(c_2)=\varepsilon_C(c_1)S(c_2)=S(c)$ so that $\varepsilon_B(c)1=S(c)$. Since we also have  $\varepsilon_B(1)1=S(1)$, we conclude that $S=u_B \circ \varepsilon_B$. Then, for every $c \in C$ we have $\varepsilon_B(c)1=c_1S(c_2)=c_1\varepsilon_B (c_2)=c$ and hence $c\in \Bbbk 1$ and so $c=0$ for every $c\in C$, a contradiction as $C\neq 0$.
%If $S$ is a left antipode, then  $\varepsilon_B(c)1=S(c_1)c_2$. Now $S(c)\in \Bbbk 1+C$ so we can write it as $S(c)=k(c)1+\alpha(c)$ with $\alpha(c)\in C$. By applying $\varepsilon_B$ to both sides we get $\varepsilon_BS(c)=k(c)+\varepsilon_B\alpha(c)$ i.e.\ $\varepsilon_B(c)=k(c)+\varepsilon_C\alpha(c)$ \ps{(because every one-sided antipode is automatically unital and counital, see \cite[Remark 3.8]{Sar21})} so that $k(c)=\varepsilon_B(c)-\varepsilon_C\alpha(c)$ and hence $S(c)=\varepsilon_B(c)1-\varepsilon_C\alpha(c)1+\alpha(c)$. Thus $\varepsilon_B(c)1=S(c_1)c_2=\varepsilon_B(c_1)c_2-\varepsilon_C\alpha(c_1)c_2+\alpha(c_1)c_2=c$. As above we are led to contradiction.
%}
% In order to compute $\mathrm{H}(B)$,
% %First, in view of \cref{coro:fdS}, since $B$ is finite-dimensional, then it has a two sided $n$-antipode for some $n\geq 0$. In fact $n\neq 0$ as $B$ is not a Hopf algebra.
% we first show that $B$ has a two sided $1$-antipode. For $c\in C$, we compute
% $\mathrm{Id}^{*2}\left( c\right)  =c_1c_2=\varepsilon(c_1)c_2=c$ and $
% \mathrm{Id}^{*2}\left( 1\right)  =1\cdot1=1$
% so that $\mathrm{Id}^{*2}=\mathrm{Id}$ and hence $S=u\varepsilon $ is a
% two-sided $1$-antipode.

Nevertheless, $B$ is a $1$-Hopf algebra with two-sided $1$-antipode $u_B \circ \varepsilon_B$. Indeed, for $c\in C$ we can compute $\mathrm{Id}^{*2}\left( c\right)  =c_1c_2=\varepsilon(c_1)c_2=c$ and $
\mathrm{Id}^{*2}\left( 1\right)  =1\cdot1=1$ so that $\mathrm{Id}^{*2}=\mathrm{Id}$ and hence $B$ is a right $1$-Hopf algebra and $S=\mathrm{Id}^{*0}=u_B \circ \varepsilon_B $ is a right $1$-antipode (it cannot be a right $0$-antipode, in view of the above). $B$ is also a left $1$-Hopf algebra and $S$ a left $1$-antipode, because on the one hand $B$ cannot be a left Hopf algebra (or \cref{lem:nthS} would entail that it is a Hopf algebra, contradicting the fact that it cannot have a right antipode), and, on the other hand, $\id^{*0} * \mathrm{Id}^{*2} = \mathrm{Id}^{*2} = \mathrm{Id}$.

Now, by the proof of \cref{prop:pconv} we know that, for every $b\in B$, we have $b_1S(b_2) -\varepsilon_B (b) 1_B \in \ker\left(
i_{B}\right)$, i.e., by definition of $S$, $b-\varepsilon_B ( b)1_B\in \ker\left( i_{B}\right)$. Thus $B^+ \coloneqq \ker(\varepsilon_B)$ is a bi-ideal such that $B/B^+ \cong \Bbbk$ is a Hopf algebra and which is contained in $ \ker\left( i_{B}\right)$. Hence, in view of \cref{lem:HopfidealiB}, $\mathrm{H}(B) = Q(B) = B/I = B/B^+ \cong \Bbbk$. Moreover, by \cref{prop:semiantip} the map $i_B$ is surjective, and so we also have $B\oslash B = \im(i_B) \cong B/\ker\left( i_{B}\right)\cong\Bbbk$.
%
% \rd{[Qui si potrebbe accorciare usando \cref{lem:HopfidealiB}]}
%
% On the other hand, since $\ker(i_B)$ is a coideal, we have $\varepsilon_B(\ker(i_B))=0$ i.e.\ $\ker\left( i_{B}\right)\subseteq B^+$.
% Thus $B^+= \ker\left( i_{B}\right)$.
% As a consequence $Q\left( B\right) =B/\ker\left( i_{B}\right) B=B/B^+\cong\Bbbk $. Since the latter is a Hopf algebra and $\qB$ has the desired universal property in view of \cref{prop:pconv}\,\ref{item:pconv2}, we have $\mathrm{H}(B)=Q(B)$.
% Note that, by \cref{prop:semiantip}, the map $i_B$ is surjective, so that we also have $B\oslash B=\im(i_B)\cong B/\ker\left( i_{B}\right) =B/B^+\cong\Bbbk$.

%\rd{[added on 2024-12-09]}
By \cref{prop:semiantip} again, $p_B$ is injective and we obtained this by taking $T=S$ in \cref{prop:pBin}. Since, by the proof of \cref{prop:pconvd}, we have
$b'_1T(b'_2)=\varepsilon(b')1$ for every $b'\in K(B)$, by the definition of $S$ we get $b'=\varepsilon(b')1$ and hence $K(B)\subseteq \Bbbk\cdot 1$. Thus, $\mathrm{C}(B)=K(B)= \Bbbk\cdot 1$.
\end{example}

The following example naturally arises as a dual version of the previous one. It is of particular interest, because it also shows the existence of infinite-dimensional left Artinian bialgebras and of bialgebras which are right perfect but not left Artinian.

\begin{example}
\label{exa:Radual}
% \rd{[added on 2024-12-16]}
Let $A$ be an algebra and consider the product algebra $B\coloneqq A\times \Bbbk$.
%\cite[page 432]{BourbakiI}
Define on it a bialgebra structure by setting ,  for every $a\in A,k\in \Bbbk$,
\[\Delta(a,k)=(1,1)\otimes(a,0)+(a,k)\otimes(0,1)\qquad \text{and} \qquad\varepsilon(a,k)=k.\]
Note that if $A$ is left Artinian and infinite-dimensional, then so is $B$ as the left ideals in $B$ are of the form $I\times 0$ and $I\times \Bbbk$ where $I$ is a left ideal in $A$. For the same reason, if $A$ is right perfect but not left Artinian, e.g.~the triangular $\Q$-algebra $\left(\begin{smallmatrix}
    \mathbb{Q} & \mathbb{R} \\ 0 & \mathbb{Q}
\end{smallmatrix}\right)$ as in  \cite[page 318]{AndFull}, then so is $B$.
We now compute $\mathrm{C}(B)$ and $\mathrm{H}(B)$. First note that $B$ is neither a left nor a right Hopf algebra, as $(0,1)\in B$ is a grouplike which is  neither left nor right invertible.
\begin{invisible}
  Otherwise we would have
$(1,1)=1_{B}=\varepsilon(0,1)1_{B}=S((0,1)_1)(0,1)_2=S(1,1)(0,0)+S(0,1)(0,1)=S(0,1)(0,1)\in 0\times \Bbbk$ or
$(1,1)=1_{B}=\varepsilon(0,1)1_{B}=(0,1)_1S((0,1)_2)=(1,1)S(0,0)+(0,1)S(0,1)=(0,1)S(0,1)\in 0\times \Bbbk$, a contradiction.
\end{invisible}
Moreover $\id^{*2}(a,k)=(1,1)(a,0)+(a,k)(0,1)=(a,0)+(0,k)=(a,k)$ so that $\id^{*2}=\id$, and hence $S \coloneqq u_B\circ \varepsilon_B$ is a two-sided $1$-antipode.

As in \cref{exa:Radford}, we conclude that $B^+= \ker\left( i_{B}\right)$, that $\mathrm{H}(B)=Q(B)\cong B\oslash B\cong \Bbbk$ and that $\mathrm{C}(B)=K(B)= \Bbbk\cdot 1_B$.
\begin{invisible}
 By the proof of \cref{prop:pconv}, we know that, for every $b\in B$, we have $b_1S(b_2) -\varepsilon_B (b)1_B \in \ker\left(
i_{B}\right) $ i.e. $b-\varepsilon_B ( b)1_B\in \ker\left( i_{B}\right) .$ Thus $B^+\subseteq \ker\left( i_{B}\right)$.
On the other hand, since $\ker(i_B)$ is a coideal, we have $\varepsilon-B(\ker(i_B))=0$ i.e. $\ker\left( i_{B}\right)\subseteq B^+$.
Thus $B^+= \ker\left( i_{B}\right)$.
As
a consequence $Q\left( B\right) =B/\ker\left( i_{B}\right) B=B/B^+\cong\Bbbk $. Since this is a Hopf algebra, the universal property of $Q(B)$ entails $\mathrm{H}(B)=Q(B)$.
Note that, by \cref{prop:semiantip}, the map $i_B$ is surjective so that we also have $B\oslash B=\im(i_B)\cong B/\ker\left( i_{B}\right) =B/B^+\cong\Bbbk$.

By \cref{prop:semiantip}, the map $p_B$ is injective and this has been obtained by taking $T=S$ in \cref{prop:pBin}. Since, by the proof of \cref{prop:pconvd}, we have
$b'_1T(b'_2)=\varepsilon_B(b')1_B$, for every $b'\in K(B)$, definition of $S$, we get $b'=\varepsilon_B(b')1_B$ and hence $K(B)\subseteq \Bbbk\cdot 1_B$ and so $\mathrm{C}(B)=K(B)= \Bbbk\cdot 1_B$.
\end{invisible}
In particular, in case $A=\Bbbk C_n$ for $C_n=\langle c\mid c^n=1\rangle$, $n>1$, we get the bialgebra $B=A\times  \Bbbk=\Bbbk\langle x\mid x^{n+1}=x\rangle$ where $x\coloneqq (c,0)$, $\Delta(x)=1\otimes x+x\otimes (1-x^n)$ and $\varepsilon(x) = 0$.
One easily shows that the set of grouplike elements in $B$ is $G(B)=\{1,1-x^n\}$ so that
$x\notin \Bbbk G(B)$.
\begin{invisible}
Set $g \coloneqq 1-x^n$ and note that $g^2=g$. Since $(1\otimes x)(x\otimes (1-x^n))=0$, for $t>0$ we get that $\Delta(x^t)=1\otimes x^t+x^t\otimes (1-x^n)^t
=1\otimes x^t+x^t\otimes (1-x^n) = 1 \otimes x^t + x^t \otimes g$. Thus, if $z=\sum_{t=0}^nk_tx^t$ is grouplike, we get $k_0 = \varepsilon(z) = 1$ and
\[z \otimes z = \Delta(z)=k_01\otimes 1+ \sum_{t=1}^nk_t1\otimes x^t+\sum_{t=1}^nk_tx^t\otimes (1-x^n) = 1 \otimes z + (z-1) \otimes g.\]
%Since this equals $z\otimes z=\sum_{i,j=0}^nk_ik_jx^i\otimes x^j$, by matching the terms with the same degree, we get $k_ik_j=0$ for $i\neq 0,1<j<n$. If $k_i=0$ for every $i\neq 0$ we get $z=1$ and hence $z\in \{1,g\}$. Otherwise we get $k_j=0$ for every $1<j<n$  and hence $z=1+k_nx^n=(1+k_n)1-k_ng\in \Bbbk 1+\Bbbk g$. \ps{[?]}
I.e. $z$ has to satisfy
\[(z-1) \otimes (z-g) = 0\]
which forces $z=1$ or $z=g$.
\end{invisible}
%which is an instance of \cref{exa:periodmon}.
%Posto $x=(c,0)$ e $g=(0,1)$, si ha $x^n+g=(c^n,1)=(1,1)=1$ e quindi $g=1-x^n$. Inoltre $x^{n+1}=(c^n+1,0)=(c,0)=x$.
\end{example}

In the examples so far where we explicitly computed $\mathrm{C}(B)$, we always found $\mathrm{C}(B) = \Bbbk$. In general, this is not the case (see, e.g., \cref{thm:monoidbialgebra}). The forthcoming \cref{ex:CBnontriv} exhibits a less elementary example.

\begin{remark}
\label{rmk:findual}
% \rd{[added on 2024-12-20]}
 It is known, see e.g. \cite[page 87]{Abe}, that the finite dual defines two auto-adjunctions namely $(-)^\circ\dashv (-)^\circ :\Bialg^{\mathrm{op}}\to\Bialg$  and $(-)^\circ\dashv (-)^\circ :\Hopf^{\mathrm{op}}\to\Hopf$. Moreover, the forgetful functor $F:\Hopf\to\Bialg$ forms an adjoint triple $\mathrm{H}\dashv F\dashv\mathrm{C}$ together with the free and the cofree functors. Thus, we also have the adjunctions $H\dashv F$ and $\mathrm{C}^{\mathrm{op}}\dashv F^{\mathrm{op}}:\Hopf^{\mathrm{op}}\to\Bialg ^{\mathrm{op}}$. Since $(-)^\circ\circ F^{\mathrm{op}} =F\circ (-)^\circ$, by uniqueness of left adjoint functor, we get $\mathrm{C}^{\mathrm{op}}\circ (-)^\circ\cong (-)^\circ\circ \mathrm{H}$ and hence one has $\mathrm{C}(B^\circ)\cong \mathrm{H}(B)^\circ$ for every bialgebra $B.$
 \[
 \begin{gathered}
 \xymatrix@R=0.6cm{
 \Bialg \ar[d]_{\mathrm{H}} \ar[r]^{(-)^\circ} & \Bialg^{\mathrm{op}} \ar[d]^{\mathrm{C}^{\mathrm{op}}} \\
 \Hopf  \ar[r]^{(-)^\circ} & \Hopf^{\mathrm{op}}}
 \end{gathered} \qedhere
 \]
\end{remark}

\begin{example}\label{ex:CBnontriv}
%\rd{[added on 2024-12-09]}
\begin{enumerate}[label=\arabic*),leftmargin=*]
\item Consider the dual bialgebra $B^*$ of the finite-dimensional bialgebra $B$ from \cref{ssec:quotquant}, i.e.\ $B=\Bbbk \left\langle x,y\mid yx=-xy,x^{3}=x,y^{2}=0\right\rangle$. Since the finite dual coincide with the linear dual in the finite-dimensional case, by \cref{rmk:findual} we have
\[\mathrm{C}(B^*)= \mathrm{C}(B^\circ)\cong \mathrm{H}(B)^\circ= H_4^\circ=H_4^*\cong H_4,\]
where we used \cref{prop:quotquant} to write $\mathrm{H}(B)=H_4$, the Sweedler's $4$-dimensional Hopf algebra.
% $B$ is a finite-dimensional bialgebra, then so is $A$. The surjection $\pi:B\to H_4$ into the Sweedler's $4$-dimensional Hopf algebra, yields a bialgebra embedding $\pi^*:H_4^*\to A$. By \cref{prop:pconvd}, we have $H_4\cong H_4^*\cong \im(\pi^*)\subseteq K(A)=\mathrm{C}(A)$.
As a consequence $\mathrm{C}(B^*)$ is not trivial.

\item Let $M$ be a monoid and let $\mathcal{R}_{\Bbbk}(M)\coloneqq (\Bbbk M)^\circ$ be the set of representative functions on $M$ (see, e.g., \cite[Chapter 2, \S 2.2]{Abe}). Then
$\mathrm{C}(\mathcal{R}_{\Bbbk}(M))=\mathrm{C}((\Bbbk M)^\circ)\cong \mathrm{H}(\Bbbk M)^\circ=(\Bbbk G)^\circ=\mathcal{R}_{\Bbbk}(G)$ where $G\coloneqq G(M)$ is the universal enveloping group of $M$. \qedhere
\end{enumerate}
\end{example}

For the convenience of the interested reader, let us conclude by pointing out that the constructions of the Hopf envelope and of the cofree Hopf algebra are also suited for the (quasi)triangular and the dual (quasi)triangular setting, respectively.

\begin{example}
% \rd{[Added on 2025-03-13. Da spostare nel luogo più adatto.]} \ps{[Bello!]} :)
 Let $(B,\mathcal{R})$ be a quasitriangular bialgebra, see \cite[Definition 12.2.3]{Radford-book}. If $\eta_B:B\to \mathrm{H}(B)$ is surjective (e.g. $B$ is finite-dimensional or, more generally, a right perfect bialgebra), then, cf.~\cite[Exercise 12.2.2]{Radford-book}, $(\mathrm{H}(B),\mathcal{S})$ is a quasitriangular Hopf algebra where $\mathcal{S}\coloneqq (\eta_B\otimes \eta_B)(\mathcal{R})$. If $f \colon (B,\mathcal{R})\to (H,\mathcal{T})$ is a morphism of quasitriangular bialgebras (see \cite[Definition 12.2.7]{Radford-book}) and $H$ is a Hopf algebra, in particular $f \colon B\to H$ is a bialgebra map into a Hopf algebra so that there is a unique Hopf algebra map $\hat{f} \colon \mathrm{H}(B)\to H$ such that $\hat{f}\circ \eta_B=f$. We have $(\hat{f}\otimes \hat{f})(\mathcal{S})=(\hat{f}\otimes \hat{f})(\eta_B\otimes \eta_B)(\mathcal{R})=(f\otimes f)(\mathcal{R})=\mathcal{T}$ so that $\hat{f}$ is a morphism of quasitriangular Hopf algebras. Thus $(\mathrm{H}(B),\mathcal{S})$ is the free quasitriangular Hopf algebra generated by $(B,\mathcal{R})$.

If $(B,\mathcal{R})$ is triangular, $(\mathrm{H}(B),\mathcal{S})$ becomes the free triangular Hopf algebra generated by it.

Dually, if $(B,\beta)$ is a coquasitriangular (or dual quasitriangular) bialgebra, see e.g.~\cite[Definition 14.2.1]{Radford-book}, and if $\epsilon_B \colon \mathrm{C}(B) \to B$ is injective, then $\big(\mathrm{C}(B), \beta\circ(\epsilon_B \otimes \epsilon_B)\big)$ is a coquasitriangular Hopf algebra, 
% This is an easy verification, and it is claimed, for instance, in Schauenburg's \emph{On the braiding on a Hopf algebra in a Braided Category}, New York Journal of Mathematics, page 262.
and it is the cofree coquasitriangular Hopf algebra on $(B,\beta)$.
% If $f \colon (H,\alpha) \to (B,\beta)$ is a morphism of coquasitriangular bialgebras and $H$ is a Hopf algebra, then there exists a unique morphism of Hopf algebras $\hat f\colon H \to \mathrm{C}(B)$ such that $\hat f \circ \epsilon_B = f$ and moreover $\beta \circ (\epsilon_B \otimes \epsilon_B) \circ (\hat f \otimes \hat f) = \beta \circ (f \otimes f) = \alpha$. Hence, $\hat f$ is a morphism of coquasitriangular Hopf algebras.
Similarly, if $(B,\beta)$ is cotriangular, then $\big(\mathrm{C}(B), \beta\circ(\epsilon_B \otimes \epsilon_B)\big)$ is the cofree cotriangular Hopf algebra on $(B,\beta)$.
 \end{example}

\bibliography{references}
\bibliographystyle{acm}

\end{document}